%% file: Habilitation_main_arxiv.tex
\documentclass[a4paper,12pt,twoside]{book}
\pdfoutput=1
\usepackage[utf8]{inputenc}

\usepackage[english, ngerman]{babel}

\usepackage{tocloft}
\usepackage[nohints]{minitoc}
\DeclareMathAlphabet{\mathcal}{OMS}{cmsy}{m}{n}

\usepackage{fancyhdr}
\usepackage{pdfpages}

\input{commands.tex}


\theoremstyle{plain}
\newtheorem{thm}{Theorem}[chapter]
\newtheorem{lem}[thm]{Lemma}
\newtheorem{ass}[thm]{Assumption}

\newtheorem{rem}[thm]{Remark}

\makeatletter
\newcommand{\faculty}[1]{\gdef\@faculty{#1}}
\newcommand{\acadtitle}[1]{\gdef\@acadtitle{#1}}
\newcommand{\birthday}[1]{\gdef\@birthday{#1}}
\newcommand{\birthplace}[1]{\gdef\@birthplace{#1}}

\newcommand{\disstitle}{%
  \thispagestyle{empty}%
  \begin{center}\leavevmode
    \normalfont
    {\Large \@title\par}
    \null
    \vfill
    \begin{LARGE}
     \textbf{H A B I L I T A T I O N S S C H R I F T}
    \end{LARGE}
    \\[1.2cm]
    \vfill
    vorgelegt\\[0.5cm]
    der Fakult\"at \@faculty\\
    der Technischen Universit\"at Dresden\\
    \vfill
    von\\[0.7cm]
    \begin{Large}
     \@acadtitle~\@author
    \end{Large}
    \\[0.7cm]
    geboren am \@birthday~in \@birthplace
  \end{center}
  \vfill
  \begin{tabbing}
  Tag der Disputation\,\=: \kill
  Eingereicht am       \>: 31.05.2013\\
  Tag der Disputation  \>: 20.01.2014\\[0.7cm]
  \end{tabbing}
  Die Habilitationsschrift wurde in der Zeit von Juni 2012 bis Mai 2013
  im Institut für Numerische Mathematik angefertigt.
  \clearpage\pagestyle{empty}\mbox{}\clearpage
  \setcounter{page}{1}
  }
\makeatother

\title{Uniform Error Estimation for Convection-Diffusion Problems}
\author{Sebastian~Franz}
\faculty{Mathematik und Naturwissenschaften}
\acadtitle{Dr.}
\birthday{29.12.1978}
\birthplace{Dresden}
\date{\today}

\newcommand{\appauthor}{}
\newcommand{\apptitle}{}
\newcommand{\apppub}{}

\renewcommand{\maketitle}{\disstitle}
\fancypagestyle{plain}{
\fancyhf{}
\fancyfoot[C]{\textbf{\thepage}}

}
\begin{document}
%
%
   \selectlanguage{english}
   \maketitle
   \pagestyle{fancy}
   \fancyhf{}
   \fancyhead[C]{\nouppercase\leftmark}
   \fancyfoot[C]{\textbf{\thepage}}
%
%
   \pagenumbering{roman}
   \dominitoc[n]
   \tableofcontents
   \clearpage

%
%
   \addstarredchapter{Acknowledgement}
   \include{acknowledgement}
   ~\pagestyle{empty}\newpage
%
%
   \pagenumbering{arabic}
   \pagestyle{fancy}
   \fancyhf{}
   \fancyhead[C]{}
   \fancyfoot[C]{\textbf{\thepage}}
   \addstarredchapter{Abstract}
   \include{abstract}
%
%
%
   \include{introduction}
%
%
   \include{mesh-method}
%
%
   \include{results}
%
%
   \include{norm}
   \include{green}

%
%
   \include{outlook}
%
%
   \addstarredchapter{Bibliography}
   \fancyhead[C]{\nouppercase\leftmark\\}
   \include{bibliography}
%
%
   \input{appendix_SLUB.tex}
%
%
   \include{affirmation}

\end{document}

%% file: commands.tex
\usepackage{amsmath}
\usepackage{amssymb}
\usepackage{amsthm}

\usepackage{cite}

\usepackage{mathptmx}      

\usepackage{color}
\newcommand{\ord}[1]{\mathcal{O}\left(#1\right)}

\newcommand{\e}{\mathrm{e}}

\newcommand{\laplace}{\Delta}
\newcommand{\grad}{\nabla}
\newcommand{\pt}{\partial}

\newcommand{\eps}{\varepsilon}
\newcommand{\norm}[2]{\|{#1}\|_{#2}}

\newcommand{\bignorm}[2]{\left\| {#1} \right\|  _{#2}}
\newcommand{\tnorm}[1]{\left|\!\!\;\left|\!\!\;\left| {#1}
                       \right|\!\!\;\right|\!\!\;\right|}
\newcommand{\enorm}[1]{\tnorm{#1}_\eps}
\newcommand{\lnorm}[1]{\tnorm{#1}_{LPS}}

\newcommand{\half}{\frac{1}{2}}

\newcommand{\N}{\mathbb{N}}

\newcommand{\R}{\mathbb{R}}

\newcommand{\B}{\mathcal{B}}
\newcommand{\I}{\mathcal{I}}
\newcommand{\J}{\mathcal{J}}
\newcommand{\PS}{\mathcal{P}}
\newcommand{\QS}{\mathcal{Q}}

\newcommand{\ve}[1]{\underline{\mathbf{#1}}}

\newcounter{tmp}
\newcommand{\makeballnumber}[1]{\setcounter{tmp}{\theenumi}%
\setcounter{enumi}{#1}%
\leavevmode \csname beamer@@tmpl@enumerate item\endcsname%
\setcounter{enumi}{\thetmp}}
\newcommand{\makeball}{\leavevmode \csname beamer@@tmpl@itemize item\endcsname}

\definecolor{seb}{rgb}{0.5,0,0}

%% file: acknowledgement.tex
\chapter*{Acknowledgement}

I would like to thank all my colleagues whom I had the 
pleasure to work with over the recent years. This includes especially
the group of Prof. Hans-Görg Roos and Prof. Torsten Linß 
(now in Hagen) in Dresden, the Irish guys Dr. Natalia Kopteva and 
Prof. Martin Stynes, and Prof. Gunar Matthies in Kassel.\\

Life is not only mathematics --- although a good part of it is. 
I'm very grateful that Anja chose to follow me to Ireland and back.
Thanks for staying at my side, keeping me down-to-earth and 
becoming my wife!

%% file: abstract.tex
\chapter*{Abstract}
  Let us consider the singularly perturbed model problem
   \begin{align*}
     Lu:=-\eps\laplace u-bu_x+c u & =f\\
   \intertext{with homogeneous Dirichlet boundary conditions on $\Gamma=\partial\Omega$}
     u|_\Gamma & =0
   \end{align*}
  on the unit-square $\Omega=(0,1)^2$.
  Assuming that $b>0$ is of order one,
  the small
  perturbation parameter $0<\eps\ll 1$ causes boundary layers in the solution.\\

  In order to solve above problem numerically, it is beneficial to resolve these
  layers. On properly layer-adapted meshes we can apply finite element methods and
  observe convergence.\\

  We will consider standard Galerkin and stabilised FEM applied to above problem.
  Therein the polynomial order $p$ will be usually greater then two, i.e. we will
  consider higher-order methods.\\

  Most of the analysis presented here is done in the standard energy norm. 
  Nevertheless, the question arises: Is this the right norm for this kind of problem, 
  especially if characteristic layers occur? 
  We will address this question by looking into a balanced norm.\\

  Finally, a-posteriori error analysis is an important tool to construct adapted meshes
  iteratively by solving discrete problems, estimating the error and adjusting the mesh
  accordingly. We will present estimates on the Green's function associated with $L$, that
  can be used to derive pointwise error estimators.
  

%% file: introduction.tex
\chapter{Introduction}
   \fancyhead[C]{\nouppercase\leftmark\\}
  Simple model problems are often helpful in understanding the behaviour 
  of numerical methods in presence of layers for more complicated problems.
  We will consider the singularly perturbed convection-diffusion problem with
  an exponential layer at the outflow boundary and two characteristic layers,
  given by
  \begin{subequations}\label{eq:Lu}
   \begin{alignat}{2}
    -\eps\laplace u-b u_x+c u & =f &\qquad& \text{in} \ \Omega=(0,1)^2,\\
     u & =0 &\qquad& \text{on}\ \partial\Omega.
   \end{alignat}
  \end{subequations}
  We assume $f\in C(\bar\Omega)$, $b\in W^1_\infty(\bar\Omega)$ and $c\in L_\infty(\bar\Omega)$.
  Furthermore, let $b(x,y)\ge\beta$ for $(x,y)\in\bar{\Omega}$ with some positive
  constant $\beta$ of order one, while $0<\eps\ll 1$ is a small perturbation parameter.
  For further assumptions on $f$ see Remark~\ref{rem:plaus}.

  This combination gives rise to an exponential layer of width $\ord{\eps|\ln\eps|}$
  at $x=0$ and to two characteristic layers of width $\ord{\sqrt{\eps}|\ln\eps|}$
  at $y=0$ and $y=1$. Figure~\ref{fig:solution}
  \begin{figure}[bp]
    \centerline{
      \includegraphics[width=0.6\textwidth]{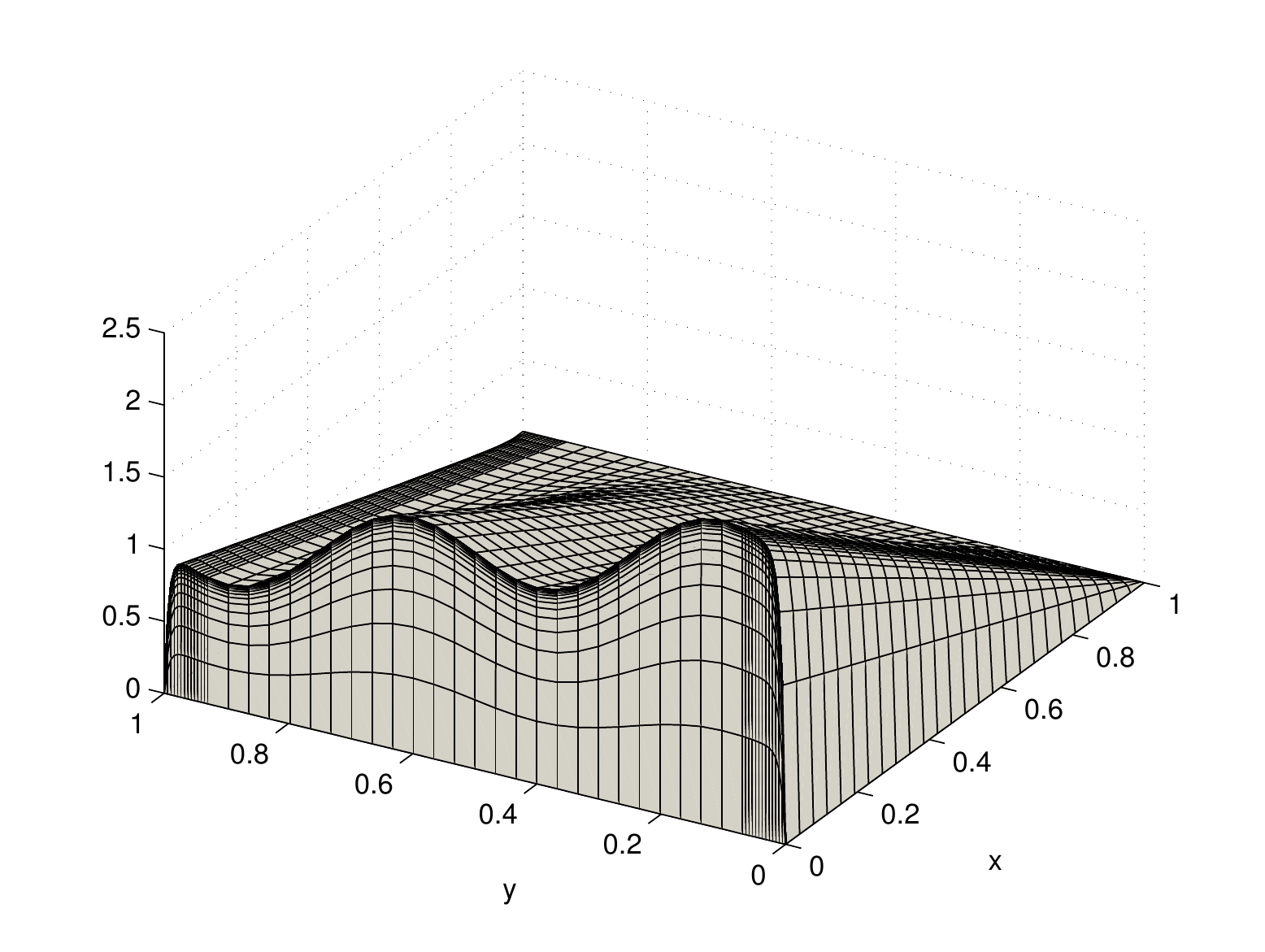}
    }
    \caption{Typical solution of~\eqref{eq:Lu} with two parabolic layers
             and an exponential layer.
             \label{fig:solution}}
  \end{figure}
  shows a typical example of a solution $u$ to \eqref{eq:Lu}.

  Under the condition
  \begin{gather}\label{eq:coer_ass}
     c+\textstyle\frac{1}{2} b_x\geq\gamma>0
  \end{gather}
  problem~\eqref{eq:Lu} possesses a unique solution in
  $H^1_0(\Omega)\cap H^2(\Omega)$. Note that \eqref{eq:coer_ass}
  can always be satisfied by a transformation
  $\tilde{u}(x,y) = u(x,y) e^{\varrho x}$ with a suitably chosen constant $\varrho$.
  In our case $\varrho$ with $\varrho(b-\eps\varrho)\geq c+\frac{1}{2}b_x+\gamma$ suffices.

  When quasi uniform meshes are used, numerical methods do
  not give accurate approximations of~\eqref{eq:Lu} unless the mesh size
  is of the order of the perturbation parameter~$\eps$.
  On the one hand this constitutes a prohibitive restriction for a practical
  treatment of singularly perturbed problems.
  But on the other hand, the mesh sizes do only have to be small in the layer region.
  Therefore, layer-adapted meshes are often used to obtain efficient discretisations.
  
  Based on a priori knowledge of the layer behaviour, we apply
  a-priori adapted meshes.
  Early ideas on layer-adapted meshes can be found in \cite{Bak69,MO'RS96,SO'R97,vVel78}.
  We will use generalisations of \emph{Shishkin meshes}, so called \emph{S-type meshes}
  \cite{RoosLins99, Linss00_1, Linss10}, that resolve the layers and yield robust (or uniform) 
  convergence.
  
  In Figure~\ref{fig:solution} the layer-resolving effect of Shishkin's idea can be 
  seen clearly. We have condensed meshes near the characteristic boundaries 
  ($y=1$ and $y=0$, resp.) and the outflow boundary ($x=0$).

  Even on layer-adapted meshes the standard Galerkin method shows instabilities, see~\cite{LS01a,SCX10}.
  Therefore, stabilised discretisations have to be considered. 
  The recent book by Roos, Stynes and Tobiska~\cite{RST08} gives an overview of 
  many stabilisation ideas.
  
  We will apply and analyse two stabilisation techniques. The first one will be
  the streamline-diffusion finite element method (SDFEM), introduced by Hughes 
  and Brooks~\cite{HB79}. For problems with characteristic layers,
  the SDFEM with bilinear elements was analysed in~\cite{FrLR08}.
  Here we will look into higher-order finite element methods.
  A disadvantage of the SDFEM accounts in particular for discretisations
  with higher-order elements.
  Several additional terms like second order derivatives have
  to be assembled in order to ensure Galerkin orthogonality of the resulting method.

  The second stabilisation technique does not fulfil the Galerkin orthogonality.
  It is the Local Projection Stabilisation method, proposed originally for the Stokes
  problem in~\cite{BB01}. Although, the Galerkin orthogonality is not valid, the remainder
  can be bounded such that the optimal order of convergence is maintained.
  Again, we will look into higher-order methods.
  
  The main focus of our analysis will be the uniform convergence and supercloseness
  of the numerical methods with respect to $\eps$.
  Most of it is done in the so-called energy norm
  \begin{gather}\label{eq:norm:energy}
    \enorm{v} := \left(\eps \norm{\grad v}{0}^2 + \gamma \norm{v}{0}^2\right)^{1/2}.
  \end{gather}
  We denote by $\norm{\cdot}{L_p(D)}$ the standard $L_p$-norm over $D\subset\R^2$. 
  Whenever $p=2$ we write $\norm{\cdot}{0,D}$ and if $D=\Omega$ we skip the reference 
  to the domain. 
  
  Not all norms are equally adequate in measuring errors for problems with layers.
  Although the energy-norm is the associated norm to the weak formulation of
  \eqref{eq:Lu}, not all features of the solution are ``seen''. Especially for
  small $\eps$ the characteristic layer term is less represented then the
  exponential one. Therefore,s we will also consider a balanced norm, where
  both types of layer are equally well represented.

  Another norm that is suitable in recognising the layer behaviour is the $L_\infty$-norm.
  We will not present a-priori results in the maximum-norm but an approach to uniform pointwise
  a-posteriori error estimation using the Green's function.
  
  This habilitation treatise is structured as follows. 
  In Chapter~\ref{cha:meshandmethod} the basics are given, i.e.
  a solution decomposition of $u$ is assumed,
  meshes, polynomial spaces and interpolation operators defined,
  and finally the numerical methods are given.
  In Chapter~\ref{cha:results} we present several analytical and numerical results on the
  convergence and supercloseness of the numerical methods in the energy and related norms.
  In Chapter~\ref{cha:norm} we consider the question, whether a different norm then the
  energy norm could and should be used in the analysis.
  Finally, in Chapter~\ref{cha:green} we present $L_1$-norm estimates of the Green's 
  function associated with problems like \eqref{eq:Lu}. Moreover, they are applied 
  in a first a-posteriori error-estimator for a simple finite difference method.
  
  Most of the results of the Chapters~\ref{cha:meshandmethod}-\ref{cha:green} are from
  already published work. Eight of the papers, whose content is contained in these chapters, 
  are given in the appendix.

  \textbf{Notation.} Throughout this treatise, $C$ denotes a generic constant that is
  independent of both the perturbation parameter $\eps$ and the mesh parameter $N$.
  The dependence of any constant on the polynomial order $p$ will not be
  elaborated.

%% file: mesh-method.tex
\chapter{Meshes and Numerical Methods}\label{cha:meshandmethod}
   \fancyhead[C]{\nouppercase\leftmark\\
                 Section \nouppercase\rightmark}

This chapter contains results from \cite{FrM10, FrM10_1, Fr12} that are also 
given in Appendix \ref{app:LPS}, \ref{app:nonstandard} and \ref{app:GLinter}.

\section{Solution Decomposition}

  Our uniform numerical analysis is based on a decomposition of the solution $u$
  of~\eqref{eq:Lu}. To be more precise: We suppose the existence of a decomposition of $u$
  into a regular solution component and various layer parts.
%
  \begin{ass}\label{ass:dec}
   The solution $u$ of problem~\eqref{eq:Lu} can be decomposed as
   \begin{gather*}
      u =v+w_1+w_2+w_{12},
   \end{gather*}
   where we have for all $x,y\in[0,1]$ and $0\le i+j\le p+1$ the
   pointwise estimates
   \begin{equation}\label{eq:dec:C0}
    \left.
      \begin{aligned}
        \left|\frac{\partial^{i+j} v}{\partial x^i \partial y^j}(x,y)\right|
             &\le C,
        \quad
        \left|\frac{\partial^{i+j} w_1}{\partial x^i \partial y^j}(x,y)\right|
              \le C\eps_{}^{-i}e_{}^{-\beta x/\eps},\\[0.2cm]
        \left|\frac{\partial^{i+j} w_2}{\partial x^i \partial y^j}(x,y)\right|
             &\le C\eps_{}^{-j/2}
                \left(e_{}^{-y/\sqrt\eps}+e_{}^{-(1-y)/\sqrt\eps} \right), \\[0.2cm]
        \left|\frac{\partial^{i+j} w_{12}}{\partial x^i \partial y^j}(x,y)\right|
             &\le C\eps^{-(i+j/2)} e^{-\beta x/\eps}
                    \left(e^{-y/\sqrt\eps}+e^{-(1-y)/\sqrt\eps}\right).
      \end{aligned}
    \right\}
   \end{equation}
   Here $w_1$ is the exponential boundary layer, $w_2$ covers the characteristic
   boundary layers, $w_{12}$ the corner layers, and $v$ is the regular part of 
   the solution.
  \end{ass}
%
  \begin{rem}\label{rem:plaus}
    The validity of Assumption~\ref{ass:dec} is proved in \cite{KSt05,KSt07} 
    for constant functions $b,\, c$ under certain compatibility 
    and smoothness conditions on the right-hand side $f$. 

    In \cite{FrK10_1} the Green's function associated with problem \eqref{eq:Lu}
    was analysed. It was shown, that the Green's function $G$ in the variable-coefficient
    case and the Green's function $\bar G$ in the constant coefficient case show
    a similar behaviour and the same estimates. As a Green's function
    can be used to represent the solution $u$ of its associated problem by
    \[
     u(x,y)=\iint_\Omega G(x,y;\xi,\eta)f(\xi,\eta)d\xi d\eta,
    \]
    it is reasonable to assume the validity of Assumption~\ref{ass:dec}
    in the variable-coefficient case too.
  \end{rem}

\section{Layer-Adapted Meshes}
  
  A discretisation of a singularly perturbed problem on an equidistant mesh results
  in oscillatory solutions unless the mesh-size is of order $\eps$. A loophole
  lies in layer-adapted meshes that are fine in layer regions and coarse in regions,
  where the solution and its derivatives are uniformly bounded.
  
  Back in 1969 Bakhvalov~\cite{Bak69} proposed one of the first layer-adapted
  meshes. Analysis on these kinds of graded meshes is somewhat difficult. The piecewise
  uniform Shishkin meshes~\cite{MO'RS96} proposed in 1996 are easier to handle.
  The first analysis of finite element methods on Shishkin meshes was
  published in~\cite{SO'R97}.
  For a detailed discussion of properties of Shishkin meshes and their uses
  see \cite{RST08} and also \cite{Linss10} for a survey on layer-adapted meshes.
  
  Here we use a tensor-product mesh that
  is constructed by taking in both $x$- and $y$-direction so called 
  \emph{S-type meshes} \cite{RoosLins99} with $N$ mesh intervals in each direction.
  These meshes condense in the layer regions and are equidistant outside 
  the layer region. The points, where the mesh-character changes, are called
  \emph{transition points}. We define them by
  \begin{gather}\label{eq:ass:trans}
    \lambda_x :=\frac{\sigma\eps}{\beta}\ln N \le \frac{1}{2}
   \quad\text{and}\quad
    \lambda_y :=\sigma\sqrt\eps\ln N          \le \frac{1}{4},
  \end{gather}
  with some user-chosen positive parameter $\sigma>0$.
  In \eqref{eq:ass:trans} we assumed 
  \begin{gather}\label{eq:ass:eps}
   \eps\leq \min\left\{\frac{\beta}{2\sigma}(\ln N)^{-1},\frac{1}{16\sigma^2}(\ln N)^{-2}\right\}
       \leq C(\ln N)^{-2}
  \end{gather}
  which is typically for \eqref{eq:Lu} as otherwise $N$ would be exponentially 
  large in $\eps$.

  Using these transition points, the domain $\Omega$ is divided into the 
  subdomains $\Omega_{11}$, $\Omega_{12}$, $\Omega_{21}$ and $\Omega_{22}$ 
  as shown in Fig.~\ref{fig:s_mesh}.
  Here $\Omega_{12}$ covers the
  exponential layer, $\Omega_{21}$ the characteristic layers, $\Omega_{22}$
  the corner layers and $\Omega_{11}$ the remaining non-layer region.
  \begin{figure}[tb]
   \begin{center}
     \begin{minipage}[c]{6cm}
      \vspace*{0cm}
      \setlength{\unitlength}{0.55pt}
      \begin{picture}(256,256)
       \put(  0,  0){\line( 0, 1){256}}
       \put( 36,  0){\line( 0, 1){256}}
       \put(256,  0){\line( 0, 1){256}}
       \put(  0,  0){\line( 1, 0){256}}
       \put(  0, 40){\line( 1, 0){256}}
       \put(  0,216){\line( 1, 0){256}}
       \put(  0,256){\line( 1, 0){256}}

       \put(  2, 16){$\Omega_{22}$}
       \put(  2,128){$\Omega_{12}$}
       \put(  2,226){$\Omega_{22}$}
       \put(128, 16){$\Omega_{21}$}
       \put(128,128){$\Omega_{11}$}
       \put(128,226){$\Omega_{21}$}
      \end{picture}
     \end{minipage}
     \begin{minipage}[c]{6.5cm}
      \begin{align*}
       \Omega_{11}&:=[\lambda_x,1]\times [\lambda_y,1-\lambda_y],\\
       \Omega_{12}&:=[0,\lambda_x]\times [\lambda_y,1-\lambda_y],\\
       \Omega_{21}&:=[\lambda_x,1]\times \big([0,\lambda_y]\cup[1-\lambda_y,1]\big),\\
       \Omega_{22}&:=[0,\lambda_x]\times \big([0,\lambda_y]\cup[1-\lambda_y,1]\big)\\
      \end{align*}
     \end{minipage}
   \end{center}
   \caption{Decomposition of $\Omega$ into subregions.\label{fig:s_mesh}}
  \end{figure}

  By choosing the transition points $\lambda_x$ and $\lambda_y$ according to
  \eqref{eq:ass:trans}, the layer terms $w_1$, $w_2$, and $w_{12}$ of $u$ are of
  size $\ord{N^{-\sigma}}$ on $\Omega_{11}$, i.e.,
  \begin{gather*}
     \big|w_1(x,y)\big| + \big|w_2(x,y)\big| + \big|w_{12}(x,y)\big|
     \leq C N^{-\sigma}  \quad\text{for} \ (x,y)\in\Omega_{11}.
  \end{gather*}
  The parameter $\sigma$ is typically equal to the formal
  order of the numerical method or is chosen slightly larger to accommodate the error analysis. The 
  precise definition of $\sigma$ will be given later.

  The domain $\Omega_{11}$ will be dissected uniformly while the dissection in the
  other subdomains depends on the \emph{mesh generating} function $\phi$.
  This function is monotonically increasing and satisfies
  $\phi(0)=0$ and $\phi(1/2)=\ln N$.
  The precise definition of the tensor product mesh $T^N$ is given by the mesh points
  \begin{align*}
   x_i&:=\begin{cases}
          \frac{\sigma\eps}{\beta}\phi\left(\frac{i}{N}\right),
                                          &i=0,\dots,N/2,\\
          1-2(1-\lambda_x)(1-\frac{i}{N}),&i=N/2,\dots,N,
         \end{cases}\\
   y_j&:=\begin{cases}
          \sigma\sqrt{\eps}\phi\left(\frac{2j}{N}\right),  &j=0,\dots,N/4,\\
          \half+(1-2\lambda_y)(\frac{2j}{N}-1),            &j=N/4,\dots,3N/4,\\
          1-\sigma\sqrt{\eps}\phi\left(2-\frac{2j}{N}\right),&j=3N/4,\dots,N.
         \end{cases}
  \end{align*}
%
  Now with an arbitrary function $\phi$ fulfilling
  above conditions, an \emph{S-type mesh} is defined.
  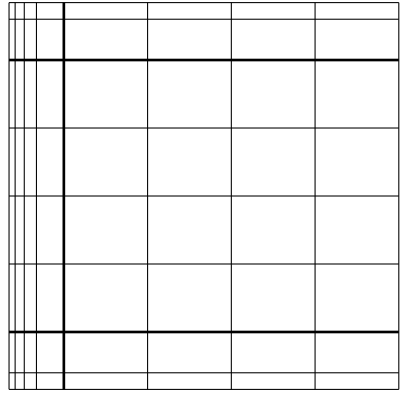
\begin{figure}[tb]
   \begin{center}
     \begin{minipage}[c]{6cm}
      \vspace*{0cm}
      \setlength{\unitlength}{0.57pt}
      \begin{picture}(256,256)
       \put(  0,  0){\line( 0, 1){256}}
       \put(  4,  0){\line( 0, 1){256}}
       \put( 10,  0){\line( 0, 1){256}}
       \put( 18,  0){\line( 0, 1){256}}

       \put( 91,  0){\line( 0, 1){256}}
       \put(146,  0){\line( 0, 1){256}}
       \put(201,  0){\line( 0, 1){256}}
       \put(256,  0){\line( 0, 1){256}}

       \put(  0,  0){\line( 1, 0){256}}
       \put(  0, 11){\line( 1, 0){256}}

       \put(  0, 83){\line( 1, 0){256}}
       \put(  0,128){\line( 1, 0){256}}
       \put(  0,173){\line( 1, 0){256}}

       \put(  0,245){\line( 1, 0){256}}
       \put(  0,256){\line( 1, 0){256}}
       
       \linethickness{1pt}
       \put(  0,218){\line( 1, 0){256}}
       \put(  0, 38){\line( 1, 0){256}}
       \put( 36,  0){\line( 0, 1){256}}
      \end{picture}
     \end{minipage}
   \end{center}
   \caption{Layer-adapted mesh $T^8$ of $\Omega$.\label{fig:mesh}}
  \end{figure}
  Fig.~\ref{fig:mesh} shows an example of such a mesh.

  Related to the mesh-generating function $\phi$, we define by
  \[
    \psi=\e^{-\phi}
  \]
  the \emph{mesh-characterising function} $\psi$
  which is monotonically decreasing with $\psi(0)=1$ and $\psi(1/2)=N^{-1}$.

  In Table~\ref{tab:S-mesh} some representatives of S-type meshes
  from~\cite{RoosLins99} are given.
  The polynomial S-mesh has an additional parameter $m>0$ to
  adjust the grading inside the layer.
  \begin{table}[tb]
    \caption{Some examples of mesh-generating and mesh-characterising functions of
             S-type meshes.\label{tab:S-mesh}}
     \begin{center}
      \begin{tabular}{l|l l| l l}
        Name & $\phi(t)$ & $\max\phi'$ & $\psi(t)$ & $\max|\psi'|$\\
        \hline
        \rule{0pt}{1.1em}Shishkin mesh & $2t\ln N$
                      & $2\ln N$
                      & $N^{-2t}$
                      & $2\ln N$\\
        Bakhvalov S-mesh & $-\!\ln(1\!-\!2t(1\!-\!N^{-1}))$
                 & $2N$
                 & $1\!-\!2t(1\!-\!N^{-1})$
                 & 2\\
        polynomial S-mesh & $(2t)^m\ln N$
                          & $2m\ln N$
                          & $N^{-(2t)^m}$
                          & $C(\ln N)^{1/m}$\\
        modified Bakhvalov S-mesh & $\frac{t}{q-t},\,q=\half(1+\frac{1}{\ln N})$
                          & $3\ln^2N$
                          & $\e^{-\frac{t}{q-t}}$
                          & $3/(2q)\leq 3$
      \end{tabular}
    \end{center}
  \end{table}

  In order to provide sufficient properties for our convergence analysis, the meshes need to
  fulfil some additional assumptions.
  \begin{ass}\label{ass:s_type}
   Let the mesh-generating function $\phi$ be piecewise differentiable
   such that
   \begin{gather}\label{eq:ass:s_type1}
    \max_{t\in[0,\half]} \phi'(t)\leq C N\,\text{ or equivalently }
    \max_{t\in[0,\half]} \frac{|\psi'(t)|}{\psi(t)}\leq C N
   \end{gather}
   is fulfilled. Moreover, let $\phi$ fulfil
   \begin{gather}\label{eq:ass:s_type2}
    \min_{i=0,\dots,N/2-1}\left(\phi\left(\frac{i+1}{N}\right)-
                                \phi\left(\frac{i}{N}\right)\right)\geq C N^{-1}.
   \end{gather}
   Finally we assume 
   \begin{gather}\label{eq:ass:s_type3}
     \max|\psi'|:=\max_{t\in[0,\half]}|\psi'(t)|\leq C \left(\frac{N}{\ln N}\right)^{1/2}.
   \end{gather}
  \end{ass}
  \begin{rem}
   Note that \eqref{eq:ass:s_type1} is satisfied for all meshes given in
   Table~\ref{tab:S-mesh}.
   Assumption~\eqref{eq:ass:s_type2} allows to bound the mesh width in the layer
   regions from below while applying an inverse inequality. This additional
   assumption restricts the use of S-type meshes from Table~\ref{tab:S-mesh}.
   For the original Shishkin mesh, we have
   \begin{gather*}
    \min_{i=0,\dots,N/2-1}\left(\phi\left(\frac{i+1}{N}\right)-
                                \phi\left(\frac{i}{N}\right)\right)
       = C N^{-1}\ln N
       \geq C N^{-1}.
   \end{gather*}
   The Bakhvalov S-mesh and its modification both
   fulfil
   \begin{gather*}
    \min_{i=0,\dots,N/2-1}\left(\phi\left(\frac{i+1}{N}\right)-
                                \phi\left(\frac{i}{N}\right)\right)
       \geq C N^{-1}.
   \end{gather*}
   But the polynomial S-type mesh yields
   \begin{gather*}
    \min_{i=0,\dots,N/2-1}\left(\phi\left(\frac{i+1}{N}\right)-
                                \phi\left(\frac{i}{N}\right)\right)
       \geq C N^{-m}
   \end{gather*}
   such that Assumption~\eqref{eq:ass:s_type2} fails for $m>1$.
   
   The restriction \eqref{eq:ass:s_type3} is fulfilled for all meshes of 
   Table~\ref{tab:S-mesh}. Nevertheless, S-meshes fulfilling the other two
   assumptions such that \eqref{eq:ass:s_type3} is violated are possible, see
   \cite[Remark 14]{FrM10}. The quantity $1+(N^{-1}\ln N)^{1/2}\max|\psi'|$ 
   arises in the convergence analysis of the Galerkin FEM, see~\cite{FrM10},
   and can be bounded by a constant $C$ with the help of \eqref{eq:ass:s_type3}.
  \end{rem}

  Using \eqref{eq:ass:s_type1} we bound the mesh width inside the layers from above.
  Let $h_i:=x_i-x_{i-1}$ and $t_i=i/N$.
  Then, it holds for $i=1,\dots,N/2$ and $t\in[t_{i-1},t_i]$ (with $\max \phi'$
  taken over $t\in[t_{i-1},t_i]$)
  \begin{equation}\label{eq:psi}
  \begin{aligned}
    \psi(t_i)  = \e^{-\phi(t_i)}
               = \e^{-(\phi(t_i)-\phi(t))}\e^{-\phi(t)}
               & \geq \e^{-(\phi(t_i)-\phi(t_{i-1}))}\psi(t)\\
               & \geq \e^{-N^{-1}\max\phi'}\psi(t)
               \geq C\psi(t)
  \end{aligned}
  \end{equation}
  where we used~\eqref{eq:ass:s_type1} for the last estimate.
  Furthermore, we have
  \[
   x=\frac{\sigma\eps}{\beta}\phi(t)=-\frac{\sigma\eps}{\beta}\ln\psi(t)
   \quad\mbox{which gives}\quad
   \psi(t)=\e^{-\beta x/(\sigma\eps)}.
  \]
  Using this, the monotonicity of $\psi$, and~\eqref{eq:psi}, we obtain for
  $i=1,\ldots,N/2$ and $x\in[x_{i-1},x_i]$
  \begin{align}
    h_i & = \frac{\sigma\eps}{\beta}(\phi(t_i)-\phi(t_{i-1}))
          \leq \frac{\sigma}{\beta} \eps N^{-1}\max_{t\in[t_{i-1},t_i]}\phi'(t)
          \leq \frac{\sigma}{\beta} \eps N^{-1}\left(
          \max_{t\in[t_{i-1},t_i]}|\psi'(t)|\right)/\psi(t_i)\notag\\
        & \leq C \eps N^{-1}\left(\max_{t\in[t_{i-1},t_i]}|\psi'(t)|\right)/\psi(t)
          \leq C \eps N^{-1}\max|\psi'|\e^{\beta x/(\sigma\eps)}
           \label{eq:hi_estimate}
  \end{align}
  where again $\max|\psi'|:=\max\limits_{t\in[0,1/2]}|\psi'(t)|$.\\
  Similarly, we get for $j=1,\dots,N/4$ and $j=3N/4+1,\dots,N$
  \begin{align}
   k_j:=y_j-y_{j-1}
      &\leq C \eps^{1/2}N^{-1}\max|\psi'|
      \begin{cases}
       \e^{y/(\sigma\eps^{1/2})},    & j\leq N/4,\\
       \e^{(1-y)/(\sigma\eps^{1/2})},& j>3N/4,
      \end{cases}
   \label{eq:kj_estimate}
  \end{align}
  with $y\in[y_{j-1},y_j]$. Of course, the simpler bounds
  \begin{alignat*}{3}
   h_i & \leq C \eps N^{-1}\max\phi' && \leq C \eps,
   &\qquad& i=1,\ldots,N/2,\\
   k_j & \leq C \eps^{1/2}N^{-1}\max\phi' && \leq C \eps^{1/2},
   && j=1,\ldots,N/4, \; 3N/4+1,\ldots,N,
  \end{alignat*}
  follow also from~\eqref{eq:ass:s_type1}.

  For the maximal mesh sizes inside the layer regions
  \[
    h:=\max_{i=1,\dots,N/2} h_i
       \quad\mbox{and}\quad
    k:=\max_{j=1,\dots,N/4} k_j
  \]
  we assume for simplicity of the presentation
  \begin{gather}\label{eq:ass:hk}
    h\leq k\leq N^{-1}\max|\psi'|
  \end{gather}
  which represents for some meshes a restriction on $\eps$. 
  With this assumption convergence results like $\ord{h+k+N^{-1}\max|\psi'|}$
  become $\ord{N^{-1}\max|\psi'|}$.
  
  We denote by $\tau_{ij}=[x_{i-1},x_i]\times[y_{j-1},y_j]$ a specific element and by
  $\tau$ a generic mesh rectangle. Note that the mesh cells are assumed to be
  closed.

\section{Polynomial Spaces and Interpolation}

  Having a discretisation of the domain $\Omega$, let us come to discretising the
  infinite-dimensional function space $H_0^1(\Omega)$ by higher-order,
  finite-dimensional polynomial spaces.
  Let our discrete space be given by
  \begin{gather}\label{eq:fespace}
    V^N:=\Big\{v\in H_0^1(\Omega):v|_\tau\in \mathcal{E}(\tau)\;
    \forall\tau\in T^N\Big\}
  \end{gather}
  with an yet unspecified local finite element space $\mathcal{E}(\tau)$.
  
  Let $\hat{\tau}=[-1,1]^2$ denote the reference element. We will look 
  at two different polynomial spaces, the \emph{full $\QS_p$-space} given locally by
  \[
    \QS_p(\hat{\tau})
      =\text{span}\Big\{\{1,\xi,\dots,\xi^p\}\times
                    \{1,\eta,\dots,\eta^p\}\Big\},
  \]
  and the \emph{Serendipity-space} $\QS_p^\oplus$ defined locally by
  enriching the polynomial space $\PS_p$ with two edge-bubble functions:
  \begin{alignat*}{2}
   \QS_p^{\oplus}(\hat{\tau})=\PS_p(\hat{\tau})
                 &\oplus \text{span}\Big\{&&(1+\xi)(1-\eta^2)\eta^{p-2},
                                            (1+\eta)(1-\xi^2)\xi^{p-2}\Big\}.
  \end{alignat*}
  This polynomial space is also known as ``trunk element''~\cite{SzBab91,Melenk03,ABFG01,ArAw10}.
  It is the continuous quadrilateral element with the fewest degrees of freedom
  containing $\PS_p$.
  
  Both spaces can be written in the general form
  \[
   \QS_p^\clubsuit(\hat{\tau})=\text{span}
    \bigg\{\{1,\xi\}\times\{1,\eta,\dots,\eta^p\} \cup
           \{1,\xi,\dots,\xi^p\}\times\{1,\eta\} \cup \xi^2\eta^2 \widetilde{\QS}(\hat{\tau})\bigg\}
  \]
  with
  \[
   \widetilde{\QS}(\hat{\tau})=\QS_{p-2}(\hat{\tau})
  \]
  for the full space and
  \[
   \widetilde{\QS}(\hat{\tau})=\begin{cases}
                                \emptyset,&\text{ for }{p=2,\,3},\\
                                \PS_{p-4}(\hat{\tau}),&\text{ for }{p\geq 4}
                              \end{cases}
  \]
  for the Serendipity space.
  
  Note that in both cases we find $s_0\ge s_1\ge\dots\ge s_{p-2}$, such that
  \begin{gather}\label{eq:Lagrange:ki}
    \widetilde{\QS}(\hat{\tau}) = \text{span} \Big\{ \xi^i\eta^j\::\:i=0,\dots,p-2,\:
    j=0,\dots,s_i\Big\}.
  \end{gather}
  Therein the $s_i$ can also be negative.
  
  Figure~\ref{fig:Q}
  \begin{figure}[tb]
    \centerline{\includegraphics[width=0.4\textwidth]{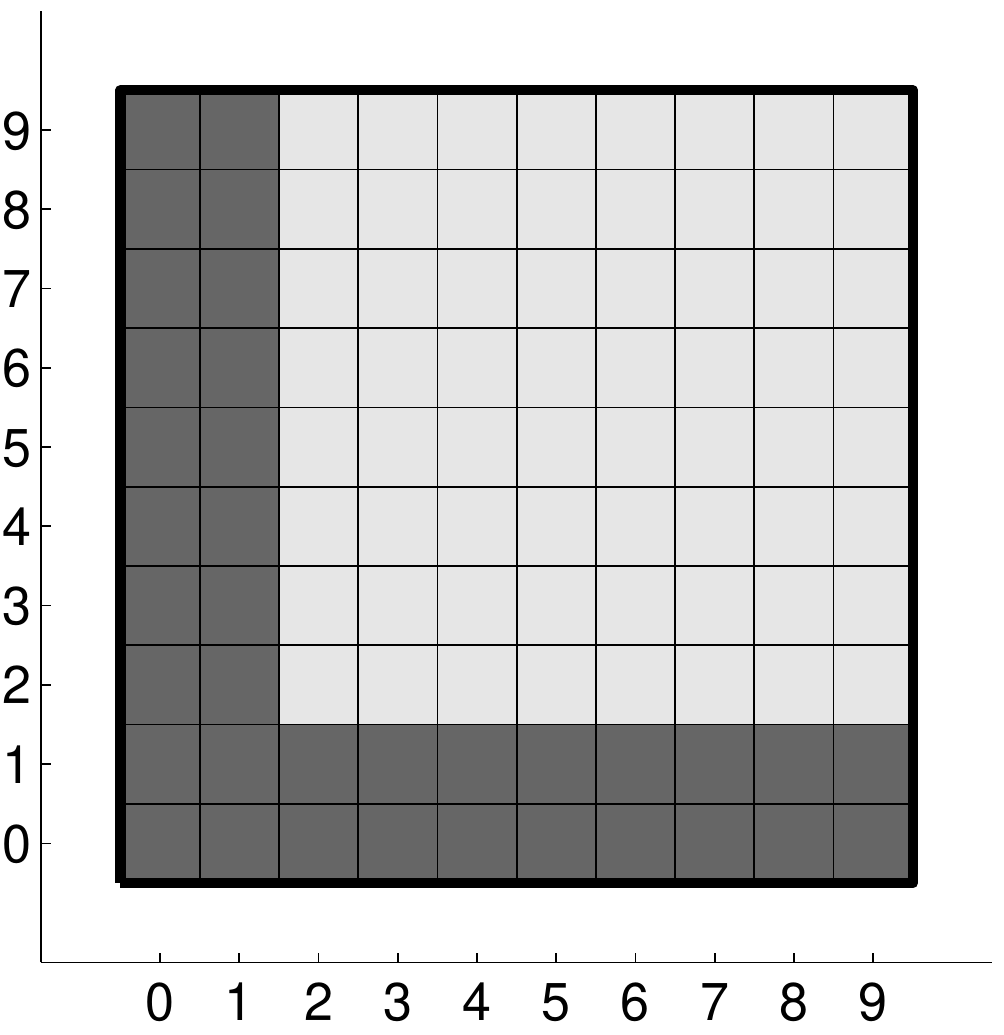}
                \quad
                \includegraphics[width=0.4\textwidth]{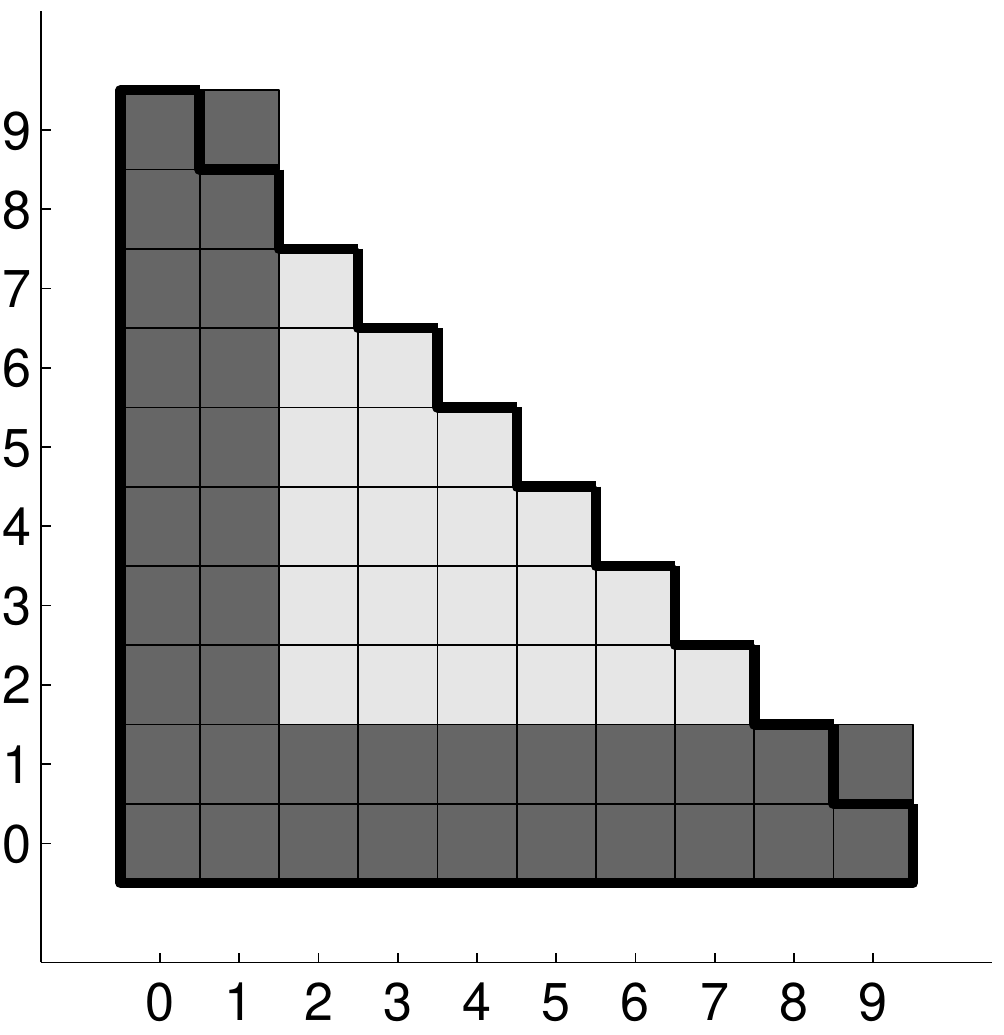}}
    \caption{Full space $\QS_p(\hat{\tau})$ (left) and
             Serendipity space $\QS_p^{\oplus}(\hat{\tau})$ (right) for $p=9$.
             \label{fig:Q}}
  \end{figure}
  gives a graphical representation of the two spaces.
  Therein a square at position $(i,j)$ stands for a basis function $\xi^i\eta^j$ of $\QS_p^\clubsuit(\hat{\tau})$.
  The darker squares correspond to those functions present in both spaces,
  while the lighter ones represent $\xi^2 \eta^2\widetilde{\QS}(\hat{\tau})$.
  Note that it holds
  \[
   \PS_p\subset\QS_p^\oplus\subset\QS_p
  \]
  and that $\QS_p^\oplus$ uses only about half the number of degrees of freedom that $\QS_p$
  uses.

  Now let us come to defining interpolation operators for these two spaces. 
  We will consider two types of interpolation: vertex-edge-cell interpolation 
  and Lagrange interpolation.
  
  \subsection*{Vertex-edge-cell interpolation operator}
  The first interpolation operator is
  based on point evaluation at the vertices, line integrals along the
  edges and integrals over the cell interior, see \cite{GR86,Lin91}.

  Let $\hat{a}_i$ and $\hat{e}_i$, $i=1,\ldots,4$, denote the vertices and
  edges of $\hat{\tau}$, respectively. We define the vertex-edge-cell
  interpolation operator
  $\hat{\pi}:C(\hat{\tau})\to\QS_p^\clubsuit(\hat{\tau})$ by
  \begin{subequations}\label{eq:general:def_inter}
  \begin{alignat}{2}
   \hat{\pi} \hat{v}(\hat{a}_i)&=\hat{v}(\hat{a}_i),\,\quad i=1,\dots,4,
   &&\label{eq:general:def_inter1}\\
   \int_{\hat{e}_i}(\hat{\pi}\hat{v})\hat{q} &= \int_{\hat{e}_i} \hat{v} \hat{q},
   \quad i=1,\dots,4,\quad
   &&\hat{q}\in \PS_{p-2}(\hat{e}_i),\label{eq:general:def_inter2}\\
   \iint_{\hat{\tau}} (\hat{\pi}\hat{v})\hat{q} &= \iint_{\hat{\tau}} \hat{v} \hat{q},
   &&\hat{q}\in \widetilde{\QS}(\hat{\tau}).
                                            \label{eq:general:def_inter3}
  \end{alignat}
  \end{subequations}
  This operator is uniquely defined and can be extended to the globally 
  defined interpolation operator
  $\pi^N:C(\overline{\Omega})\to V^N$ by
  \[
   (\pi^N v)|_\tau := \big(\hat{\pi}(v\circ F_\tau)\big)\circ F_\tau^{-1}
   \quad\forall\tau\in T^N,\,v\in
   C(\overline{\Omega}),
  \]
  with the bijective reference mapping $F_\tau:\hat{\tau}\to\tau$.
  \begin{lem}
    For the interpolation operator $\pi^N:C(\overline{\Omega})\rightarrow V^N$
    the stability property 
   \begin{gather}\label{eq:high:prop_stab:Lin}
    \bignorm{\pi^N w}{L_\infty(\tau)}\leq C\norm{w}{L_\infty(\tau)}\quad\forall w\in C(\tau),\,
    \forall \tau\subset\overline{\Omega},
   \end{gather}
   holds and we have the anisotropic error estimates
   \begin{subequations}\label{eq:high:prop_error_L2:Lin}
   \begin{align}
    \bignorm{w-\pi^N w}{L_q(\tau_{ij})}
     &\leq C \sum_{r=0}^s
                \bignorm{h_i^{s-r}k_j^r\frac{\partial^{s}w}
                              {\partial x^{s-r}\partial y^{r}}}
                        {L_q(\tau_{ij})},\\
    \bignorm{(w-\pi^N w)_x}{L_q(\tau_{ij})}
     &\leq C \sum_{r=0}^t
                \bignorm{h_i^{t-r}k_j^r\frac{\partial^{t+1}w}
                              {\partial x^{t-r+1}\partial y^{r}}}
                        {L_q(\tau_{ij})},\\
    \bignorm{(w-\pi^N w)_y}{L_q(\tau_{ij})}
     &\leq C \sum_{r=0}^t
                \bignorm{h_i^{t-r}k_j^r\frac{\partial^{t+1}w}
                              {\partial x^{t-r}\partial y^{r+1}}}
                        {L_q(\tau_{ij})}
   \end{align}
   \end{subequations}
   for $\tau_{ij}\subset\overline{\Omega}$ and $q\in[1,\infty]$, $2\leq s\leq p+1$, $1\leq t\leq p$.
  \end{lem}
  \begin{proof}
   The proof for arbitrary $\QS^\clubsuit_p$ can be found in \cite{FrM10_1} and
   for the full space $\QS_p$ also in e.g. \cite{ST08}.
  \end{proof}

  \subsection*{Lagrange-type interpolation}
  The second interpolation type we consider is the Lagrange type, i.e. it uses only
  point-value information.

  Let $-1=\xi_0<\xi_1<\dots<\xi_{p-1}<\xi_p=+1$ and
  $-1=\eta_0<\eta_1<\dots<\eta_{p-1}<\eta_p=+1$ be two increasing sequences
  of $p+1$ points of $[-1,+1]$ which include both end points. 
  We define the Lagrange-type interpolation operator
  $\hat{J}:C(\hat{\tau})\to\QS_p^\clubsuit(\hat{\tau})$ by values at
  the vertices
  \begin{subequations}\label{eq:hatJ}
  \begin{align}
  \begin{alignedat}{2}
  (\hat{J}\hat{v})(\pm1,-1) & := \hat{v}(\pm1,-1),&\qquad
  (\hat{J}\hat{v})(\pm1,+1) & := \hat{v}(\pm1,+1)
  \end{alignedat}
  \\
  \intertext{values on the edges}
  \left.
  \begin{alignedat}{2}
  (\hat{J}\hat{v})(\xi_i,\pm1) & := \hat{v}(\xi_i,\pm1),&\qquad& i=1,\dots,p-1,\\
  (\hat{J}\hat{v})(\pm1,\eta_j) & := \hat{v}(\pm1,\eta_j),&\qquad& j=1,\dots,p-1,\\
  \end{alignedat}
  \qquad \right\}&\\
  \intertext{and values in the interior}
  \begin{alignedat}{2}\label{eq:hatJ:int}
  (\hat{J}\hat{v})(\xi_{i+1},\eta_{j+1})&
  := \hat{v}(\xi_{i+1},\eta_{j+1}),&\qquad&
  i=0,\dots,p-2, j=0,\dots,s_i,
  \end{alignedat}
  \end{align}
  \end{subequations}
  where the $s_i$ are those given in \eqref{eq:Lagrange:ki}.

  In \cite{FrM10_1} it is shown that this operator is uniquely defined.
  What is left to specify are the sequences $\{\xi_i\}$ and $\{\eta_j\}$.
  Here we consider two choices:\\
  
  \textbf{1) equidistant distribution}: We define the operator
  $J^N:C(\overline{\Omega})\to V^N$ by
  \[
   (J^N v)|_\tau := \big(\hat{J}(v\circ F_\tau)\big)\circ F_\tau^{-1}
   \quad\forall\tau\in T^N,\,v\in
   C(\overline{\Omega}),
  \]
  with the bijective reference mapping $F_\tau:\hat{\tau}\to\tau$ and 
  the local sequences
  \[
   \xi_i=\eta_i=-1+2i/p,\,i=0,\dots,p.
  \]
  
  \textbf{2) distribution according to the Gauß-Lobatto quadrature rule}:\\
  Let $-1=t_0<t_1<\dots<t_p=1$, be the zeros of
  \[
    (1-t^2)L_{p}'(t)=0,\quad t\in[-1,1],
  \]
  where $L_p$ is the Legendre polynomial of degree $p$.
  These points are also used in the Gauß-Lobatto
  quadrature rule of approximation order $2p-1$. Therefore, we refer
  to them as Gauß-Lobatto points.
  In literature they are also named Jacobi points \cite{Li04} as they are 
  also the zeros of the orthogonal Jacobi-polynomials $P_p^{(1,1)}$ of 
  order $p$.
  Now we define the operator
  $I^N:C(\overline{\Omega})\to V^N$ by
  \[
   (I^N v)|_\tau := \big(\hat{J}(v\circ F_\tau)\big)\circ F_\tau^{-1}
   \quad\forall\tau\in T^N,\,v\in
   C(\overline{\Omega}),
  \]
  with the bijective reference mapping $F_\tau:\hat{\tau}\to\tau$ and 
  the local sequences
  \[
   \xi_i=\eta_i=t_i,\,i=0,\dots,p.
  \]

  \begin{lem}
    The interpolation operators $J^N:C(\overline{\Omega})\rightarrow V^N$
    and $I^N:C(\overline{\Omega})\rightarrow V^N$ yield the stability property
   \begin{gather}\label{eq:high:prop_stab:Lagrange}
    \bignorm{J^N w}{L_\infty(\tau)}+
    \bignorm{I^N w}{L_\infty(\tau)}
    \leq C\norm{w}{L_\infty(\tau)}\quad\forall w\in C(\tau),\,
    \forall \tau\subset\overline{\Omega},\\
   \end{gather}
   and we have the anisotropic error estimates
   \begin{subequations}\label{eq:high:prop_error_L2:Lagrange}
   \begin{align}
    \bignorm{w-J^N w}{L_q(\tau_{ij})}+
    \bignorm{w-I^N w}{L_q(\tau_{ij})}
     &\leq C \sum_{r=0}^s
                \bignorm{h_i^{s-r}k_j^r\frac{\partial^{s}w}
                              {\partial x^{s-r}\partial y^{r}}}
                        {L_q(\tau_{ij})},\\
    \bignorm{(w-J^N w)_x}{L_q(\tau_{ij})}+
    \bignorm{(w-I^N w)_x}{L_q(\tau_{ij})}
     &\leq C \sum_{r=0}^t
                \bignorm{h_i^{t-r}k_j^r\frac{\partial^{t+1}w}
                              {\partial x^{t-r+1}\partial y^{r}}}
                        {L_q(\tau_{ij})},\\
    \bignorm{(w-J^N w)_y}{L_q(\tau_{ij})}+
    \bignorm{(w-I^N w)_y}{L_q(\tau_{ij})}
     &\leq C \sum_{r=0}^t
                \bignorm{h_i^{t-r}k_j^r\frac{\partial^{t+1}w}
                              {\partial x^{t-r}\partial y^{r+1}}}
                        {L_q(\tau_{ij})}
   \end{align}
   \end{subequations}
   for $\tau_{ij}\subset\overline{\Omega}$ and $q\in[1,\infty]$, $2\leq s\leq p+1$, $1\leq t\leq p$.
  \end{lem}
  \begin{proof}
   The proof for arbitrary $\QS^\clubsuit_p$ can be found in \cite{FrM10_1} and
   for the full space $\QS_p$ also in e.g. \cite{Apel99}.
  \end{proof}

   There is a strong connection between $\pi^N$ and $I^N$ in the case of $\QS_p$-elements.
   Let us spend a subscript for the polynomial order $p$, i.e. we write $\pi_p^N$ and $I_p^N$
   for the interpolation operators mapping into $V^N$ with local polynomial spaces $\QS_p$.
   Then it holds the identity
   \begin{gather}\label{eq:inter:connection1}
    \pi_p^N=I_p^N\pi_{p+1}^N,
   \end{gather}
   see \cite[Lemma 3.3]{Fr12}. A direct consequence is the additional identity
   \begin{gather}\label{eq:inter:connection2}
    \pi_p^Nv=I_p^Nv
             +(\pi^N_{p+1}v-v)
             +\left(I^N_p(\pi^N_{p+1}v-v)-(\pi^N_{p+1}v-v)\right)
   \end{gather}
   for arbitrary $v\in C(\bar \Omega)$.
   It shows the distance between both interpolation operators to be proportional to terms
   of order $p+1$. The identity \eqref{eq:inter:connection1} (with the properly redefinition
   of the interpolation operators therein) does also hold for
   the Serendipity spaces $\QS_2^\oplus$ and $\QS_3^\oplus$, but not for $\QS_p^\oplus$
   with $p\geq 4$. This can be shown analogously to the proof of \cite[Lemma 3.3]{Fr12}.
   
   The reason for the failed identity lies in the definition of the interior degrees of freedom
   \eqref{eq:general:def_inter3} and \eqref{eq:hatJ:int}. For $\QS_2^\oplus$ and $\QS_3^\oplus$
   these conditions are not existent and therefore always fulfilled,
   while for higher order $p$ they do not match any more.
  
  \section{Numerical Methods}
  Let us come to the numerical methods that we will consider in the next chapter.
  
  \subsection{Galerkin FEM}\label{ssec:Galerkin}
  The first method will be the unstabilised Galerkin FEM given by:\bigskip

  Find $u_{Gal}^N\in V^N$ such that
  \begin{equation}\label{eq:Gal_form}
    a_{Gal}(u_{Gal}^N, v^N) = (f,v^N)\qquad \forall v^N\in V^N.
  \end{equation}
  This problem possesses a unique solution due to~\eqref{eq:coer_ass}. 
  Furthermore, the Galerkin orthogonality
  \begin{equation} \label{eq:Gal:ortho}
    a_{Gal}(u - u_{Gal}^N, v^N) = 0\qquad \forall v^N\in V^N
  \end{equation}
  holds true and we have coercivity
  \begin{gather}\label{eq:Gal:coer}
   a_{Gal}(v,v)\geq \enorm{v}^2,\qquad v\in H^1_0(\Omega)
  \end{gather}
  where the energy norm is defined by~\eqref{eq:norm:energy}
  \begin{gather*}
    \enorm{v} := \left(\eps \norm{\grad v}{0}^2 + \gamma \norm{v}{0}^2\right)^{1/2}.
  \end{gather*}
  
  Since the standard Galerkin discretisation lacks stability even on S-type
  meshes, see the numerical results given in~\cite{LS01c,SCX10}, we will also 
  consider stabilised methods. A survey of several different stabilised method for 
  singularly perturbed problems can be found in the book~\cite{RST08}.

  \subsection{Streamline Diffusion FEM}\label{ssec:SDFEM}

  In 1979 Hughes and Brooks~\cite{HB79} introduced the streamline-diffusion finite element
  method (SDFEM), sometimes also called streamline upwind Petrov Galerkin finite element method (SUPG-FEM). 
  This method provides highly accurate solutions outside the layers and
  good stability properties. 
  Its basic idea is to add weighted local residuals to the variational formulation, i.e. to add
  \[
   \delta_\tau(Lu-f,-bw_x)_\tau=0
  \]
  where the constant parameters $\delta_\tau=\delta_{ij}\geq 0$ for $\tau\subset\Omega_{ij}$
  are user chosen and influence both stability and convergence. 
  A slightly different approach will be used in Chapter~\ref{cha:norm}. 
  
  Defining
  \[
      a_{stabSD}(v,w):=\sum_{\tau\in T^N}\delta_\tau(\eps\laplace v+bv_x-c v,bw_x)_\tau,
                             \qquad\mbox{for all }v,\,w\in H^1_0(\Omega)
  \]
  and 
  \[
      f_{SD}(w):=(f,w)-\sum_{\tau\in T^N}\delta_\tau(f,bw_x)_\tau,
                       \quad\mbox{for all }w\in H_0^1(\Omega)
  \]
  we obtain the streamline diffusion formulation of~\eqref{eq:Lu} by:
  Find $u_{SD}^N\in V^N$ such that
  \begin{equation}\label{eq:SD_form}
    a_{SD}(u_{SD}^N,v^N):=a_{Gal}(u_{SD}^N, v^N)+a_{stabSD}(u_{SD}^N, v^N) = f_{SD}(v^N),
    \quad\mbox{for all }v^N\in V^N.
  \end{equation}

  Associated with this method is the streamline diffusion norm,
  defined by
  \begin{gather}\label{eq:norm:SD}
   \tnorm{v}_{SD}:=\left(\eps\norm{\grad v}{0}^2+\gamma\norm{v}{0}^2+
                         \sum_{\tau\in T^N}\delta_\tau\norm{bv_x}{0,\tau}^2\right)^{1/2}.
  \end{gather}
  We have Galerkin orthogonality, and for
  \begin{gather}\label{eq:delta_coer}
   0\leq\delta_\tau\leq \frac{1}{2}\min\left\{\frac{\gamma}{\norm{c}{L_\infty(\tau)}^2},
                                      \frac{h_\tau^2}{\mu^2\eps}\right\},
  \end{gather}
  where $\mu\geq 0$ is a fixed constant such that the inverse inequality
  \[
   \norm{\laplace v^N}{0,\tau}\leq \mu h_\tau^{-1}\norm{\grad v^N}{0,\tau},
   \qquad \forall v^N\in V^N,\,\tau\in T^N
  \]
  holds with $h_{\tau_{ij}}:=\min\{h_i,k_j\}$,
  we have coercivity
  \begin{gather}\label{eq:SDFEM:coer}
   a_{SD}(v,v)\geq \frac{1}{2}\tnorm{v}_{SD}^2,\qquad v\in H^1_0(\Omega).
  \end{gather}

  A disadvantage of the SDFEM are several additional terms including second order derivatives 
  that have to be assembled in order to ensure the Galerkin orthogonality of the resulting method.
  Moreover, for systems of differential equations additional coupling between different species occurs.

  \subsection{Local Projection Stabilisation FEM}\label{ssec:LPSFEM}
  
  An alternative stabilisation technique overcoming some drawbacks of the SDFEM
  is the Local Projection Stabilisation method LPSFEM. Instead of adding weighted residuals,
  only weighted fluctuations $(id-\pi)$ of the streamline derivatives are added. 
  Therein $\pi$ denotes a projection into a discontinuous 
  finite element space. 
  
  Originally the method was introduced for Stokes and transport problems~\cite{BB01,BB04},
  but also applied to the Oseen problem in~\cite{BB06,MST07}. 
  In its original definition, the local projection method was proposed as a two-level
  method, where the projection space is defined on a coarser mesh consisting of patches of
  elements~\cite{BB01,BB04,BB06}. In this case, standard finite element spaces can
  be used for both the approximation space and the projection space.   
  Based on the existence of a special interpolation operator~\cite{MST07}, the one level 
  approach using enriched spaces was constructed.
  It was shown in~\cite{MST07} that it suffices to enrich the standard $\QS_p$-element, 
  $p\ge 2$, in 2d by just two additional bubble functions of higher order. For its 
  application on layer-adapted meshes for problems with exponential
  boundary layers see~\cite{Mat09,Mat09b}. 
  
  Here we will use the one level approach without enriching the polynomial spaces.
  Let $\pi_{\tau}$ denote the $L_2$-projection into the finite dimensional
  function space $D(\tau)=\mathcal{P}_{p-2}(\tau)$. The fluctuation operator
  $\kappa_{\tau}:L_2(\tau)\to L_2(\tau)$ is defined by
  $\kappa_{\tau} v:= v - \pi_{\tau}v$.
  In order to get additional control on the derivative in streamline
  direction, we define the stabilisation term
  \[
    s(u,v) := \sum_{\tau\in T^N}\delta_{\tau}
                   \big(\kappa_{\tau}(b u_x),
                        \kappa_{\tau}(b v_x)\big)_{\tau}
  \]
  with the parameters $\delta_\tau=\delta_{ij}\geq 0$, 
  $\tau\subset\Omega_{ij}$, which will be specified later.
  It was stated in~\cite{Fr08_1,Fr08_thesis} for different stabilisation methods
  that stabilisation is best if only applied in $\Omega_{11}\cup\Omega_{21}$.
  Therefore, we set $\delta_{12}=\delta_{22}=0$ in the following.

  The stabilised bilinear form $a_{LPS}$ is defined by
  \[
    a_{LPS}(u,v) := a_{Gal}(u,v) + s(u,v),\qquad u,v\in H^1_0(\Omega),
  \]
  and the stabilised discrete problem reads:\bigskip

  Find $u_{LPS}^N\in V^N$ such that
  \begin{equation} \label{eq:LPSFEM_form}
    a_{LPS}(u_{LPS}^N,v^N) = (f,v^N)\qquad \forall v^N\in V^N.
  \end{equation}
  Associated with this bilinear form is the LPS norm
  \begin{gather}\label{eq:norm:LPS}
    \lnorm{v} := \left(\eps \norm{\grad v}{0}^2 + \gamma \norm{v}{0}^2 + s(v,v)\right)^{1/2}.
  \end{gather}
  The bilinear form is coercive w.r.t. this norm
  \begin{equation} \label{eq:LPS:coer}
    a_{LPS}(v,v)\geq \lnorm{v}^2, \qquad v\in H^1_0(\Omega).
  \end{equation}
  Moreover, the solutions $u$ of~\eqref{eq:Lu} and $u_{LPS}^N$
  of~\eqref{eq:LPSFEM_form} do not fulfil the Galerkin orthogonality, but 
  \begin{equation} \label{eq:LPS:weak_conc}
    a_{LPS}(u - u_{LPS}^N, v^N) = s(u, v^N) \qquad \forall v^N\in V^N.
  \end{equation}

  The LPSFEM gives control over the fluctuations of the streamline derivative.
  In \cite{Knob10} a slight variation of the formulation is considered and an
  inf-sup condition w.r.t. the SDFEM norm is shown on a quasi-regular mesh. 
  Thus, this LPSFEM gives control over the full streamline derivative. Whether
  such a result holds on S-type meshes is not known.
  

%% file: results.tex
 \chapter{Uniform a-priori Error Estimation in Energy Norms}\label{cha:results}
 This chapter contains results from \cite{FrM10, FrM10_1, Fr11, Fr13_1} that are also given in 
 Appendix \ref{app:LPS}, \ref{app:nonstandard}, \ref{app:SDFEM} and \ref{app:phenomena}.
 All theoretical results will be accompanied by a numerical study using
 the singularly perturbed convection-diffusion problem
  \begin{subequations}\label{eq:num_example}
  \begin{alignat}{2}
    -\eps\laplace u - (2-x) u_x + \frac{3}{2} u & = f
    &\quad&\text{in }\Omega=(0,1)^2,\\
    u & = 0 && \text{on }\partial\Omega,
  \end{alignat}
  where the right-hand side $f$ is chosen such that
  \begin{gather}
    u(x,y) = \left(\cos\frac{\pi x}{2} - \frac{ \e^{-x/\eps} - \e^{-1/\eps}}%
    {1-\e^{-1/\eps}}
    \right)
    \frac{\left(1-\e^{-y/\sqrt{\eps}}\right)
    \left( 1-\e^{-(1-y)/\sqrt{\eps}} \right)}{1-\e^{-1/\sqrt{\eps}}}
  \end{gather}
  \end{subequations}
  is the solution. We will used a fixed perturbation parameter $\eps=10^{-6}$.
  Computations verifying the uniformity w.r.t. $\eps$ were also done. 
  Figure~\ref{fig:sol_1}
  \begin{figure}[btp]
    \centerline{
      \includegraphics[width=0.6\textwidth]{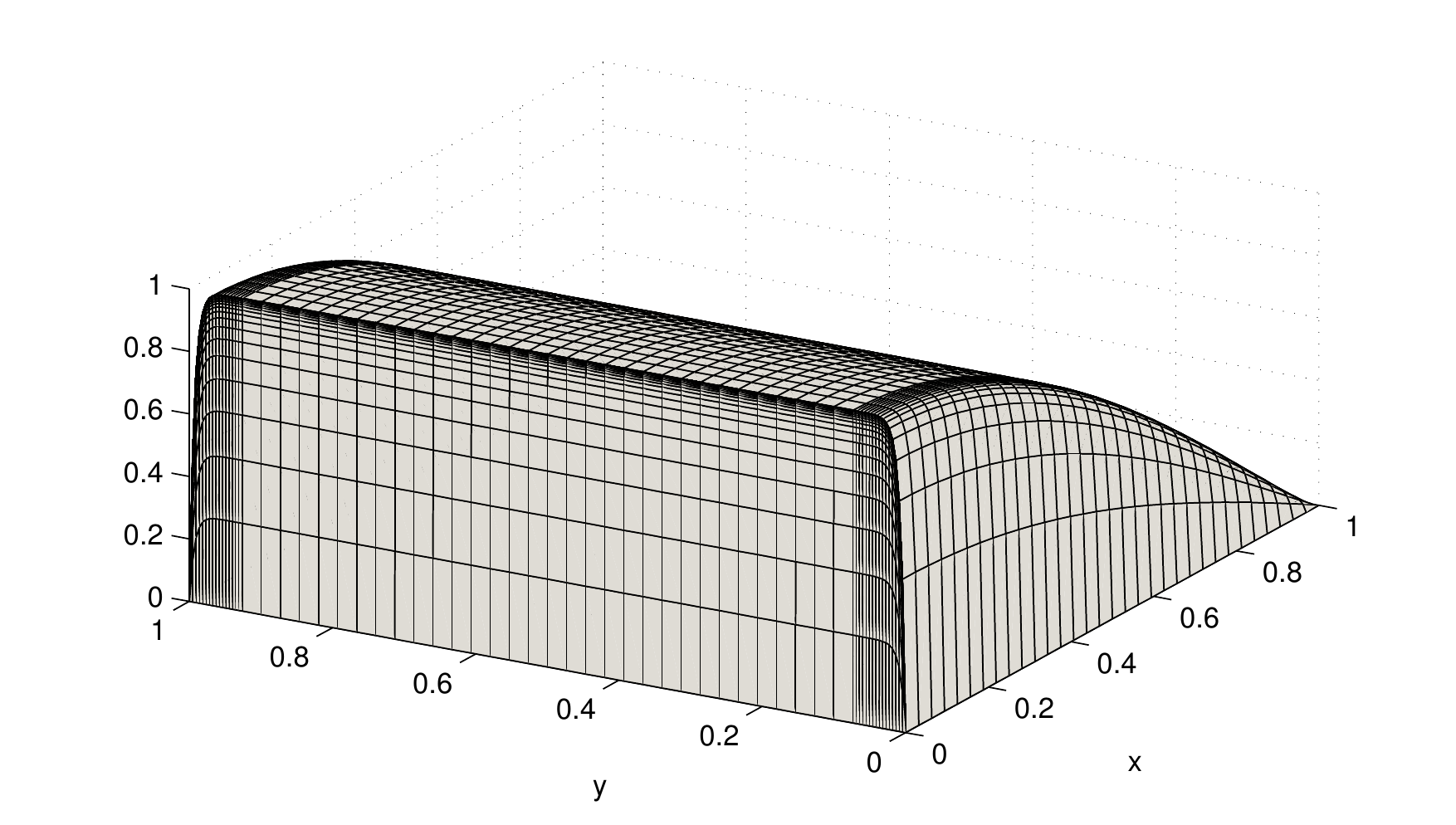}
    }
    \caption{Typical solution of~\eqref{eq:Lu} with two parabolic layers
             and an exponential layer.
             \label{fig:sol_1}}
  \end{figure}
  shows the resulting solution. For comparison,
  the energy norm of $u$ is in this case $\enorm{u}\approx 0.9975$.

 \section{Results for Galerkin FEM}\label{sec:results:GFEM}
%
 Let us start with results for the standard Galerkin FEM. In \cite{FrL08,Fr08_thesis}
 results for bilinear elements are presented. If the mesh parameter $\sigma$ fulfils
 $\sigma\geq 2$, then the convergence result of \cite{Roos02} holds
 \[
   \enorm{u-u_{Gal}^N}\leq C(N^{-1}\max|\psi'|)
 \]
 with $\max|\psi'|$ from e.g. Table~\ref{tab:S-mesh} and for
 $\sigma\geq 5/2$ the supercloseness result \cite{FrL08}
 \[
  \enorm{u^I-u_{Gal}^N}\leq C(N^{-1}\max|\psi'|)^2
 \]
  where $u^I$ denotes the nodal bilinear interpolant. In the higher-order case with
  either the full space $\QS_p$ or the serendipity space $\QS_p^\oplus$ results can be found in
  \cite{FrM10_1}. 
  
  \begin{thm}[Theorem 6 of \cite{FrM10_1}]\label{thm:Gal:conv}
%
   Let the solution $u$ of~\eqref{eq:Lu} satisfy Assumption~\ref{ass:dec} 
   and let $u_{Gal}^N$ denote the Galerkin solution of~\eqref{eq:Gal_form}.
   Then, we have for $\sigma\geq p+1$
   \begin{gather}\label{eq:high:scg}
    \enorm{u-u_{Gal}^N}\leq C\big(N^{-1}\max|\psi'|\big)^p.
   \end{gather}
  \end{thm}
  Thus, similar to the bilinear case, we achieve convergence of order $p$ in 
  the energy norm. To our knowledge, no supercloseness result is available in 
  literature in the higher-order case. Nevertheless, it can be observed numerically
  for the full space $\QS_p$.

  Let us come to the numerical example ~\eqref{eq:num_example}. We will use a Bakhvalov-S-mesh,
  as here $|\max\psi'|$ is bounded by a constant, see Table~\ref{tab:S-mesh}, and the convergence
  rates can be observed easiest. According to Theorem~\ref{thm:Gal:conv}
  we expect
  \[
   \enorm{u-u_{Gal}^N}\leq CN^{-p}.
  \]
  Table~\ref{tab:GFEM:conv}
  \begin{table}[bp]
   \begin{center}
    \caption{Convergence errors of Galerkin FEM for $\QS_p$- and $\QS_p^\oplus$-elements, and $p=4,\,5$\label{tab:GFEM:conv}}
    \begin{tabular}{r|ll|ll|ll|ll}
     \multicolumn{1}{c}{}& \multicolumn{8}{c}{$\enorm{u-u_{Gal}^N}$}\\
     \rule{0pt}{1.1em}$N$
         & \multicolumn{2}{c}{$\QS_4$}
         & \multicolumn{2}{c|}{$\QS_4^\oplus$}   
         & \multicolumn{2}{c}{$\QS_5$}
         & \multicolumn{2}{c}{$\QS_5^\oplus$}   \\[2pt]
     \hline\rule{0pt}{1.1em}
       8 & 6.633e-04 & 3.65 & 1.469e-03 & 3.68 & 1.330e-04 & 4.59 & 5.002e-04 & 4.53\\
      16 & 5.274e-05 & 3.83 & 1.147e-04 & 3.87 & 5.506e-06 & 4.79 & 2.160e-05 & 4.81\\
      32 & 3.715e-06 & 3.91 & 7.857e-06 & 3.94 & 1.985e-07 & 4.89 & 7.722e-07 & 4.92\\
      64 & 2.467e-07 & 3.96 & 5.106e-07 & 3.97 & 6.682e-09 & 4.95 & 2.553e-08 & 4.94\\
     128 & 1.590e-08 & 3.98 & 3.248e-08 & 3.99 & 2.169e-10 & 4.72 & 8.319e-10 & 0.12\\
     256 & 1.009e-09 & 3.98 & 2.046e-09 & 3.99 & 8.216e-12 &      & 7.644e-10 &     \\
     320 & 4.148e-10 &      & 8.396e-10 &      &           &      &           &
    \end{tabular}
   \end{center}
  \end{table}
  confirms our expectation. In this table the errors and their estimated orders
  of convergence are given for $\sigma=p+3/2$. We see for the spaces $\QS_4$ and $\QS_4^\oplus$ a 
  convergence of order four, while for the spaces $\QS_5$ and $\QS_5^\oplus$ we 
  obtain order five. Moreover, the switch from the full space to Serendipity-space 
  does increase the error only by a factor of two for $p=4$ and four for $p=5$.
  Thus the error is increased, but at the same time only 
  about half the number of degrees of freedom are used.

  Let us also look at supercloseness. Although no analytical result is given, 
  Table~\ref{tab:GFEM:super}
  \begin{table}[tbp]
   \begin{center}
    \caption{Supercloseness property of Galerkin FEM for $p=5$\label{tab:GFEM:super}}
    \begin{tabular}{r|lr|lr|lr|lr}
     \multicolumn{1}{c}{}& \multicolumn{6}{c}{$\QS_5$}&\multicolumn{2}{c}{$\QS_5^\oplus$}\\
     \rule{0pt}{1.1em}$N$
         & \multicolumn{2}{c}{$\enorm{\pi^Nu-u_{Gal}^N}$}
         & \multicolumn{2}{c}{$\enorm{I^Nu-u_{Gal}^N}$}   
         & \multicolumn{2}{c|}{$\enorm{J^Nu-u_{Gal}^N}$}
         & \multicolumn{2}{c}{$\enorm{\pi^Nu-u_{Gal}^N}$}\\[2pt]
     \hline\rule{0pt}{1.1em}
       8 & 3.026e-05 & 5.48  & 3.408e-05 & 5.41  & 9.474e-05 & 4.60 & 2.825e-04 & 4.37 \\
      16 & 6.765e-07 & 5.80  & 8.003e-07 & 5.74  & 3.894e-06 & 4.79 & 1.366e-05 & 4.68 \\
      32 & 1.213e-08 & 5.92  & 1.496e-08 & 5.88  & 1.406e-07 & 4.89 & 5.314e-07 & 4.83 \\
      64 & 1.999e-10 & 5.87  & 2.537e-10 & 5.88  & 4.736e-09 & 4.94 & 1.871e-08 & 4.86 \\
     128 & 3.428e-12 & -0.38 & 4.320e-12 & -0.04 & 1.538e-10 & 4.55 & 6.423e-10 & -0.25\\
     256 & 4.461e-12 &       & 4.442e-12 &       & 6.580e-12 &      & 7.642e-10 &      
    \end{tabular}
   \end{center}
  \end{table}
  shows for $p=5$ a supercloseness property of order $p+1$ for the Galerkin FEM with $\QS_p$-elements
  and the two interpolation operators $\pi^N$ (vertex-edge-cell interpolation) and $I^N$ 
  (Gauß-Lobatto interpolation). No such property is evident for $J^N$
  (equidistant Lagrange interpolation) or the Serendipity-elements. 
  
  For other polynomial degrees similar tables and conclusions can be given and are therefore omitted.
  We come back to the behaviour of $\pi^N$ and $I^N$ in the next section.
  
\section{Results for SDFEM}\label{sec:results:SDFEM}
  One of the most popular stabilisation methods is the SDFEM. This method can also 
  be used in connection with the general higher-order elements.
  Under certain restrictions on the stabilisation parameters
  convergence of order $p$ can be proved.
  
%
  \begin{thm}[Theorem 8 of \cite{Fr11}]\label{thm:SDFEM:conv}
%
   Let
   \begin{gather*}
    \delta_{11}\leq C,\quad
    \delta_{21}\leq C\max\{1,\eps^{-1/2}(N^{-1}\max|\psi'|)^{2/3}\}(N^{-1}\max|\psi'|)^{4/3},\quad
    \delta_{12}=\delta_{22}=0,
   \end{gather*}
   and \eqref{eq:delta_coer} be satisfied.
   Let $u$ be the solution of \eqref{eq:Lu} fulfilling Assumption~\ref{ass:dec}
   and $u^N_{SD}$ be the streamline diffusion solution of \eqref{eq:SD_form}.
   Then it holds for $\sigma\geq p+1$ that
   \[
    \tnorm{u-u^N_{SD}}_{\eps}\leq C(N^{-1}\max|\psi'|)^p.
   \]
  \end{thm}
  \begin{proof}
   For the standard Shishkin mesh the proof is given in \cite[Theorem 8]{Fr11} based mainly on
   Lemma 6 therein. For the Bakhvalov S-mesh the result is stated in~\cite{Fr13_1}.
   The proof for a general S-type mesh can be done in a very similar way to \cite{Fr11}
   and one obtains
   \begin{align}\label{eq:SDFEM:stabest}
    a_{stabSD}(u-\pi^N u,\chi)
     \leq C \big[
              \delta_{11}^{1/2}&N^{-p}
             + \delta_{12}\eps^{-1}(N^{-1}\max|\psi'|)^{p-1}\\
             +&\min\{\delta_{21}^{1/2}\eps^{1/4},\delta_{21}^{3/4}\}(N^{-1}\max|\psi'|)^{p-1}\notag\\
             +&\min\{\delta_{22}\eps^{-3/4},\delta_{22}^{1/2}\eps^{-1/4}\}(\ln N)^{1/2}(N^{-1}\max|\psi'|)^{p-1}\notag
            \big]\tnorm{\chi}_{SD}
   \end{align}
   which together with the result for the Galerkin bilinear form \cite[Theorem 13]{FrM10}, 
   coercivity \eqref{eq:SDFEM:coer} and the interpolation error \cite[Theorem 12]{FrM10} gives above theorem.   
  \end{proof}

  It can be seen quite nicely, that $|a_{stabSD}(u-\pi^Nu,\chi)|$ becomes smaller, if the
  stabilisation parameters are reduced. But there is also an interaction between the Galerkin
  bilinear form $a_{Gal}(u-\pi^Nu,\chi)$ and the SDFEM norm, that can be exploited to prove 
  supercloseness. In order to do so, we will need an extension of Assumption~\ref{ass:dec} 
  on the solution decomposition.
  
  \begin{ass}\label{ass:dec_ext}
   Let the solution $u$ of~\eqref{eq:Lu} be decomposable according to Assumption~\ref{ass:dec}
   into
   \begin{gather*}
    u =v+w_1+w_2+w_{12}.
   \end{gather*}
   In addition to the pointwise bounds for $i+j\leq p+1$ stated in Assumption~\ref{ass:dec}
   we assume the $L_2$-norm bounds
   \begin{align*}
    \bignorm{\frac{\partial^{p+2}v}{\partial x^i\partial y^j}}{0}
     &\leq C,&
    \bignorm{\frac{\partial^{p+2}w_1}{\partial x^i\partial y^j}}{0}
     &\leq C\eps^{-i+1/2},\\
    \bignorm{\frac{\partial^{p+2}w_2}{\partial x^i\partial y^j}}{0}
     &\leq C\eps^{-j/2+1/4},&
    \bignorm{\frac{\partial^{p+2}w_{12}}{\partial x^i\partial y^j}}{0}
     &\leq C\eps^{-i-j/2+3/4}
   \end{align*}
   for $i+j=p+2$ with either $i=1$ or $j=1$.
  \end{ass}
  Having this additional smoothness, the integral identities by Lin,
  see~\cite{ST08, Lin91, Zhang03} can be used. Here we cite \cite[Lemma 4]{ST08}.
  \begin{lem}\label{lem:Lin}
    Let $w\in H^{p+2}(\tau_{ij})$. Then for each $\chi\in\QS_p(\tau_{ij})$ we have
    \begin{align*}
     \left|\left((\pi^Nw-w)_x,\chi_x\right)_{\tau_{ij}}\right|
      &\leq C \bignorm{k_j^{p+1}\frac{\partial^{p+2}w}{\partial x\partial y^{p+1}}}
                               {0,\tau_{ij}}
                       \norm{\chi_x}{0,\tau_{ij}}\\
     \mbox{and}\hspace*{1cm}
     \left|\left((\pi^Nw-w)_y,\chi_y\right)_{\tau_{ij}}\right|
      &\leq C \bignorm{h_i^{p+1}\frac{\partial^{p+2}w}{\partial x^{p+1}\partial y}}
                               {0,\tau_{ij}}
                       \norm{\chi_y}{0,\tau_{ij}}.\hspace*{1.5cm}~
    \end{align*}
  \end{lem}

  A different approach was used in \cite{DL06,DLP13}. Therein a method attributed
  to Zl\'{a}mal \cite{Zlamal78} is applied by adding and subtracting a certain
  higher-order polynomial and using its approximation properties. Although only done
  for bilinear finite elements, it seems plausible that a similar technique might work
  in the higher-order case.
  
  Note that identities like those given in Lemma~\ref{lem:Lin} do not hold for proper subspaces
  $\QS_p^\clubsuit\subset\QS_p$. Therefore, they cannot be used to prove a supercloseness property for
  spaces like the Serendipity space. This is not a real drawback, as for proper subspaces no
  supercloseness property is observed numerically.

  Under above assumptions, \cite{Fr11} gives a supercloseness result for the SDFEM method.

  \begin{thm}[Theorem 13 of \cite{Fr11}]\label{thm:SDFEM:superclose}
%
   For $\QS_p^{\clubsuit}=\QS_p$, $\sigma\geq p+1$
   \[
    \delta_{11}=C N^{-1},\quad
    \delta_{21}\leq C\max\{1,\eps^{-1/2}(N^{-1}\max|\psi'|)\}(N^{-1}\max|\psi'|)^{2},\quad
    \delta_{12}=\delta_{22}=0
   \]
   and \eqref{eq:delta_coer} we have
   \[
    \tnorm{\pi^N u-u^N_{SD}}_{SD}\leq C (N^{-1}\max |\psi'|)^{p+1/2}(\max|\psi'|\ln N)^{1/2}.
   \]
  \end{thm}
  \begin{proof}
   In \cite{Fr11} the proof for the standard Shishkin mesh can be found. The adaptation
   to general S-type meshes is straight-forward. The proof itself is based on the idea to
   estimate parts of the convective term of $a_{Gal}(\cdot,\cdot)$ by the SDFEM norm instead 
   of the energy norm, see \cite{ST08}. To be more precise, it's main step is
   \begin{align*}
    |(\pi^Nu-u,b\chi_x)_{\Omega_{11}}|
       &\leq C \norm{\pi^Nu-u}{0,\Omega_{11}}\norm{b\chi_x}{0,\Omega_{11}}\\
       &\leq C N^{-(p+1)}\norm{b\chi_x}{0,\Omega_{11}}\\
       &\leq C \min\{\eps^{-1/2},\delta_{11}^{-1/2}\}N^{-(p+1)}\tnorm{\chi}_{SD}
   \end{align*}
   that leads to
   \begin{multline}\label{eq:SDFEM:super:conv}
       |((\pi^Nu-u),b\chi_x)|\leq C \bigg(\min\{\eps^{-1/2},\delta_{11}^{-1/2}\}N^{-(p+1)}+\\
                                  (1+\min\{\delta_{21}^{-1/4}, N^{1/2}\})(N^{-1}\max |\psi'|)^{p+1}(\ln N)^{1/2}\bigg)
                                  \tnorm{\chi}_{SD}.
   \end{multline}
   The new bounds on the stabilisation parameters are consequences of \eqref{eq:SDFEM:stabest}.
  \end{proof}

  \begin{rem}
%
   In order to achieve the supercloseness property we have to stabilise in $\Omega_{11}$.
   In the characteristic layer region we may stabilise, but this is not necessary
   for supercloseness. By choosing $\delta_{21}=C(N^{-1}\max|\psi'|)^2$ above result 
   can be slightly improved to
   \[
    \tnorm{\pi^N u-u^N_{SD}}_{SD}\leq C (N^{-1}\max|\psi'|)^{p+1/2}(\ln N)^{1/2}.
   \]
   The bound \eqref{eq:SDFEM:super:conv} does also show, that for $\eps\geq N^{-1}$
   even the Galerkin FEM ($\delta_{11}=\delta_{21}=0$)
   fulfils a supercloseness property of order $p+1/2$.
   Unfortunately, this case is of little interest in general.
  \end{rem}
  
  We have already seen in Section~\ref{sec:results:GFEM} that the two interpolation 
  operators $\pi^N$ (vertex-edge-cell interpolation) and $I^N$ (Gauß-Lobatto interpolation)
  show a similar numerical behaviour. Recalling \eqref{eq:inter:connection2}
  \[
    \pi_p^Nv=I_p^Nv
             +(\pi^N_{p+1}v-v)
             +\left(I^N_p(\pi^N_{p+1}v-v)-(\pi^N_{p+1}v-v)\right)
  \]
  we obtain
  \[
   \enorm{I_p^Nu-u_{SD}^N}
    \leq \enorm{\pi_p^Nu-u^N_{SD}}
        +\enorm{I_p^N(\pi_{p+1}^Nu-u)-(\pi_{p+1}^Nu-u)}
        +\enorm{\pi_{p+1}^Nu-u}.
  \]
  Now the first term is estimated in Theorem \ref{thm:SDFEM:superclose}, while the other two
  terms are interpolation errors of higher-order. Combining the results gives \cite[Theorem 4.8]{Fr12}.

  \begin{thm}[Theorem 4.8 of \cite{Fr12}]\label{thm:SDFEM:superclose2}
   Let $\sigma\geq p+2$. Then it holds for the streamline-diffusion solution $u^N_{SD}$
   under the restrictions on the stabilisation parameters given in Theorem~\ref{thm:SDFEM:superclose}
   \[
    \enorm{I^N u-u^N_{SD}}\leq C(N^{-1}\max |\psi'|)^{p+1/2}(\max|\psi'|\ln N)^{1/2}.
   \]
  \end{thm}
  
  Thus, the Gauß-Lobatto interpolation inherits the supercloseness property from the vertex-edge-cell
  interpolation. A supercloseness property can be used to enhance the quality of the solution 
  by a simple postprocessing routine. 
  
  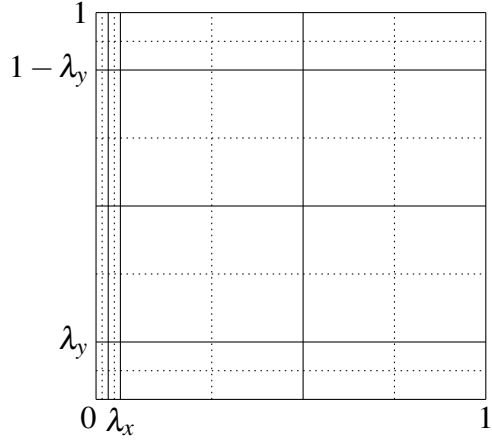
\begin{figure}
   \begin{center}
     \begin{minipage}[c]{6cm}
      \vspace*{0cm}
      \setlength{\unitlength}{0.57pt}
      \begin{picture}(256,256)
       \multiput(  4,  0)(0,4){64}{\line(0,1){1}}
       \multiput( 12,  0)(0,4){64}{\line(0,1){1}}
       \multiput( 76,  0)(0,4){64}{\line(0,1){1}}
       \multiput(196,  0)(0,4){64}{\line(0,1){1}}

       \multiput(  0, 19)(4,0){64}{\line(1,0){1}}
       \multiput(  0, 83)(4,0){64}{\line(1,0){1}}
       \multiput(  0,173)(4,0){64}{\line(1,0){1}}
       \multiput(  0,237)(4,0){64}{\line(1,0){1}}

       \put(  0,  0){\line( 0, 1){256}}
       \put(  8,  0){\line( 0, 1){256}}
       \put( 16,  0){\line( 0, 1){256}}
       \put(136,  0){\line( 0, 1){256}}
       \put(256,  0){\line( 0, 1){256}}

       \put(  0,  0){\line( 1, 0){256}}
       \put(  0, 38){\line( 1, 0){256}}
       \put(  0,128){\line( 1, 0){256}}
       \put(  0,218){\line( 1, 0){256}}
       \put(  0,256){\line( 1, 0){256}}

       \put(-5,-5){\makebox(0,0)[t]{$0$}}
       \put(16,-5){\makebox(0,0)[t]{$\lambda_x$}}
       \put(-5,38){\makebox(0,0)[r]{$\lambda_y$}}
       \put(-5,218){\makebox(0,0)[r]{$1-\lambda_y$}}
       \put(256,-5){\makebox(0,0)[t]{$1$}}
       \put(-5,256){\makebox(0,0)[r]{$1$}}

      \end{picture}
     \end{minipage}
   \end{center}
   \caption{\label{fig:post:mesh}Macroelements $M$ of $\tilde T^{N/2}$
            constructed from $T^N$}
  \end{figure}
  Suppose $N$ is divisible by 8. We construct a coarser macro mesh $\tilde T^{N/2}$
  composed of macro rectangles $M$, each consisting of four rectangles of~$T^N$.
  The construction of these macro elements $M$ is done such that the union of them covers
  $\Omega$ and none of them crosses the transition lines at $x=\lambda_x$ and at $y=\lambda_y$
  or $y=1-\lambda_y$, see Figure~\ref{fig:post:mesh}.
  Remark that in general $\tilde T^{N/2}\neq T^{N/2}$ due to different transition
  points $\lambda_x$ and $\lambda_y$, and the mesh generating function $\phi$.

  We now define local postprocessing operators for one macro element $M\in\tilde T^{N/2}$.
  The precise definition can be found in \cite{Fr12}, we will give only the basic ideas here.
  
  The first one was presented in 1d in \cite{Tob06} and is a modification of an operator given 
  in \cite{Lin91}. Let the local operator
  $\widehat{P}_{vec}:C[-1,1]\to \PS_{p+1}[-1,1]$ be given on the reference interval $[-1,1]$ by
  \begin{align*}
    \widehat{P}_{vec} \hat v(-1)&=v(x_{i-1}),\qquad
    \widehat{P}_{vec} \hat v( a) =v(x_{i}),\qquad
    \widehat{P}_{vec} \hat v( 1) =v(x_{i+1}),\\
    \mbox{and for $p=2$:}\hspace*{1cm}
    \int_{-1}^{1} (\widehat{P}_{vec}\hat v-\hat v)&=0,\\
    \mbox{while for $p\geq 3$:}\hspace*{1cm}
    \int_{-1}^{a} (\widehat{P}_{vec}\hat v-\hat v)&=0,\quad
    \int_{a}^1    (\widehat{P}_{vec}\hat v-\hat v) =0,\\
    \int_{-1}^1   (\widehat{P}_{vec}\hat v-\hat v)q&=0,\quad q\in \PS_{p-2}[-1,1]\setminus\R,
  \end{align*}
  where $\hat v$ is a function $v|_{[x_{i-1},x_{i+1}]}$ linearly mapped onto the reference interval
  and $a\in(-1,1)$ is the point that $x_i$ is mapped onto.
  By using the reference mapping and the tensor product structure, we obtain the full postprocessing operator
  $P_{vec,M}:C(M)\to \QS_{p+1}(M)$ on each macro element.
  Then, this piecewise projection is extended to a global, continuous operator $P_{vec}$.

  The second postprocessing operator is defined by using 
  the ordered sample of Gau\ss-Lobatto points $\{(\tilde x_i,\tilde y_j)\}$, $i,j=0,\dots,2p$ 
  of the four rectangles that $M$ consists of.
  
  Let $P_{GL,M}:C(M)\to \QS_{p+1}(M)$ denote the projection/interpolation operator
  fulfilling
  \[
   P_{GL,M}v(\tilde x_i,\tilde y_j)=v(\tilde x_i,\tilde y_j),\quad i,j=0,\,1,\,3,\,5,\dots,2p-3,2p-1,\,2p.
  \]
  Then, this piecewise projection is extended to a global, continuous operator $P_{GL}$.

  We have for the postprocessed numerical solutions the following superconvergence result~\cite[Theorem 5.2]{Fr12}.

  \begin{thm}[Theorem 5.2 of \cite{Fr12}]\label{thm:SDFEM:post}
   Let $\sigma\geq p+2$. Then it holds for the streamline-diffusion solution $u^N_{SD}$
   under the restrictions on the stabilisation parameters given in Theorem~\ref{thm:SDFEM:superclose}
   \[
     \enorm{u-P_{GL}u^N_{SD}}
    +\enorm{u-P_{vec}u^N_{SD}}\leq C(N^{-1}\max|\psi'|)^{p+1/2}(\max|\psi'|\ln N)^{1/2}.
   \]
  \end{thm}
  
  Let us come to the numerical example ~\eqref{eq:num_example}.
  Although Theorems~\ref{thm:SDFEM:superclose2} and~\ref{thm:SDFEM:post} assume $\sigma\geq p+2$
  we will use a Bakhvalov-S-mesh with $\sigma=p+3/2$ and $\eps=10^{-6}$ as numerical results
  suggest this to be enough. Note also, that in the bilinear case $\sigma=1+3/2$ is a standard choice
  for superconvergence simulations,~\cite{Zhang03,FrLR08,FrL08}.
  For the stabilisation parameters we have two choices, according to Theorems~\ref{thm:SDFEM:conv} 
  and \ref{thm:SDFEM:superclose}:
%
  \begin{subequations}
  \begin{align}
    \delta_{11}&=C_{SD},\quad
    \delta_{21}=C_{SD}\eps^{-1/2}N^{-2},\quad
    \delta_{12}=\delta_{22}=0,\label{eq:SDFEM:delta:set1}\\
  \mbox{or}\qquad
    \delta_{11}&=C_{SD}N^{-1},\quad
    \delta_{21}=C_{SD}\eps^{-1/2}N^{-3},\quad
    \delta_{12}=\delta_{22}=0.\label{eq:SDFEM:delta:set2}
  \end{align}
  \end{subequations}
  
  For both our investigations into convergence and superconvergence we will
  use the smaller parameters, i.e. \eqref{eq:SDFEM:delta:set2}.
  The influence of $C_{SD}$ to various norms can be seen in Figure~\ref{fig:SDFEM:C}
  using $N=64$ and $\eps=10^{-6}$.
  \begin{figure}
   \centerline{\includegraphics[width=0.78\textwidth]{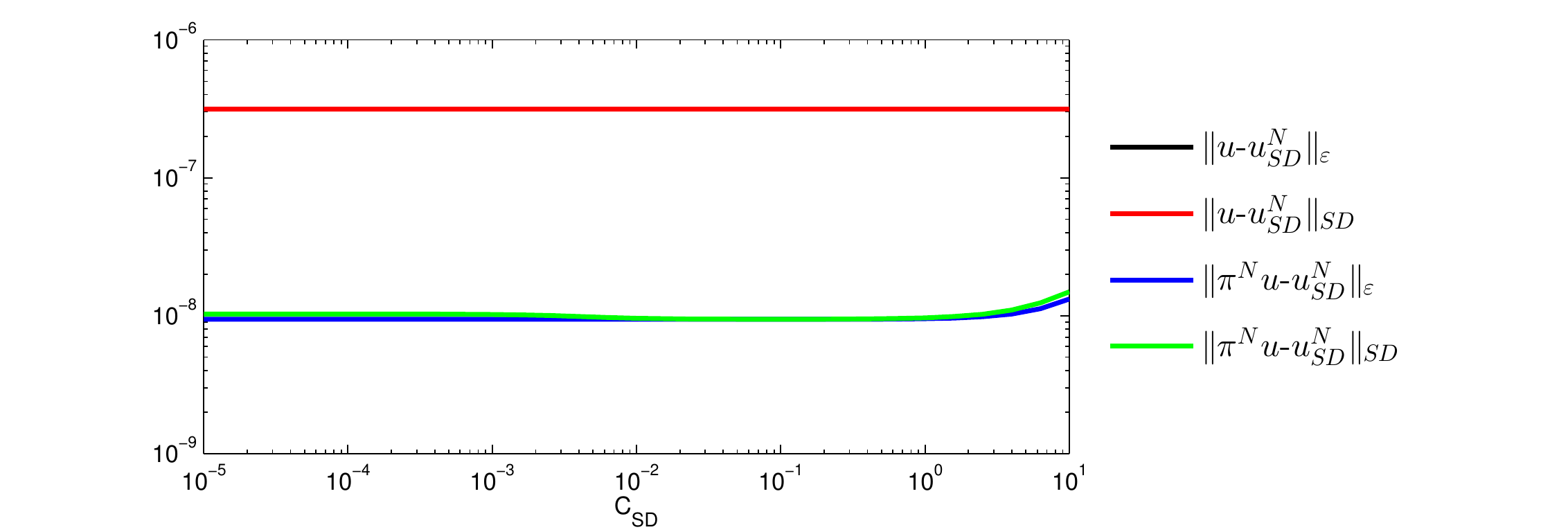}}
   \caption{Influence of the stabilisation constant $C_{SD}$ onto the error behaviour
            \label{fig:SDFEM:C}}
  \end{figure}
   Therein, the norms are not strongly influenced by the choice of moderate values of $C_{SD}$.
   Thus, in the following we will use $C_{SD}=1$.

  Table~\ref{tab:SDFEM:conv2}
  \begin{table}[tb]
   \begin{center}
    \caption{Convergence errors of SDFEM for $\QS_p$- and $\QS_p^\oplus$-elements, 
             and $p=4,\,5$ with $\delta_{ij}$ according to \eqref{eq:SDFEM:delta:set2}
             \label{tab:SDFEM:conv2}}
    \begin{tabular}{r|ll|ll|ll|ll}
     \multicolumn{1}{c}{}& \multicolumn{8}{c}{$\enorm{u-u_{SD}^N}$}\\
     \rule{0pt}{1.1em}$N$
         & \multicolumn{2}{c}{$\QS_4$}
         & \multicolumn{2}{c|}{$\QS_4^\oplus$}   
         & \multicolumn{2}{c}{$\QS_5$}
         & \multicolumn{2}{c}{$\QS_5^\oplus$}   \\[2pt]
     \hline\rule{0pt}{1.1em}
       8 &  6.709e-04 & 3.66 & 1.480e-03 & 3.69 & 2.198e-04 & 4.67 & 5.331e-04 & 4.56\\
      16 &  5.308e-05 & 3.84 & 1.150e-04 & 3.87 & 8.634e-06 & 5.34 & 2.258e-05 & 4.86\\
      32 &  3.715e-06 & 3.91 & 7.859e-06 & 3.94 & 2.134e-07 & 4.99 & 7.761e-07 & 4.93\\
      64 &  2.467e-07 & 3.96 & 5.107e-07 & 3.97 & 6.722e-09 & 4.95 & 2.554e-08 & 4.94\\
     128 &  1.590e-08 & 3.98 & 3.248e-08 & 3.99 & 2.170e-10 & 4.59 & 8.330e-10 & 0.14\\
     256 &  1.009e-09 & 3.99 & 2.046e-09 & 3.99 & 8.989e-12 &      & 7.585e-10 &     \\
     320 &  4.147e-10 &      & 8.396e-10 &      &           &      &           &     
    \end{tabular}
   \end{center}
  \end{table}
  shows the results for the polynomial spaces $\QS_p$ and $\QS_p^\oplus$ in the cases 
  $p=4$ and $p=5$. As we can see, the convergence orders of $p$ are achieved and again
  we only have a constant factor of about 2 ($p=4$) and about 3 ($p=5$) in the errors
  when switching from the full to the Serendipity space.

%
  \begin{table}[tb]
   \begin{center}
    \caption{Supercloseness property of SDFEM for $p=4$ and $\delta_{ij}$ according to \eqref{eq:SDFEM:delta:set2}
             \label{tab:SDFEM:super}}
    \begin{tabular}{r|lr|lr|lr|lr}
     \multicolumn{1}{c}{}& \multicolumn{6}{c|}{$\QS_4$}&\multicolumn{2}{c}{$\QS_4^\oplus$}\\
     \rule{0pt}{1.1em}$N$
         & \multicolumn{2}{c}{$\enorm{\pi^Nu-u_{SD}^N}$}
         & \multicolumn{2}{c}{$\enorm{I^Nu-u_{SD}^N}$}   
         & \multicolumn{2}{c|}{$\enorm{J^Nu-u_{SD}^N}$}
         & \multicolumn{2}{c}{$\enorm{\pi^Nu-u_{SD}^N}$}\\[2pt]
     \hline\rule{0pt}{1.1em}
       8 & 1.717e-04 & 4.28 & 2.004e-04 & 4.31 & 3.241e-04 & 3.78 & 6.824e-04 & 3.46\\
      16 & 8.810e-06 & 5.10 & 1.009e-05 & 5.00 & 2.358e-05 & 3.91 & 6.204e-05 & 3.69\\
      32 & 2.566e-07 & 4.95 & 3.150e-07 & 4.92 & 1.572e-06 & 3.92 & 4.798e-06 & 3.83\\
      64 & 8.302e-09 & 4.95 & 1.041e-08 & 4.94 & 1.036e-07 & 3.96 & 3.375e-07 & 3.91\\
     128 & 2.679e-10 & 4.91 & 3.385e-10 & 4.89 & 6.669e-09 & 3.98 & 2.242e-08 & 3.96\\
     256 & 8.919e-12 & 1.38 & 1.144e-11 & 2.14 & 4.232e-10 & 3.98 & 1.445e-09 & 3.97\\
     320 & 6.550e-12 &      & 7.098e-12 &      & 1.740e-10 &      & 5.959e-10 &     
    \end{tabular}
   \end{center}
  \end{table}
  As predicted by Theorems~\ref{thm:SDFEM:superclose} and \ref{thm:SDFEM:superclose2} we observe 
  in Table~\ref{tab:SDFEM:super} for the case $p=4$ a supercloseness property. 
  But, the order is $p+1$ 
  for both the vertex-edge-cell interpolation operator $\pi^N$ and 
  the Gauß-Lobatto interpolation operator $I^N$ instead of the predicted $p+1/2$. 
  Thus the analytical results may not be sharp. 
  Note that for the equidistant-interpolation operator $J^N$ and for the Serendipity space this
  property is not evident.
  
  Let us now come to exploiting the supercloseness property. Table~\ref{tab:SDFEM:post}
%
  \begin{table}[tb]
   \begin{center}
    \caption{Postprocessing of SDFEM for $\QS_4$ and $\delta_{ij}$ according to \eqref{eq:SDFEM:delta:set2}
             \label{tab:SDFEM:post}}
    \begin{tabular}{r|lr|lr}
     \rule{0pt}{1.1em}$N$
         & \multicolumn{2}{c}{$\enorm{u-P_{vec}u_{SD}^N}$}
         & \multicolumn{2}{c}{$\enorm{u-P_{GL}u_{SD}^N}$}\\[2pt]
     \hline\rule{0pt}{1.1em}
       8 & 4.725e-03 & 4.68 & 1.196e-02 & 4.80\\
      16 & 1.850e-04 & 4.92 & 4.301e-04 & 5.15\\
      32 & 6.104e-06 & 4.99 & 1.210e-05 & 5.27\\
      64 & 1.918e-07 & 5.01 & 3.145e-07 & 5.28\\
     128 & 5.961e-09 & 5.01 & 8.091e-09 & 5.24\\
     256 & 1.853e-10 & 5.00 & 2.143e-10 & 5.15\\
     320 & 6.071e-11 &      & 6.794e-11 &     
    \end{tabular}
   \end{center}
  \end{table}
  gives the results of applying the postprocessing operators $P_{vec}$ and $P_{GL}$ to 
  the SDFEM-solution. In correspondence with Theorem~\ref{thm:SDFEM:post} we observe an improved
  convergence behaviour. We see a superconvergence of order $p+1$, half an order better 
  than predicted. 
  
  Note that simulations with other polynomial degrees show similar results.
  
 \section{Results for LPSFEM}\label{sec:results:LPSFEM}
  Finally, we analyse the LPSFEM. For its application to \eqref{eq:Lu}
  with general higher-order elements we find a convergence result of order $p$
  in~\cite{FrM10_1}.
%
  \begin{thm}[Theorem 6 of \cite{FrM10_1}]\label{thm:LPSFEM:conv}
%
   Let the solution $u$ of~\eqref{eq:Lu} satisfy Assumption~\ref{ass:dec}.
   If the stabilisation parameters are chosen according to
    \begin{gather}\label{eq:LPSFEM:delta_max}
     \delta_{11}
      \leq C_{LPS} N^{-2}\big(\max|\psi'|\big)^{2p},\,
     \delta_{21}
      \leq C_{LPS} \eps^{-1/2}\ln^{-1}N \big(N^{-1}\max|\psi'|\big)^2,\,
      \delta_{12}  =\delta_{22} =0
    \end{gather}
   where $C_{LPS}>0$ is a constant and $\sigma\geq p+1$,
   we have for the LPSFEM solution $u_{LPS}^N$ of~\eqref{eq:LPSFEM_form}
   \begin{gather}\label{eq:LPSFEM:conv}
    \enorm{u-u_{LPS}^N}\leq C\big(N^{-1}\max|\psi'|\big)^p.
   \end{gather}
  \end{thm}
  
  Thus, if the stabilisation parameters are not too large then the convergence order $p$ of
  the Galerkin FEM is not disturbed. Similarly to the Galerkin FEM, no supercloseness
  result is known in the higher-order case. 
  A supercloseness property of order two was shown for bilinear elements in \cite{FrM10}.
  
  When analysing the SDFEM we proved superconvergence by bounding the convective term of 
  the Galerkin bilinear form against terms in the SDFEM norm.
  Unfortunately this trick does not help here with the LPSFEM. 
  Basically, there are two problems. First, the convective term cannot easily be bounded
  by the stabilisation term, as the stabilisation terms only include fluctuations of the convection.
  Here the idea of~\cite{Knob10} may help and we may use a stronger LPS-SDFEM norm, where the full 
  weighted streamline derivative is included. But then the second problem comes into play. 
  In order to estimate with the streamline derivative part of the norm we have to borrow 
  half an order of the stabilisation parameter $\delta_{11}$, cf. \eqref{eq:SDFEM:super:conv}.
  This costs us $\delta_{11}^{-1/2}\geq N/(\max|\psi'|)^p$ by \eqref{eq:LPSFEM:delta_max}. 
  Thus there would be no benefit in estimating with the stronger LPS norm.

  Let us now look at the numerical example ~\eqref{eq:num_example}. Again we will use a 
  Bakhvalov-S-mesh with $\sigma=p+3/2$ and $\eps=10^{-6}$.
  The stabilisation parameters are chosen according to Theorem~\ref{thm:LPSFEM:conv}.
  The influence of $C_{LPS}$ to various norms can be seen in Figure~\ref{fig:LPSFEM:C}
  using $N=64$ and $\eps=10^{-6}$.
  \begin{figure}[tbp]
   \centerline{\includegraphics[width=0.78\textwidth]{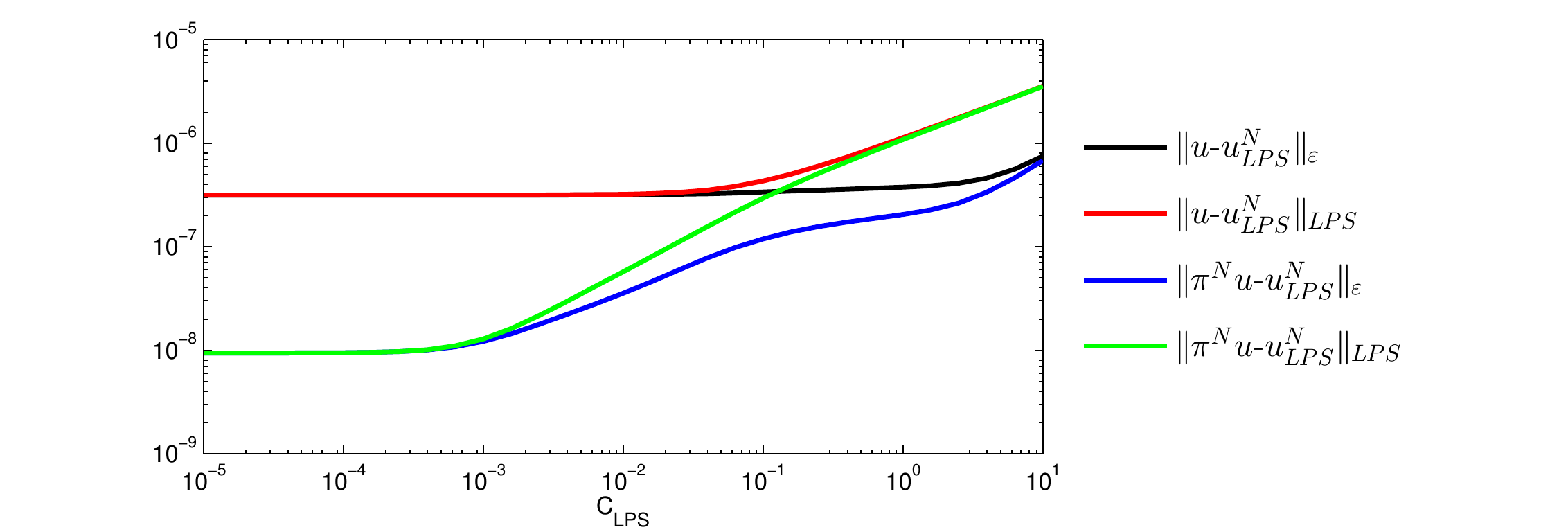}}
   \caption{Influence of the stabilisation constant $C_{LPS}$ onto the error behaviour
            \label{fig:LPSFEM:C}}
  \end{figure}
  Clearly, for larger values of $C_{LPS}$ more stabilisation is introduced. On the downside,
  if $C_{LPS}$ is too large the stabilisation term dominates the weak formulation of 
  \eqref{eq:Lu} unless $N$ is large enough. Therefore we have chosen for the following
  simulations $C_{LPS}=0.001$, i.e.
%
  \begin{align}\label{eq:LPSFEM:delta:set}
     \delta_{11}
      \leq 0.001 N^{-2}\big(\max|\psi'|\big)^{2p},\,
     \delta_{21}
      \leq 0.001 \eps^{-1/2}\ln^{-1}N \big(N^{-1}\max|\psi'|\big)^2,\,
      \delta_{12} =
      \delta_{22} =0.
  \end{align}

  Table~\ref{tab:LPSFEM:conv}
  \begin{table}[tbp]
   \begin{center}
    \caption{Convergence errors of LPSFEM for $\QS_p$- and $\QS_p^\oplus$-elements, 
             and $p=4,\,5$ with $\delta_{ij}$ according to \eqref{eq:LPSFEM:delta:set}
             \label{tab:LPSFEM:conv}}
    \begin{tabular}{r|ll|ll|ll|ll}
     \multicolumn{1}{c}{}& \multicolumn{8}{c}{$\enorm{u-u_{LPS}^N}$}\\
     \rule{0pt}{1.1em}$N$
         & \multicolumn{2}{c}{$\QS_4$}
         & \multicolumn{2}{c|}{$\QS_4^\oplus$}   
         & \multicolumn{2}{c}{$\QS_5$}
         & \multicolumn{2}{c}{$\QS_5^\oplus$}   \\[2pt]
     \hline\rule{0pt}{1.1em}
       8 & 8.348e-04 & 3.64 & 1.892e-03 & 3.64  & 1.088e-04 & 4.61 & 3.715e-04 & 4.57 \\
      16 & 6.716e-05 & 3.82 & 1.518e-04 & 3.86  & 4.457e-06 & 4.81 & 1.560e-05 & 4.82 \\
      32 & 4.752e-06 & 3.91 & 1.046e-05 & 3.94  & 1.584e-07 & 4.90 & 5.532e-07 & 4.92 \\
      64 & 3.160e-07 & 3.96 & 6.796e-07 & 3.98  & 5.291e-09 & 4.95 & 1.828e-08 & 4.94 \\
     128 & 2.037e-08 & 3.98 & 4.316e-08 & 3.99  & 1.713e-10 & 4.42 & 5.970e-10 & -0.29\\
     256 & 1.293e-09 & 3.99 & 2.716e-09 & 3.99  & 8.005e-12 &      & 7.319e-10 &      \\
     320 & 5.313e-10 &      & 1.114e-09 &       &           &      &           & 
    \end{tabular}
   \end{center}
  \end{table}
  shows the convergence behaviour of the LPSFEM for the same polynomial spaces as Table~\ref{tab:SDFEM:conv2}.
  Again we see convergence of order $p$ for the full and the Serendipity spaces.
  Although the Serendipity spaces need only half the number of degrees of freedom,
  and are therefore much cheaper in computation, only a factor of 2-4 lies between
  the errors of the full space and those of the Serendipity space.  

  Numerically, the LPSFEM does possess a supercloseness property too. Table~\ref{tab:LPSFEM:super}
  \begin{table}[tbp]
   \begin{center}
    \caption{Supercloseness property of LPSFEM for $p=4$ and $\delta_{ij}$ according to \eqref{eq:LPSFEM:delta:set}
             \label{tab:LPSFEM:super}}
    \begin{tabular}{r|lr|lr|lr|lr}
     \multicolumn{1}{c}{}& \multicolumn{6}{c|}{$\QS_4$}&\multicolumn{2}{c}{$\QS_4^\oplus$}\\
     \rule{0pt}{1.1em}$N$
         & \multicolumn{2}{c}{$\enorm{\pi^Nu-u_{LPS}^N}$}
         & \multicolumn{2}{c}{$\enorm{I^Nu-u_{LPS}^N}$}   
         & \multicolumn{2}{c|}{$\enorm{J^Nu-u_{LPS}^N}$}
         & \multicolumn{2}{c}{$\enorm{\pi^Nu-u_{LPS}^N}$}\\[2pt]
     \hline\rule{0pt}{1.1em}
       8 & 1.794e-04 & 4.50 & 2.167e-04 & 4.48 & 3.868e-04 & 3.74 & 8.717e-04 & 3.42\\
      16 & 7.909e-06 & 4.66 & 9.740e-06 & 4.68 & 2.885e-05 & 3.85 & 8.156e-05 & 3.67\\
      32 & 3.127e-07 & 4.68 & 3.790e-07 & 4.74 & 2.007e-06 & 3.92 & 6.395e-06 & 3.82\\
      64 & 1.216e-08 & 4.84 & 1.419e-08 & 4.86 & 1.328e-07 & 3.96 & 4.524e-07 & 3.91\\
     128 & 4.248e-10 & 4.99 & 4.871e-10 & 4.96 & 8.547e-09 & 3.98 & 3.017e-08 & 3.95\\
     256 & 1.339e-11 & 3.27 & 1.566e-11 & 3.52 & 5.422e-10 & 3.99 & 1.949e-09 & 3.97\\
     320 & 6.456e-12 &      & 7.133e-12 &      & 2.228e-10 &      & 8.039e-10 &     
    \end{tabular}
   \end{center}
  \end{table}
  shows it for the standard choice of the stabilisation parameters~\eqref{eq:LPSFEM:delta:set}.
  Here for $p=4$ the vertex-edge-cell interpolation operator $\pi^N$ and the Gauß-Lobatto
  interpolation operator $I^N$ show for the full space $\QS_p$ a supercloseness property of
  order $p+1$. So far, there is no theoretical explanation known for this fact.
  Similarly to the SDFEM and the Galerkin method, the equidistant interpolation operator
  $J^N$ and the Serendipity space do not possess such a property.

%% file: norm.tex
\chapter{What is the right norm?}\label{cha:norm}
  This chapter contains results from \cite{FrR11_2} that are also given in Appendix \ref{app:balanced}.
  Here we consider only bilinear elements, i.e.
  \[
    V^N:=\Big\{v\in H_0^1(\Omega):v|_\tau\in \QS_1(\tau)\;
    \forall\tau\in T^N\Big\}
  \]
  and restrict ourselves to the standard Shishkin mesh. See Remark~\ref{rem:norm:general}
  for ideas about the general case.

  As assumed in Assumption \ref{ass:dec}, the solution $u$ of \eqref{eq:Lu} has 
  an exponential outflow layer of the type $e^{-x/\eps}$ and
  a characteristic layer of the type $e^{-y/\sqrt\eps}$.
  The energy norms of these two components are
  \[
   \enorm{e^{-x/\eps}}=\ord{1}
   \quad\mbox{and}\quad
   \enorm{e^{-y/\sqrt\eps}}=\ord{\eps^{1/4}}.
  \]
  Thus, the last one, characterising the characteristic layer, is not
  well represented in the energy norm and
  is dominated by the exponential layer for small $\eps$.

%
  In the following we will present results in the balanced norm
  \begin{gather}\label{def:balnorm}
   \tnorm{v}_b:=\sqrt{\eps\norm{v_x}{0}^2+\eps^{1/2}\norm{v_y}{0}^2+\gamma\norm{v}{0}^2}.
  \end{gather}
  Now it holds
  \[
   \tnorm{e^{-x/\eps}}_b=\ord{1}
   \quad\mbox{and}\quad
   \tnorm{e^{-y/\sqrt\eps}}_b=\ord{1}
  \]
  and therefore
  both layer components are equally well represented in this norm.

  One possible application of balanced norms are uniform $L_\infty$-bounds
  of the error using a supercloseness result in a balanced norm, see 
  \cite[p. 399]{RST08}. Therein the concept is shown for a convection-diffusion problem
  with exponential layers only where the standard energy norm suffices.
  
  Considering reaction-diffusion problems, the standard energy norm is not well 
  balanced either. Here, first results in a balanced norm were obtained 
  in~\cite{LinStynes11} for a mixed finite element formulation and in~\cite{RSch11} 
  for a standard Galerkin approach. 

\section{A Streamline Diffusion Method}
  
  We will prove estimates in the balanced norm for a modified streamline diffusion method.
  Let us define the stabilisation bilinear form
  \begin{align*}
    a_{stab}(v,w)&:= \sum_{\tau\in T^N}(\eps\laplace v+b v_x-c v,\delta_\tau b w_x)_\tau,
                    && \!\!\!v\in H^1_0(\Omega)\cap H^2(\Omega),w\in H^1_0(\Omega),
  \end{align*}
  and the linear form
  \[
   f_{modSD}(v):=(f,v)-\sum_{\tau\in T^N}(f,\delta_\tau bv_x)_\tau,
                    \qquad v\in H^1_0(\Omega).
  \]
  Following the suggestion of \cite{CX08} we choose $\delta_\tau$
  as a stabilisation function on $\tau$ given by
  \[
   \delta|_{\tau_{ij}}
   :=\delta_{\tau_{ij}}
   :=\min\left\{\frac{h_i}{2\eps},\frac{1}{\norm{b}{\infty,\tau_{ij}}}\right\}
      h_i\frac{(x_i-x)(x-x_{i-1})}{h_i^2}.
  \]
  Thus $\delta_{\tau}$ is a quadratic bubble function in $x$-direction.
  This enables us to apply integration by parts in $x$ to some terms in our analysis
  without additional inner-boundary terms. Numerically, we see no difference to the
  standard SDFEM-formulation of Section~\ref{ssec:SDFEM} with constant $\delta_\tau$. 
  Note, that by definition it holds
  \begin{align}
    \norm{\delta}{L_\infty(\Omega_{12}\cup\Omega_{22})}
        &\leq C\eps(N^{-1}\ln N)^2,\label{eq:delta:estimate:1}\\
    \norm{\delta}{L_\infty(\Omega_{11}\cup\Omega_{21})}
        &\leq C N^{-1}\quad\mbox{and}\quad
    \norm{\eps\delta}{L_\infty(\Omega_{11}\cup\Omega_{21})}\leq C N^{-2}.\label{eq:delta:estimate:2}
  \end{align}

  We obtain the modified SDFEM formulation of \eqref{eq:Lu}:\bigskip

  Find $u_{modSD}^N\in V^N$ such that
  \begin{equation}\label{eq:modSD_form}
    a_{modSD}(u_{modSD}^N, v^N):=a_{Gal}(u_{modSD}^N,v^N)+a_{stab}(u_{modSD}^N,v^N)
     = f_{modSD}(v^N),\quad \forall v^N\in V^N.
  \end{equation}
  Associated with this method is the modified streamline diffusion norm, defined by
  \begin{gather}\label{def:modsdnorm}
   \tnorm{v}_{modSD}:=\left(\eps\norm{\grad v}{0}^2+
                            \gamma\norm{v}{0}^2+
                            \sum_{\tau\in T^N}\norm{\delta_\tau^{1/2}bv_x}{0,\tau}^2\right)^{1/2}.
  \end{gather}
  Under similar conditions on $\delta_\tau$ as in \eqref{eq:delta_coer}
  we have coercivity in this norm:
  \begin{gather}\label{eq:modSD:coer}
    a_{modSD}(v^N,v^N)\geq\frac{1}{2}\tnorm{v^N}_{modSD},\qquad \forall v^N\in V^N.
  \end{gather}
  Let us now come to the error analysis in the balanced norm.
  Although the modified SDFEM is coercive w.r.t. the modified SDFEM-norm, it is not 
  uniformly coercive w.r.t. the balanced norm. 
  Therefore, we use an additional projection to prove the error estimates.
  
  Let a projection operator $\pi:H^1(\Omega)\cap C(\Omega)\to V^N$ be given by
  \begin{gather}\label{eq:def_projection}
   a_{proj}(\pi u-u,\chi)=0\quad\mbox{for all }\chi\in V^N
  \end{gather}
  where
  \[
    a_{proj}(v,w)=\eps(v_x,w_x) + (c v-b v_x, w)+\sum_{\tau\in T^N}(\eps v_{xx}+b v_x-c v,\delta_\tau b w_x)_\tau.
  \]
  The operator is defined in such a way, that for all $\chi\in V^N$ it holds
  \begin{gather}\label{eq:apply_piu}
   a_{modSD}(\pi u-u,\chi)
    = \eps((\pi u-u)_y,\chi_y)+\sum_{\tau\in T^N}(\eps (\pi u-u)_{yy},\delta_\tau b \chi_x)_\tau.
  \end{gather}
  Combining coercivity \eqref{eq:modSD:coer}, Galerkin orthogonality and
  \eqref{eq:apply_piu} gives
  \begin{multline*}
   \frac{1}{2}\tnorm{\pi u-u_{modSD}^N}^2_{modSD}
    \leq a_{modSD}(\pi u-u,\pi u-u_{modSD}^N)\\
    =    \eps((\pi u-u)_y,(\pi u-u_{modSD}^N)_y)
        +\sum_{\tau\in T^N}(\eps (\pi u-u)_{yy},\delta_\tau b (\pi u-u_{modSD}^N)_x)_\tau.
  \end{multline*}
  By omitting terms on the left-hand side, bounding the scalar product on the right-hand side
  by its $L_2$-norms, multiplying by $\eps^{-1/2}$ and setting $\chi:=\pi u-u_{modSD}^N\in V^N$ we obtain
  \begin{multline}\label{eq:optimal}
   \frac{1}{2}\norm{(\pi u-u_{modSD}^N)_y}{0}\tnorm{\chi}_{modSD}
        \leq\\
        \norm{(u-\pi u)_y}{0}\tnorm{\chi}_{modSD}+
        \eps^{-1/2}\left|\sum_{\tau\in T^N}(\eps (\pi u-u)_{yy},\delta_\tau b \chi_x)_\tau\right|.
  \end{multline}
  
  The goal is to bound the right-hand side of \eqref{eq:optimal} by 
  $\eps^{-1/4}$ times $\tnorm{\chi}_{modSD}$ and
  a term of order $N^{-1}$. This can be done, as shown in \cite{FrR11_2}
  with one main ingredient being the $L_\infty$-stability of $\pi$.
  
  \begin{thm}[Theorem 1 of \cite{FrR11_2}]\label{thm:modSDFEM}
   Let $\sigma\geq 2$, $\eps\leq C(\ln N)^{-2}$, $u_{modSD}^N$ be the discrete solution of \eqref{eq:modSD_form}
   and $u$ the weak solution of \eqref{eq:Lu}.
   Then it holds
   \[
     \tnorm{u-u_{modSD}^N}_{b}\leq C N^{-1}(\ln N)^{3/2}.
   \]
  \end{thm}

  \begin{rem}\label{rem:norm:general}
    The result of Theorem \ref{thm:modSDFEM} can in theory be generalised in the following way
    for S-type meshes and higher-order polynomials.
    Let a consistent numerical method be given by: 
    Find $\tilde u^N\in V^N=\{v\in H_0^1(\Omega):v|_\tau\in\QS_p(\tau),\,\tau\in T^N\}$ 
    with
    \[
     a_{Gal}(\tilde u^N,v^N)+a_{stab}(\tilde u^N,v^N)=f(v^N)+f_{stab}(v^N)\quad\mbox{ for }v^N\in V^N
    \]
    where $a_{stab}(\cdot,\cdot)$ is a bilinear form and $f_{stab}(\cdot)$ is a 
    linear form. 
    Suppose $a_{Gal}(\cdot,\cdot)+a_{stab}(\cdot,\cdot)$
    is coercive w.r.t. a norm $\tnorm{\cdot}$ that contains the energy norm.
    
    Define the projection $\pi u\in V^N$ by
    \[
      a_{proj}(\pi u,\chi)=a_{proj}(u,\chi)\quad\mbox{for all }\chi\in V^N
    \]
    where
    \[
      a_{proj}(u,v)=a_{Gal}(u,v)+a_{stab}(u,v)-\eps(u_y,v_y)-a_{rest}(u,v)
    \]
    for some suitable bilinear form $a_{rest}(\cdot,\cdot)$.
    Note that for our modified SDFEM we have 
    \[
        a_{rest}(u,v)=\sum_{\tau\in T^N}(\eps u_{yy},\delta_\tau b v_x)_\tau.
    \]
    In the general setting we obtain instead of \eqref{eq:optimal} the estimate
    \[
      \norm{(\pi u-\tilde u^N)_y}{0}\tnorm{\chi}
          \leq
          C\left(\norm{(u-\pi u)_y}{0}\tnorm{\chi}+
          \eps^{-1/2}|a_{rest}(\pi u-u,\chi)|\right).
    \]
    If we had the convergence result
    \[
     \enorm{u-\tilde u^N}\leq C (N^{-1}\max |\psi'|)^p,
    \]
    the stability result
    \[
     \norm{\pi u}{L_\infty}\leq C\norm{u}{L_\infty}
    \]
    and the estimate
    \[
     |a_{rest}(u-\pi u,\chi)|\leq C \eps^{1/4}(N^{-1}\max |\psi'|)^p(\ln N)^{1/2}\tnorm{\chi},
    \]
    then it would follow
    \[
     \tnorm{u-\tilde u^N}_{b}\leq C (N^{-1}\max |\psi'|)^p(\ln N)^{1/2},
    \]
    thus convergence of order $p$ in the balanced norm.
    
    While the adaptation of the proof for our modified SDFEM to S-type meshes is 
    straight-forward, higher-order polynomials are more problematic. To our knowledge,
    no result generalising the stability given in \cite{CX08} for linear elements
    to higher-order elements is available in literature.
    
    Setting $\delta_\tau\equiv 0$ everywhere gives the unstabilised Galerkin method. 
    Unfortunately, the corresponding projection $\pi$ is not known to be $L_\infty$-stable. 
    Thus, our method of proof does not help with the pure Galerkin method.
  \end{rem}

  \section{Numerical Results}\label{sec:numerics:modSDFEM}
  We use the test problem \eqref{eq:num_example} from Chapter~\ref{cha:results}, i.e.
  \[
   -\eps \laplace u - (2-x) u_x + \frac{3}{2} u = f
  \]
  with homogeneous Dirichlet boundary conditions and the right-hand side $f$ chosen such that
  \[
    u = \left(\cos(\pi x/2)-\frac{e^{-x/\eps}-e^{-1/\eps}}{1-e^{-1/\eps}}\right)
        \frac{(1-e^{-y/\sqrt\eps})(1-e^{-(1-y)/\sqrt\eps})}{1-e^{-1/\sqrt\eps}}
  \]
  is the exact solution. 
  
  In the following, 'order' will always denote the exponent $\alpha$ in a
  convergence order of form $\mathcal{O}(N^{-\alpha})$ while 'ln-order'
  corresponds to the exponent $\alpha$ in a convergence order given by
  $\mathcal{O}\big((N^{-1}\ln N)^\alpha\big)$. It is computed as usual using
  two consecutive numerical solutions.
  The experiments are carried out with
  $\sigma=5/2$ and all integrations are approximated by a Gauss-Legendre
  quadrature of $6\times 6$-points.

  In our first experiment we look into the $\eps$-uniformity of our calculations.
  Table~\ref{tab:uniform}%
  \begin{table}[tbp]
   \begin{center}
    \caption{$\eps$-uniformity of modSDFEM-errors for $N=64$\label{tab:uniform}}
    \begin{tabular}{r|c|c}
     \rule{0pt}{1.1em}$\eps$
         & $\tnorm{u-u^N_{modSD}}_b$
         & $\enorm{u-u^N_{modSD}}$   \\[2pt]
     \hline\rule{0pt}{1.1em}
     1.0e-01 & 1.932e-02 & 1.747e-02\\
     1.0e-02 & 6.356e-02 & 5.614e-02\\
     1.0e-03 & 1.181e-01 & 6.531e-02\\
     1.0e-04 & 1.462e-01 & 6.635e-02\\
     1.0e-05 & 1.465e-01 & 6.609e-02\\
     1.0e-06 & 1.466e-01 & 6.601e-02\\
     1.0e-07 & 1.466e-01 & 6.599e-02\\
     1.0e-08 & 1.466e-01 & 6.598e-02
    \end{tabular}
   \end{center}
  \end{table}
  shows the results of the modified SDFEM for fixed $N=64$ and varying 
  values of $\eps=10^{-1},\dots,10^{-8}$.
  In both norms we can clearly see $\eps$-uniformity,
  confirming Theorem~\ref{thm:modSDFEM}.
  Note that the errors measured in the balanced norm are larger than
  those measured in the energy norm, but still bounded for decreasing $\eps$.

  In the following we will always use the fixed value $\eps=10^{-6}$ that is 
  small enough to bring out the layer behaviour of the solution $u$ of \eqref{eq:Lu}.

  Table~\ref{tab:modSDFEM}%
  \begin{table}[tbp]
   \begin{center}
    \caption{Errors of the modSDFEM in the balanced and energy norm\label{tab:modSDFEM}}
    \begin{tabular}{r|crr|crr}
     \rule{0pt}{1.1em}$N$ & $\tnorm{u-u^N_{modSD}}_b$ & order & ln-order
         & $\enorm{u-u^N_{modSD}}$   & order & ln-order\\[2pt]
     \hline\rule{0pt}{1.1em}
       8 & 4.893e-01 & 0.43 & 0.74 & 2.464e-01 & 0.53 & 0.90 \\
      16 & 3.624e-01 & 0.60 & 0.88 & 1.707e-01 & 0.65 & 0.95 \\
      32 & 2.392e-01 & 0.71 & 0.96 & 1.090e-01 & 0.72 & 0.98 \\
      64 & 1.466e-01 & 0.77 & 0.99 & 6.601e-02 & 0.77 & 0.99 \\
     128 & 8.615e-02 & 0.80 & 1.00 & 3.864e-02 & 0.81 & 1.00 \\
     256 & 4.935e-02 & 0.83 & 1.00 & 2.211e-02 & 0.83 & 1.00 \\
     512 & 2.779e-02 & 0.85 & 1.00 & 1.244e-02 & 0.85 & 1.00 \\
    1024 & 1.544e-02 &      &      & 6.912e-03 &      &
    \end{tabular}
   \end{center}
  \end{table}
  shows the errors of the modified SDFEM in the given numerical example when $N$ is varied.
  Clearly we have convergence of almost order one in the balanced
  and the standard energy norm. Whether the exponent of the logarithmic factor is
  $1$ or $3/2$ cannot be decided from this experiment, as the numerical behaviour of 
  the two functions $N^{-1}\ln N$ and $N^{-1}(\ln N)^{3/2}$ is almost the same.
  Nevertheless, this table corresponds well with Theorem~\ref{thm:modSDFEM}.

  Table~\ref{tab:GFEM:bal}%
  \begin{table}[tbp]
   \begin{center}
    \caption{Errors of the Galerkin FEM in the balanced and energy norm\label{tab:GFEM:bal}}
    \begin{tabular}{r|crr|crr}
     \rule{0pt}{1.1em}$N$ & $\tnorm{u-u^N}_b$ & order & ln-order
         & $\enorm{u-u^N}$   & order & ln-order\\[2pt]
     \hline\rule{0pt}{1.1em}
       8 & 5.025e-01 & 0.45 & 0.78 & 2.686e-01 & 0.60 & 1.02 \\
      16 & 3.667e-01 & 0.61 & 0.90 & 1.778e-01 & 0.68 & 1.01 \\
      32 & 2.404e-01 & 0.71 & 0.96 & 1.108e-01 & 0.74 & 1.00 \\
      64 & 1.469e-01 & 0.77 & 0.99 & 6.640e-02 & 0.78 & 1.00 \\
     128 & 8.623e-02 & 0.80 & 1.00 & 3.872e-02 & 0.81 & 1.00 \\
     256 & 4.937e-02 & 0.83 & 1.00 & 2.212e-02 & 0.83 & 1.00 \\
     512 & 2.779e-02 & 0.85 & 1.00 & 1.244e-02 & 0.85 & 1.00 \\
    1024 & 1.544e-02 &      &      & 6.912e-03 &      &
    \end{tabular}
   \end{center}
  \end{table}
  shows the results of standard Galerkin FEM applied to our numerical example.
  Although we could not prove convergence for the Galerkin FEM in the balanced norm,
  we see convergence of almost order one in both norms.

  Let $u^I$ denote the standard bilinear interpolant of $u$.
  Tables~\ref{tab:modSDFEM:superclose}
  \begin{table}[tbp]
   \begin{center}
    \caption{Supercloseness errors of the modSDFEM in the balanced and energy norm\label{tab:modSDFEM:superclose}}
    \begin{tabular}{r|crr|crr}
     \rule{0pt}{1.1em}$N$ & $\tnorm{u^N_{modSD}-u^I}_b$ & order & ln-order
         & $\enorm{u^N_{modSD}-u^I}$   & order & ln-order\\[2pt]
     \hline\rule{0pt}{1.1em}
       8 & 1.097e-01 & 0.48 & 0.82 & 2.307e-02 & 2.18 & 3.73 \\
      16 & 7.881e-02 & 0.96 & 1.42 & 5.082e-03 & 1.30 & 1.92 \\
      32 & 4.044e-02 & 1.30 & 1.77 & 2.065e-03 & 1.41 & 1.91 \\
      64 & 1.641e-02 & 1.52 & 1.96 & 7.775e-04 & 1.54 & 1.98 \\
     128 & 5.705e-03 & 1.67 & 2.07 & 2.671e-04 & 1.65 & 2.04 \\
     256 & 1.794e-03 & 1.80 & 2.17 & 8.540e-05 & 1.73 & 2.08 \\
     512 & 5.162e-04 & 1.97 & 2.32 & 2.580e-05 & 1.77 & 2.09 \\
    1024 & 1.317e-04 &      &      & 7.559e-06 &      &
    \end{tabular}
   \end{center}
  \end{table}
  and \ref{tab:GFEM:bal:superclose}
  \begin{table}[tbp]
   \begin{center}
    \caption{Supercloseness errors of the Galerkin FEM in the balanced and energy norm\label{tab:GFEM:bal:superclose}}
    \begin{tabular}{r|crr|crr}
     \rule{0pt}{1.1em}$N$ & $\tnorm{u^N-u^I}_b$ & order & ln-order
         & $\enorm{u^N-u^I}$   & order & ln-order\\[2pt]
     \hline\rule{0pt}{1.1em}
       8 & 1.601e-01 & 0.73 & 1.24 & 1.107e-01 & 1.14 & 1.96 \\
      16 & 9.666e-02 & 1.04 & 1.53 & 5.010e-02 & 1.33 & 1.96 \\
      32 & 4.704e-02 & 1.31 & 1.77 & 1.997e-02 & 1.46 & 1.98 \\
      64 & 1.900e-02 & 1.49 & 1.91 & 7.252e-03 & 1.55 & 2.00 \\
     128 & 6.770e-03 & 1.59 & 1.97 & 2.473e-03 & 1.61 & 2.00 \\
     256 & 2.246e-03 & 1.65 & 1.99 & 8.075e-04 & 1.66 & 2.00 \\
     512 & 7.142e-04 & 1.69 & 2.00 & 2.552e-04 & 1.70 & 2.00 \\
    1024 & 2.207e-04 &      &      & 7.863e-05 &      &
    \end{tabular}
   \end{center}
  \end{table}
  show convergence of $u^N_{modSD}-u^I$ and $u^N-u^I$
  in both norms to be of almost second order.
  Thus, we have supercloseness and via a simple postprocessing, e.g.
  biquadratic interpolation on a macro mesh, a numerical solution that is almost
  second order superconvergent can be constructed, see e.g. \cite{ST03}.
  
  For this purpose assume $N$ to be divisible by 8. We construct 
  a macro mesh of the original mesh by fusing 2-by-2 elements such that the
  macro elements are pairwise disjoint and do not cross the boundaries of the
  subdomains $\Omega_{ij}$, $i,j=1,2$, see also Figure~\ref{fig:post:mesh}.
  Tables \ref{tab:modSDFEM:post}
  \begin{table}[tbp]
   \begin{center}
    \caption{Superconvergence errors of the modSDFEM in the balanced and energy norm\label{tab:modSDFEM:post}}
    \begin{tabular}{r|crr|crr}
     \rule{0pt}{1.1em}$N$ & $\tnorm{u-Pu^N_{modSD}}_b$ & order & ln-order
                          & $\enorm{u-Pu^N_{modSD}}$   & order & ln-order\\[2pt]
     \hline\rule{0pt}{1.1em}
       8 & 1.298e-01 & 0.94 & 1.60 & 3.317e-01 & 0.64 & 1.09 \\
      16 & 6.773e-02 & 1.22 & 1.80 & 2.132e-01 & 0.97 & 1.43 \\
      32 & 2.911e-02 & 1.41 & 1.91 & 1.087e-01 & 1.27 & 1.73 \\
      64 & 1.095e-02 & 1.53 & 1.97 & 4.502e-02 & 1.48 & 1.90 \\
     128 & 3.787e-03 & 1.61 & 1.99 & 1.617e-02 & 1.59 & 1.98 \\
     256 & 1.244e-03 & 1.66 & 2.00 & 5.353e-03 & 1.66 & 2.01 \\
     512 & 3.942e-04 & 1.70 & 2.00 & 1.688e-03 & 1.71 & 2.02 \\
    1024 & 1.217e-04 &      &      & 5.145e-04 &      &     
    \end{tabular}
   \end{center}
  \end{table}
  and \ref{tab:GFEM:bal:post}
  \begin{table}[tbp]
   \begin{center}
    \caption{Superconvergence errors of the Galerkin FEM in the balanced and energy norm\label{tab:GFEM:bal:post}}
    \begin{tabular}{r|crr|crr}
     \rule{0pt}{1.1em}$N$ & $\tnorm{u-Pu^N}_b$ & order & ln-order
                          & $\enorm{u-Pu^N}$   & order & ln-order\\[2pt]
     \hline\rule{0pt}{1.1em}
       8 & 1.740e-01 & 1.02 & 1.74 & 3.549e-01 & 0.68 & 1.16\\
      16 & 8.602e-02 & 1.26 & 1.87 & 2.217e-01 & 0.99 & 1.46\\
      32 & 3.580e-02 & 1.44 & 1.95 & 1.118e-01 & 1.28 & 1.73\\
      64 & 1.321e-02 & 1.54 & 1.99 & 4.615e-02 & 1.48 & 1.90\\
     128 & 4.529e-03 & 1.61 & 2.00 & 1.660e-02 & 1.59 & 1.97\\
     256 & 1.482e-03 & 1.66 & 2.00 & 5.526e-03 & 1.65 & 1.99\\
     512 & 4.691e-04 & 1.70 & 2.00 & 1.760e-03 & 1.69 & 2.00\\
    1024 & 1.447e-04 &      &      & 5.441e-04 &      &     
    \end{tabular}
   \end{center}
  \end{table}
  show the resulting errors after applying a biquadratic interpolation $P$ to
  the discrete solutions on a macro mesh. It can be seen quite clearly,
  that $u-Pu^N_{modSD}$ and $u-Pu^N$ achieve (almost) second order
  convergence for both methods and in both norms.

%% file: green.tex
\chapter{Green's Function Estimates}\label{cha:green}
Another norm that ``sees'' all features of the solution is the $L_\infty$-norm.
In this chapter we want to look into pointwise \emph{a-posteriori} error estimation.
A-priori error estimation in the $L_\infty$-norm for convection-diffusion problems
is still an open field of research.
Some results for stabilised methods can be found in e.g.~\cite[p. 399]{RST08} or
\cite[Theorem 9.1]{Linss10} for an upwind finite difference method.

This chapter contains results from \cite{FrK12_1, FrK10_1} that are also
given in Appendix \ref{app:green} and \ref{app:green_sharp}.

\section{$L_1$-Norm Estimates of the Green's Function}
   Let us rewrite problem \eqref{eq:Lu} in a slightly different form:
   \begin{subequations}\label{eq:Lu'}
   \begin{align}
   \label{eq:Lu_a}
      L_{xy}u(x,y):=-\eps(u_{xx}+u_{yy})-(b(x,y)\,u)_x+c(x,y)\,u&=f(x,y)\quad
      \mbox{for }(x,y)\in\Omega,\\
      u(x,y)&=0\qquad\quad\;\,\mbox{for }(x,y)\in\partial\Omega
   \end{align}
   \end{subequations}
   where the coefficients $b$ and $c$ are sufficiently smooth (e.g., $b,\,c\in C^\infty(\bar\Omega)$). 
   Let us also assume, for some positive constant $\beta$, that
   \[
      b(x,y)\geq \beta>0,
      \qquad c(x,y)-b_x(x,y)\geq 0
      \qquad\mbox{for~all~}(x,y)\in\bar\Omega.
   \]
   Note that $b$ has a different meaning here compared with the previous chapters.
   
   We are interested in estimates of the Green's function $G(x,y;\xi,\eta)$
   associated with problem~\eqref{eq:Lu'}. For each fixed $(x,y)\in\Omega$, it satisfies
   the adjoint problem
   \begin{align*}
     L^*_{\xi\eta}G(x,y;\xi,\eta)
          =-\eps(G_{\xi\xi}+G_{\eta\eta})\!+\!b(\xi,\eta)G_\xi\!+\!c(\xi,\eta)G
         &=\delta(x-\xi)\,\delta(y-\eta),\,(\xi,\eta)\in\Omega,\\
     G(x,y;\xi,\eta)
         &=0,\hspace{75pt}(\xi,\eta)\in\partial\Omega.
   \end{align*}
   Here $L^*_{\xi\eta}$ is the adjoint differential operator to $L_{xy}$,
   and $\delta(\cdot)$ is the one-dimensional Dirac $\delta$-distribution.
   Figure~\ref{fig:Green}
   \begin{figure}[ht]
      \centerline{\includegraphics[width=0.7\textwidth]{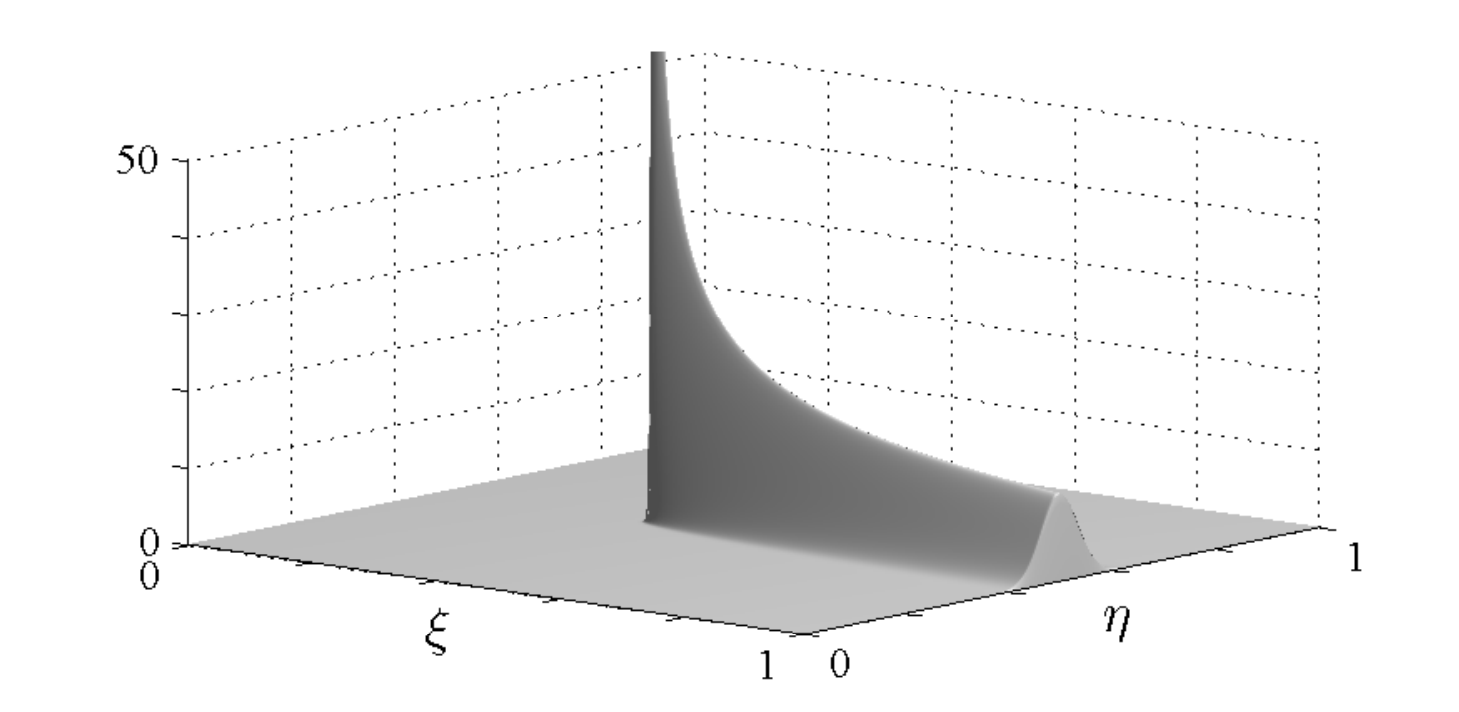}}
      \caption{Typical behaviour of the Green's function $G(\frac13,\frac12;\cdot,\cdot)$
               for problem \eqref{eq:Lu'} with $b=1$, $c=0$ and $\eps=10^{-3}$.}
             \label{fig:Green}
   \end{figure}
   shows a representation of a Green's function $G(1/3,1/2;\cdot,\cdot)$
   for a small value of $\eps=10^{-3}$ and coefficients $b=1$ and $c=0$.
   The singularity at $(\xi,\eta)=(x,y)$ and strong anisotropic behaviour
   of $G$ can be seen quite nicely. Near the boundary $\xi=1$ the Green's
   function has a strong boundary layer -- an outflow boundary layer.
   
   The unique solution $u$ of \eqref{eq:Lu'} has the representation
   \begin{gather}\label{eq:sol_prim}
     u(x,y)=\iint_{\Omega}G(x,y;\xi,\eta)\,f(\xi,\eta)\,d\xi\, d\eta.
   \end{gather}

   Our goal is to use \eqref{eq:sol_prim} and $L_1$-norm estimates of $G$
   to obtain pointwise error bounds of $u-u^N$, where $u^N$ is the numerical solution
   of a certain method. By using this idea, we get \emph{a-posteriori} error bounds
   with computable terms.

   In a more general numerical-analysis context, 
   we note that sharp estimates for continuous Green's functions (or
   their generalised versions) frequently play a crucial role in a
   priori and a posteriori error analyses \cite{erikss,Leyk,notch}.

   The main result for $L_1$-norm bounds on $G$ is the following from~\cite{FrK12_1}.

   \begin{thm}[Theorem 2.2 of \cite{FrK12_1}]\label{thm:green:main}
      Let $\eps\in(0,1]$.
      The Green's function $G(x,y;\xi,\eta)$ associated with \eqref{eq:Lu'}
      on the unit square $\Omega=(0,1)^2$  satisfies, for any $(x,y)\in\Omega$,
      the following bounds
      \begin{subequations}
      \begin{align}
\label{eq:green:G_xi}
         \norm{\pt_\xi  G(x,y;\cdot)}{L_1(\Omega)}
            &\leq C(1+|\ln \eps|),&
         \norm{\pt_\eta G(x,y;\cdot)}{L_1(\Omega)}
            &\leq C\eps^{-1/2}.
      \end{align}
      Furthermore, for any ball $B(x',y';\rho)$ of radius $\rho$
      centred at any $(x',y')\in\bar\Omega$, we have
      \begin{align}
         \norm{G(x,y;\cdot)}{W_1^1(B(x',y';\rho))}
            &\leq C\eps^{-1}\rho,\label{eq:green:G_grad}
      \end{align}
      while for the ball $B(x,y;\rho)$ of radius $\rho$ centred at $(x,y)$
      we have
      \begin{align}
         \norm{\pt^2_{\xi}  G(x,y;\cdot)}{L_1(\Omega\setminus B(x,y;\rho))}
            &\leq C\eps^{-1}\ln(2+\eps/\rho),\label{eq:green:G_xixi}\\
         \norm{\pt^2_{\eta}G(x,y;\cdot)}{L_1(\Omega\setminus B(x,y;\rho))}
            &\leq C\eps^{-1}(\ln(2+\eps/\rho)+|\ln\eps|).\label{eq:green:G_etaeta}
      \end{align}
      \end{subequations}
   \end{thm}

   Let us compare the first order results to those obtained in one dimension, see e.g.
   \cite[Theorems 3.23, 3.31]{Linss10}.
   Here we have for the Green's function $g^{cd}$ of a convection-diffusion problem
   \[
    \norm{\pt_\xi  g^{cd}(x;\cdot)}{L_1(\Omega)}\leq C
   \]
   and for $g^{rd}$ of a reaction-diffusion problem
   \[
    \norm{\pt_\xi  g^{rd}(x;\cdot)}{L_1(\Omega)}\leq C\eps^{-1/2}.
   \]
   Comparing these results with the results of Theorem~\ref{thm:green:main},
   we see an additional dependence on $|\ln\eps|$ in the streamline derivative.
   Thus the question for sharpness of these estimates is legitimate.
   In \cite{FrK10_1} it is shown that above bounds are sharp w.r.t. $\eps$.

   \begin{thm}[Theorem 3 of \cite{FrK10_1}]\label{thm:green:main_lower}
      Let $\eps\in(0,c_0]$ for some sufficiently small positive $c_0$.
      The Green's function $G$ associated with~the constant-coefficient
      problem \eqref{eq:Lu'}
      in the unit square $\Omega=(0,1)^2$ satisfies,
      for all $(x,y)\in[\frac14,\frac34]^2$,
      the following \underline{lower bounds}:
      There exists a constant $\underline{c}>0$ independent of $\eps$ such that
      \begin{subequations}\label{eq_thm:main_lower}
      \begin{align}
         \norm{\pt_{\xi} G(x,y;\cdot)}{L_1(\Omega)}
            &\geq \underline{c}|\ln \eps|,&
         \norm{\pt_{\eta} G(x,y;\cdot)}{L_1(\Omega)}
            &\geq \underline{c}\eps^{-1/2}.
      \end{align}
      {Furthermore, for any ball $B(x,y;\rho)$ of radius $\rho\le\frac18$, we have}
      \begin{align}
         \norm{G(x,y;\cdot)}{W_1^1(\Omega\cap B(x,y;\rho))}
            &\geq \begin{cases}
                     \underline{c}\rho/\eps, & \mbox{if~}\rho\le 2\eps,\\
                     \underline{c}(\rho/\eps)^{1/2},&\mbox{otherwise},\\
                  \end{cases}\\
         \norm{\pt^2_{\xi} G(x,y;\cdot)}{L_1(\Omega\setminus B(x,y;\rho))}
                       &\geq \underline{c}\eps^{-1}\ln(2+\eps/\rho),
           \quad&&\mbox{if~}\rho\le c_1\eps,
            \\
            \norm{\pt^2_{\eta} G(x,y;\cdot)}{L_1(\Omega\setminus B(x,y;\rho))}
            &\geq \underline{c}\eps^{-1}(\ln(2+\eps/\rho)+|\ln\eps|),
           &&\mbox{if~}\rho\le{\textstyle\frac18},
      \end{align}
      where $c_1$ is a sufficiently small positive constant.
      \end{subequations}
   \end{thm}

   Not that the restriction $(x,y)\in[\frac14,\frac34]^2$ can be replaced by
   $(x,y)\in[\theta,1-\theta]^2$ with $\theta\in(0,\frac12)$. Doing so, we
   have to replace $\rho\le\frac18$ by $\rho\le\frac12\theta$.

   Above results have been proved in 2d and 3d in \cite{FrK12_1, FrK10_1, FrK11}.
   The basic idea is to look at a frozen coefficient version of \eqref{eq:Lu'}
   and to analyse the behaviour of its fundamental solution and of the difference
   to the fundamental solution of the original problem. This approach is sometimes
   called \emph{parametrix method}.
   The results can be generalised to arbitrary dimensions, say $n\in\N$.
   In order to do so, let us denote by $\ve{x}=(x_1,x_2,\dots,x_n)$
   a vector in $\R^n$ and by $K_s$ the modified Bessel function of second kind
   and order $s$ with $s\in\R$.
   
   The basic idea is to look at the fundamental solution of
   \begin{align}
     \bar L^*_{\ve\xi}\bar g(\ve{x};\ve\xi)
         =-\eps\laplace_{\ve\xi}\bar g+b(\ve{x})\bar g_{\xi_1}
        &=\delta(\ve{x}-\ve\xi),\quad \ve\xi\in\R^n\label{eq:conv_Green_adj_const}
   \end{align}
   where $\delta(\cdot)$ is the $n$-dimensional Dirac-distribution. 
   For fixed $\ve{x}$ the coefficient $b(\ve{x})$ in \eqref{eq:conv_Green_adj_const}
   is constant and we can solve the problem explicitly.
   To simplify our presentation, let $q=\frac{1}{2}b(\ve{x})$ for fixed $\ve{x}\in(0,1)^n$.
   Now a transformation, see \cite{KSh87}, can be used to change the type of the problem
   from convection-diffusion to reaction-diffusion.
   For reaction-diffusion problems with constant coefficients the fundamental solution
   is known and we obtain the fundamental solution of
   \eqref{eq:conv_Green_adj_const} as
   \[
      \bar g(\ve{x};\ve\xi)
       =\frac{1}{(2\pi)^{n/2}\eps^{n-1}}\,\left(\frac{q}{\hat r}\right)^{n/2-1}e^{q\hat\xi_1}K_{n/2-1}(q\hat r),
              \qquad
       q=q(\ve{x})={\textstyle\frac{1}{2}}b(\ve{x})
   \]
   where $\hat r=\norm{\ve\xi-\ve{x}}{2}/\eps$ and $\hat\xi_k = (\xi_k-x_k)/\eps$.
   Note that for $n=2$ we obtain
   \[
      \bar g_2(x,y;\xi,\eta)
       =\frac{1}{2\pi\eps}\,e^{q\hat\xi_1}K_{0}(q\hat r),
              \qquad
       q=q(x,y)={\textstyle\frac{1}{2}}b(x,y),
   \]
   the fundamental solution used in \cite{FrK12_1} and for $n=3$
   \[
      \bar g_3(\ve{x};\ve\xi)
       =\frac{1}{(2\pi)^{3/2}\eps^2}\,\left(\frac{q}{\hat r}\right)^{1/2}e^{q\hat\xi_1}K_{1/2}(q\hat r)
       =\frac{1}{4\pi\eps^2}\,\frac{e^{q(\xi_1-x_1-r)/\eps}}{\hat r},
        \qquad
       q=q(\ve{x})={\textstyle\frac{1}{2}}b(\ve{x}),
   \]
   the fundamental solution used in \cite{FrK11}.
   
   The modified Bessel functions $K_s$ of order $s$ and those of
   order zero behave asymptotically very similar, see
   \cite[Sections 10.25 to 10.60]{NIST10}.
   Therefore, to modify the analysis presented in \cite{FrK12_1, FrK10_1, FrK11}
   to the $n$-dimensional case is straightforward, though tedious and
   we obtain the analogue to Theorem~\ref{thm:green:main} and
   \ref{thm:green:main_lower} also in the $n$-dimensional case.

\section{A-Posteriori Error Estimation}

   Here we want to apply the $L_1$-norms of the Green's function and derive \emph{a-posteriori}
   error estimates in the $L_\infty$-norm. The analysis following is from the forthcoming
   paper~\cite{FrK09_3}.
   Note, that in this section derivatives are to be understood in the sense of distributions.

   Let the domain $\Omega$ be discretised by a rectangular tensor-product mesh $T$
   with the nodes $(x_i,y_j)$,
   where $0=x_0<x_1<\ldots<x_N=1$ and $0=y_0<y_1<\ldots<y_M=1$ for $N,M\in\N$.
   On this mesh we derive the main ingredient of an a-posteriori error estimator,
   an $(L_\infty, W_{-1,\infty})$-stability result, following \cite[Theorem 4.1]{Kopt08}.
  \begin{thm}\label{thm:green:stability}
   Let $u$ be the unique solution of \eqref{eq:Lu'} for a given
   right-hand side $f$ satisfying
   \begin{gather}\label{eq:green:rhs_form}
      f(x,y)=\bar f(x,y)
             -\frac{\partial}{\partial x}[F_1(x,y)+\bar F_1(x,y)]
             -\frac{\partial}{\partial y}[F_2(x,y)+\bar F_2(x,y)]
   \end{gather}
   where
   \begin{align*}
      F_1(x,y)|_{(x_{i-1},x_i)}&=A_i(y)(x-x_{i-1/2}),&i=1,\dots,N\\
      F_2(x,y)|_{(y_{j-1},y_j)}&=B_j(x)(y-y_{j-1/2}),&j=1,\dots,M
   \end{align*}
   and $\bar f,\,\bar F_1,\,\bar F_2,\,A_i,$ and $B_j$ are arbitrary functions in $L_\infty(\Omega)$.

   Then it holds that
   \begin{align*}
     \norm{u}{L_\infty(\Omega)}
      \leq C\bigg[&\norm{\bar f}{L_\infty(\Omega)}+
                   (1+|\ln \eps|)\norm{\bar F_1}{L_\infty(\Omega)}+
                   \eps^{-1/2}\norm{\bar F_2}{L_\infty(\Omega)}+\\
                  &\max_{i=1,\dots,N}\left\{
                                      \min\left\{
                                            h_{i}^2\frac{\ln(2+\eps/\kappa_h)}{\eps},
                                            h_{i}(1+|\ln \eps|)
                                          \right\}
                                      \max_{y\in[0,1]}|A_i(y)|
                                     \right\}+\\&
                   \max_{j=1,\dots,M}\left\{
                                      \min\left\{
                                            k_{j}^2\frac{|\ln\eps|+\ln(2+\eps/\kappa_k)}{\eps},
                                            \frac{k_{j}}{\eps^{1/2}}
                                          \right\}
                                      \max_{x\in[0,1]}|B_j(x)|
                                     \right\}
            \bigg]
   \end{align*}
   with $\kappa_h=\min h_i,\,\kappa_k=\min k_j$.
   \end{thm}
   \begin{rem}
    The existence of $u\in L_\infty(\Omega)$ for a given right-hand side $f$ of the form
    \eqref{eq:green:rhs_form} follows from the classical results~\cite[Chap. 3 Theorems 5.2, 13.1]{Lady68}.
   \end{rem}
   \begin{proof}[Proof of Theorem~\ref{thm:green:stability}]
     Using the linearity of the operator $L$, we split $f$ into different parts
     and analyse them separately. For simplicity of the representation, denote by $g(\xi,\eta)=G(x,y;\xi,\eta)$.

     \textbf{1)} Let $\mathbf{\bar F_1=\bar F_2=F_1=F_2=0}$, i.e. $f=\bar f$.\\
      The maximum principle (or \eqref{eq:sol_prim} and $\norm{g}{L_1(\Omega)}\leq C$) implies
      \[
        \norm{u}{L_\infty(\Omega)}\leq C \norm{\bar f}{L_\infty(\Omega)}.
      \]

     \textbf{2)} Let $\mathbf{\bar f=F_1=F_2=0}$, i.e. $f=-\frac{\partial}{\partial x}\bar F_1
                                                          -\frac{\partial}{\partial y}\bar F_2$.\\
     We represent $u$ using \eqref{eq:sol_prim}.
     Integration by parts and a Cauchy-Schwarz inequality give
     \begin{align*}
       u(x,y)&=\iint_\Omega g(\xi,\eta)f(\xi,\eta) d\xi d\eta\\
             &=\iint_\Omega g_\xi(\xi,\eta)\bar F_1(\xi,\eta) d\xi d\eta+
               \iint_\Omega g_\eta(\xi,\eta)\bar F_2(\xi,\eta) d\xi d\eta\\
             &\leq \norm{G_\xi}{L_1(\Omega)}\norm{\bar F_1}{L_\infty(\Omega)}+
                   \norm{G_\eta}{L_1(\Omega)}\norm{\bar F_2}{L_\infty(\Omega)}.
     \end{align*}
     With \eqref{eq:green:G_xi} we obtain
      \[
        \norm{u}{L_\infty(\Omega)}\leq C \left[(1+|\ln\eps|)\norm{\bar F_1}{L_\infty(\Omega)}+
                                               \eps^{-1/2}\norm{\bar F_2}{L_\infty(\Omega)}\right].
      \]

     \textbf{3)} Let $\mathbf{\bar f=\bar F_1=\bar F_2=F_2=0}$, $f=-\frac{\partial}{\partial x}F_1$.\\
     Using \eqref{eq:sol_prim} and integration by parts again, we have
     \begin{align*}
       u(x,y)
        &=\iint_\Omega F_1(\xi,\eta) g_\xi(\xi,\eta) d\xi d\eta
         =\sum_{i=1}^N\iint_{\Omega_i} A_i(\eta)(\xi-\xi_{i-1/2})g_\xi(\xi,\eta)d\xi d\eta
     \end{align*}
     where $\Omega_i=(x_{i-1},x_i)\times[0,1]$.
     The Green's function $g$ has a singularity at $(x,y)$.
     Define $0<n<N$ where $x\in[x_{n-1/2},x_{n+1/2}]$
     and $\Omega'=(x_{n-1},x_{n+1})\times(y-\tilde h_n,y+\tilde h_n)$ where
     $\tilde h_n=\min\{h_n,h_{n+1}\}/2$. Note that the singularity now lies in $\Omega'$.

     Defining the singularity-free function $\tilde g$ by $\tilde g=g$ in $\Omega\setminus\Omega'$
     and $\tilde g=0$ in $\Omega'$ we obtain
     \begin{align*}
       u(x,y)
        &=
         \sum_{i=1}^N\iint_{\Omega_i}\!\!\!\! A_i(\eta)(\xi-\xi_{i-1/2})\tilde g_\xi(\xi,\eta)d\xi d\eta+
         \sum_{i=n}^{n+1}\iint_{\Omega_i\cap\Omega'}\!\!\!\!\! A_i(\eta)(\xi-\xi_{i-1/2})g_\xi(\xi,\eta)d\xi d\eta\\
        &=:S_1+S_2.
     \end{align*}
     The term $S_1$ can be estimated in two different ways. Either by
     \begin{align*}
       \left|\int_{x_{i-1}}^{x_i}\!\!(\xi-x_{i-1/2})\tilde g_\xi(\xi,\eta)d\xi\right|
        &\leq \frac{h_i}{2}\int_{x_{i-1}}^{x_i}|\tilde g_\xi(\xi,\eta)|d\xi
     \intertext{or by}
       \left|\int_{x_{i-1}}^{x_i}\!\!(\xi-x_{i-1/2})\tilde g_\xi(\xi,\eta)d\xi\right|
        &= \left|\int_{x_{i-1}}^{x_i}\!\!(\xi-x_{i-1/2})
                 \int_{x_{i-1}}^\xi\!\!\tilde g_{\xi\xi}(s,\eta)ds d\xi\right|
        \leq \frac{h_i^2}{4}\int_{x_{i-1}}^{x_i}\!\!|\tilde g_{\xi\xi}(\xi,\eta)|d\xi.
     \end{align*}
     Note that $\tilde g_{\xi\xi}$ is well defined.
     In order to use these two possibilities,
     decompose $A_i=A_i^1+A_i^2$ where
     \[
       A_i^1=\begin{cases}
               A_i, & 2h_i\eps(1+|\ln\eps|)\leq h_i^2\ln(2+\eps/\kappa_h)\\
               0,   &\mbox{otherwise}
             \end{cases}
       \quad\mbox{and }
       A_i^2=A_i-A_i^1.
     \]
     This yields by using Theorem~\ref{thm:green:main}
     \begin{align*}
       |S_1|
        &\leq \sum_{i=1}^N
                \int_0^1 |A_i(\eta)|
                         \left|
                           \int_{x_{i-1}}^{x_i}(\xi-\xi_{i-1/2})\tilde g_\xi(\xi,\eta)d\xi
                         \right| d\eta\\
        &\leq  \max_{i=1,\dots,N}\left\{\frac{h_i}{2}\max_{\eta\in[0,1]}|A_i^1(\eta)|\right\}
                 \iint_{\Omega\setminus\Omega'} |G_\xi(\xi,\eta)|d\xi d\eta+\\&\hspace{0.5cm}
               \max_{i=1,\dots,N}\left\{\frac{h_i^2}{4}\max_{\eta\in[0,1]}|A_i^2(\eta)|\right\}
                 \iint_{\Omega\setminus\Omega'} |G_{\xi\xi}(\xi,\eta)|d\xi d\eta
               \\
        &\leq C\max_{i=1,\dots,N}\left\{
                                  \min\left\{
                                        h_i(1+|\ln\eps|),\frac{h_i^2}{\eps}\ln(2+\eps/\kappa_h)
                                      \right\}\max_{\eta\in[0,1]}|A_i(\eta)|
                                 \right\}.
     \end{align*}
     For $S_2$ we use either \eqref{eq:green:G_xi}
     \[
       \norm{G_\xi}{L_1(\Omega)}\leq C(1+|\ln\eps|)
     \]
     or \eqref{eq:green:G_grad}
     \[
       \norm{G_\xi}{L_1(B(a,b,\rho))}\leq C\eps^{-1}\rho.
     \]
     Let $A_i=\tilde A_i^1+\tilde A_i^2$ with
     \[
       \tilde A_i^1=\begin{cases}
               A_i, & h_i\eps(1+|\ln\eps|)\leq h_i^2\\
               0,   &\mbox{otherwise}
             \end{cases}
       \quad\mbox{and }
       \tilde A_i^2=A_i-\tilde A_i^1.
     \]
     Then holds
     \begin{align*}
       |S_2|
        &\leq \sum_{i=n}^{n+1}\iint_{\Omega_i\cap\Omega'} |A_i(\eta)||(\xi-\xi_{i-1/2})||g_\xi(\xi,\eta)|d\xi d\eta\\
        &\leq \sum_{i=n}^{n+1}\frac{h_i}{2}\max_{\eta\in[0,1]}|A_i(\eta)|
                \iint_{B(x_{i-1/2},y,h_i)} |g_\xi(\xi,\eta)|d\xi d\eta\\
        &\leq C\left(\max_{i=n,n+1}\left\{
                                 \frac{h_i}{2}\max_{\eta\in[0,1]}|\tilde A_i^1(\eta)|
                            \right\}|\ln\eps|+
              \max_{i=n,n+1}\left\{
                                 \frac{h_i^2}{2\eps}\max_{\eta\in[0,1]}|\tilde A_i^2(\eta)|
                            \right\}\right)\\
        &\leq C\max_{i=n,n+1}\left\{
                             \min\left\{
                                   h_i(1+|\ln\eps|),\frac{h_i^2}{\eps}
                                 \right\}\max_{\eta\in[0,1]}|A_i(\eta)|
                            \right\}.
     \end{align*}
     Thus we obtain
     \[
       \norm{u}{L_\infty(\Omega)}
        \leq C\max_{i=1,\dots,N}\left\{
                                  \min\left\{
                                        h_i(1+|\ln\eps|),\frac{h_i^2}{\eps}\ln(2+\eps/\kappa_h)
                                      \right\}\max_{\eta\in[0,1]}|A_i(\eta)|
                                 \right\}.
     \]

     \textbf{4)} Let $\mathbf{\bar f=\bar F_1=F_1=\bar F_2=0}$, i.e. $f=-\frac{\partial}{\partial y}F_2$.\\
     This case can be treated similarly to the one above.
     Using a similar splitting of $u=\widetilde S_1+\widetilde S_2$ gives
     \begin{align*}
       |\widetilde S_1|
        &\leq C\max_{j=1,\dots,M}\left\{
                                  \min\left\{
                                        \frac{k_j}{\eps^{1/2}},\frac{k_j^2}{\eps}(|\ln\eps|+\ln(2+\eps/\kappa_k))
                                      \right\}\max_{\xi\in[0,1]}|B_j(\xi)|
                                 \right\}
     \intertext{and}
       |\widetilde S_2|
        &\leq C\max_{j=m,m+1}\left\{
                             \min\left\{
                                   \frac{k_j}{\eps^{1/2}},\frac{k_j^2}{\eps}
                                 \right\}\max_{\xi\in[0,1]}|B_j(\xi)|
                            \right\}
     \end{align*}
     and therefore
     \[
       \norm{u}{L_\infty(\Omega)}
        \leq C\max_{j=1,\dots,M}\left\{
                                  \min\left\{
                                        \frac{k_j}{\eps^{1/2}},\frac{k_j^2}{\eps}(|\ln\eps|+\ln(2+\eps/\kappa_k))
                                      \right\}\max_{\xi\in[0,1]}|B_j(\xi)|
                                 \right\}.
     \]
     By combining these estimates the stability result is proved.
   \end{proof}

   Note that in above results the global minima $\kappa_h$ and $\kappa_k$ can be replaced by
   local minima over two adjacent cells each.
   
\subsection*{Application to an Upwind Method}
   So far our Green's function estimates have not been applied to finite element methods.
   The Green's function $G$ is in general not in $H_0^1(\Omega)$ which complicates the derivation
   of uniform a-posteriori error estimators via above approach.
   Further research is needed to apply this approach to finite element methods.
   
   Instead, we will apply the stability result to an upwind finite difference method.
   Let us start by rewriting \eqref{eq:Lu'} as
   \begin{subequations}\label{eq:Lu_split}
   \begin{align}
     Lu&=-(A_1u)_x-(A_2u)_y-(Bu)_x+Cu=f\\
   \intertext{where}
     A_1u &=\eps u_x,\quad
     A_2u =\eps u_y,\quad
     Bu   =bu\quad\mbox{and}\quad
     Cu   =cu.\label{eq:L_split:2}
   \end{align}
   \end{subequations}

   Using the index sets $I=\{1,\dots,N-1\}$, $\bar I=\{0,\dots,N\}$, $J=\{1,\dots,M-1\}$
   and $\bar J=\{0,\dots,M\}$, we define our discrete counterpart to \eqref{eq:Lu_split}:
   \begin{subequations}\label{eq:Lu_discr}
   \begin{align}
      L^N\mathbf{u}_{ij}
          =&-\widetilde D_x A_1^N\mathbf{u}_{ij}
            -\widetilde D_y A_2^N\mathbf{u}_{ij}
            -\widetilde D_x B^N\mathbf{u}_{ij}
            +C^N\mathbf{u}_{ij}\notag\\
          =&-\eps (D_x^2\mathbf{u}_{ij}+D_y^2\mathbf{u}_{ij})
           -\widetilde D_x (\mathbf{b}_{ij}\mathbf{u}_{ij})
           +\mathbf{c}_{ij}\mathbf{u}_{ij}
          \quad=
          \mathbf{f}_{ij},
      &&i\in I,j\in J\\
      \mathbf{u}_{i,0}=&\,\mathbf{u}_{i,M}=0,&&i\in\bar I\\
      \mathbf{u}_{0,j}=&\,\mathbf{u}_{N,j}=0,&&j\in\bar J.
   \end{align}
   \end{subequations}
   where
   \begin{align*}
     \mathbf{f}_{ij}&=f(x_i,y_j)\\
     A_1^N\mathbf{u}_{ij}&= \eps D^-_x\mathbf{u}_{ij},\quad
     A_2^N\mathbf{u}_{ij} = \eps D^-_y\mathbf{u}_{ij},\quad
     B^N\mathbf{u}_{ij}   = \mathbf{b}_{ij}\mathbf{u}_{ij}\quad\mbox{and}\quad
     C^N\mathbf{u}_{ij}   = \mathbf{c}_{ij}\mathbf{u}_{ij}
   \end{align*}
   with the standard backward difference operators $D^-$.
   With $\hbar_i=(h_i+h_{i+1})/2$ 
   the other discrete operators are defined as
   \begin{align*}
     \widetilde D_x\mathbf{u}_{ij}
        &=\frac{\mathbf{u}_{i+1,j}-\mathbf{u}_{ij}}{\hbar_i},&
     D_x^2\mathbf{u}_{ij}
        &=\frac{1}{\hbar_i}\left[ \frac{\mathbf{u}_{i+1,j}-\mathbf{u}_{ij}}
                                           {h_{i+1}}
                                     -\frac{\mathbf{u}_{i,j}-\mathbf{u}_{i-1,j}}
                                           {h_{i}}
                           \right],
        &i\in I,\,
         j\in J,
   \end{align*}
   and similarly in $y$-direction.

   Note that \eqref{eq:Lu_discr} is a non-standard upwind finite difference method.
   The difference to the standard upwind FDM is the treatment of the convective term
   by $\widetilde D_x$ instead of $D^+_x$. The reason for this different treatment
   lies in the following analysis.

   Let us use the continuous residual, i.e. \mbox{$L(\mathbf{u}^\B-u)$}.
   Here $\mathbf{u}^\B$ denotes the piecewise bilinear interpolant
   of the discrete variable $\mathbf{u}$. With $\mathbf{u}^\I$ and $\mathbf{u}^\J$
   as the one-dimensional piecewise linear interpolations in $x$- and $y$-direction, respectively
   we have
   \[
      \mathbf{u}^\B=(\mathbf{u}^\I)^\J=(\mathbf{u}^\J)^\I.
   \]
   With a proper extension of our discrete operators to the boundary of $\Omega$,
   it holds for the residual
   \begin{align}
     L(\mathbf{u}^\B-u)
      &= -\left[(A_1\mathbf{u}^\I)_x+\mathbf{F_1}^\I\right]^\J
         -\left[(A_2\mathbf{u}^\J)_y+\mathbf{F_2}^\J\right]^\I\notag\\&\quad
         -\left[(B\mathbf{u}^\I)_x+\mathbf{F_3}^\I\right]^\J
         +\left[C\mathbf{u}^\B-\mathbf{F_4}^\B\right]-f+\mathbf{f}^\B\label{eq:residual}
   \end{align}
   where
   \begin{align*}
     \mathbf{F_{1}}_{ij}&:=-\widetilde D_x A_1^N\mathbf{u}_{ij},&
     \mathbf{F_{2}}_{ij}&:=-\widetilde D_y A_2^N\mathbf{u}_{ij},\\
     \mathbf{F_{3}}_{ij}&:=-\widetilde D_x B^N\mathbf{u}_{ij},&
     \mathbf{F_{4}}_{ij}&:=C^N\mathbf{u}_{ij}.
   \end{align*}

   Let us start with the first term on the right-hand side of \eqref{eq:residual}
   for $x\in[x_{i-1},x_i]$ and fixed $y=y_j$.
   With the auxiliary terms 
   \begin{align*}
    Q^1  &:= \int_x^1 \mathbf{F_{1}}^\I,&
    Q^1_i&:= \sum_{k=i}^{N-1}\mathbf{F_{1}}_{kj}\hbar_k
   \end{align*}
   we obtain 
   \[
     Q^1_i
      =A_1^N\mathbf{u}_{ij}-A_1^N\mathbf{u}_{Nj}
      =A_1^N\mathbf{u}_{ij}
   \]
   and therefore
  \[
     (A_1\mathbf{u}^\I)_x+\mathbf{F_1}^\I
       = \pt_x(A_1\mathbf{u}^\I-Q^1)
       = \pt_x(A_1\mathbf{u}^\I-A_1^N\mathbf{u}_{ij})+\pt_x(Q^1_i-Q^1)
       = \pt_x(Q^1_i-Q^1).
   \]
   Now using the summation in $Q^1_i$ we can estimate further
   \begin{align*}
     \pt_x(Q^1_i-Q^1)
      &= \pt_x\left( \int_{x_i}^1 \mathbf{F_1}^\I
                    -\int_{x}^1\mathbf{F_1}^\I
                    +\mathbf{F_1}_{ij}\frac{h_i}{2}
                    -\mathbf{F_1}_{Nj}\frac{h_N}{2}\right)\notag\\
      &= \pt_x\left(-\int_x^{x_i} \mathbf{F_1}^\I
                    +\mathbf{F_1}_{ij}\frac{h_i}{2}\right)\notag\\
      &= \pt_x\left( \mathbf{F_1}_{ij}\frac{(x-x_{i-1})^2}{2h_i}
                    -\mathbf{F_1}_{i-1,j}\frac{(x_i-x)^2}{2h_i}\right).
   \end{align*}
   Thus we obtain
   \begin{gather}\label{eq:Q1_diff}
    -\left[(A_1\mathbf{u}^\I)_x+\mathbf{F_1}^\I\right]^\J\bigg|_{x \in[x_{i-1},x_i]}
     = -\pt_x\left( \mathbf{F_1}_{ij}\frac{(x-x_{i-1})^2}{2h_i}
                    -\mathbf{F_1}_{i-1,j}\frac{(x_i-x)^2}{2h_i}\right)^\J.
   \end{gather}

   With similar techniques the other terms of \eqref{eq:residual} can be rewritten.
   Note that \eqref{eq:Q1_diff} can be further transformed to yield second-order terms
   in $h_i$ but then we obtain discrete third-order derivatives.
   We apply this for the $y$-derivatives. Now using the stability result
   of Theorem \ref{thm:green:stability} to the right-hand side of the
   continuous residual~\eqref{eq:residual} yields an \emph{a-posteriori} error estimator.
   
   \begin{thm}\label{thm:green:estimator}
    Let $\mathbf{u}$ be the solution of \eqref{eq:Lu_discr} and $u$ be the solution of \eqref{eq:Lu'}.
    Then holds
   \begin{align*}
     \norm{\mathbf{u}^\B-u}{L_\infty(\Omega)}
      &\leq C\bigg(
             \max_{\stackrel{i=0,\dots,N}{j=1,\dots,M}}M^1_{ij}+
             \max_{\stackrel{i=0,\dots,N}{j=1,\dots,M}}M^2_{ij}+
             \max_{\stackrel{i=0,\dots,N}{j=1,\dots,M}}M^3_{ij}+\\&\hspace{1.2cm}
             \max_{\stackrel{i=1,\dots,N}{j=0,\dots,M}}M^4_{ij}+
             \max_{\stackrel{i=1,\dots,N}{j=0,\dots,M}}M^5_{ij}+
             \max_{\stackrel{i=1,\dots,N}{j=0,\dots,M}}M^6_{ij}+
             \max_{\stackrel{i=1,\dots,N}{j=0,\dots,M}}M^7_{ij}
            \bigg)
      \intertext{with the terms depending on discrete $y$-derivatives}
      M_{ij}^1 &:= \min\{\eps^{1/2} k_j,k_j^2(|\ln\eps|+\ln(2+\eps/\kappa_k))\}
                  \min\{|D_y^2\mathbf{u}_{i,j-1}|,|D_y^2\mathbf{u}_{ij}|\},\\
      M_{ij}^2 &:= \eps^{1/2} k_j^2
                   |D_y^-D_y^2\mathbf{u}_{i,j}|,\\
      M_{ij}^3 &:= k_j^2
                   (1+|D_y^-\mathbf{u}_{ij}|^2),
      \intertext{and terms depending on discrete $x$-derivatives}
      M_{ij}^4 &:= \eps h_i(1+|\ln\eps|)\max\{|[D_{x}^2\mathbf{u}]_{i-1,j}|,|[D_{x}^2\mathbf{u}]_{ij}|\},\\
      M_{ij}^5 &:= h_{i}^2(1+|[D_{x}^-\mathbf{u}]_{ij}|^2),\\
      M_{ij}^6 &:= h_{i}(1+|\ln\eps|)\max\{|[\widetilde D_{x}\mathbf{u}]_{i-1,j}|,|[\widetilde D_{x}\mathbf{u}]_{ij}|\},\\
      M_{ij}^7 &:= h_{i}(1+|\ln\eps|)(1+|[D_{x}^-\mathbf{u}]_{ij}|).
   \end{align*}
   \end{thm}

   Note that formally, $M^1$ to $M^3$ are of order $k_j^2$ while $M^4$, $M^6$ and $M^7$ are of order $h_i$.
   Only $M^5$ is of order $h_i^2$ and therefore probably negligible.
   Thus, the estimator is formally of first order (if $k_j^2\leq h_i$) which
   is consistent with the formal order of an upwind method.

   The constant $C$ in the error bound of Theorem~\ref{thm:green:estimator} is unknown.
   By setting it to $C=1$ we obtain an \emph{error indicator} that gives us
   information about the convergence behaviour, though not about the exact value of
   the error. For singularly perturbed problems the uniformity of the indicator
   is usually more important than the precise value of $C$.

   \begin{rem}\label{rem:green:estimator}
     The part $M^2$ contains discrete third-order derivatives. They are costly to evaluate
     and therefore an estimator with only second-order derivatives would be beneficial.
     In \cite[§6]{Kopt07} an idea is used, that bounds the third-order derivative by a
     second-order derivative term. This approach could be used here too. We will tackle it in
     the forthcoming paper \cite{FrK09_3}.
   \end{rem}

   \subsection*{A Numerical Example}
   Let us consider the numerical example
  \begin{subequations}\label{eq:num_example2}
  \begin{alignat}{2}
    -\eps\laplace u - u_x + \frac{1}{2} u & = f
    &\quad&\text{in }\Omega=(0,1)^2,\\
    u & = 0 && \text{on }\partial\Omega,
  \end{alignat}
  where the right-hand side $f$ is chosen such that
  \begin{gather}
    u(x,y) = \left(\cos\frac{\pi x}{2} - \frac{ \e^{-x/\eps} - \e^{-1/\eps}}%
    {1-\e^{-1/\eps}}
    \right)
    \frac{\left(1-\e^{-y/\sqrt{\eps}}\right)
    \left( 1-\e^{-(1-y)/\sqrt{\eps}} \right)}{1-\e^{-1/\sqrt{\eps}}}
  \end{gather}
  \end{subequations}
  is the solution.

  \subsubsection*{Dependence on $\eps$}
  In our first experiment we look into the uniformity w.r.t. $\eps$ of the indicator given
  in Theorem~\ref{thm:green:estimator} with $C=1$.
  For given values of $\eps$ we compute the numerical
  solution on an a-priori chosen Shishkin mesh for $N=64$ and compare the
  results in Figure~\ref{fig:estim:eps}.
  \begin{figure}[tb]
      \centerline{\includegraphics[width=0.8\textwidth]{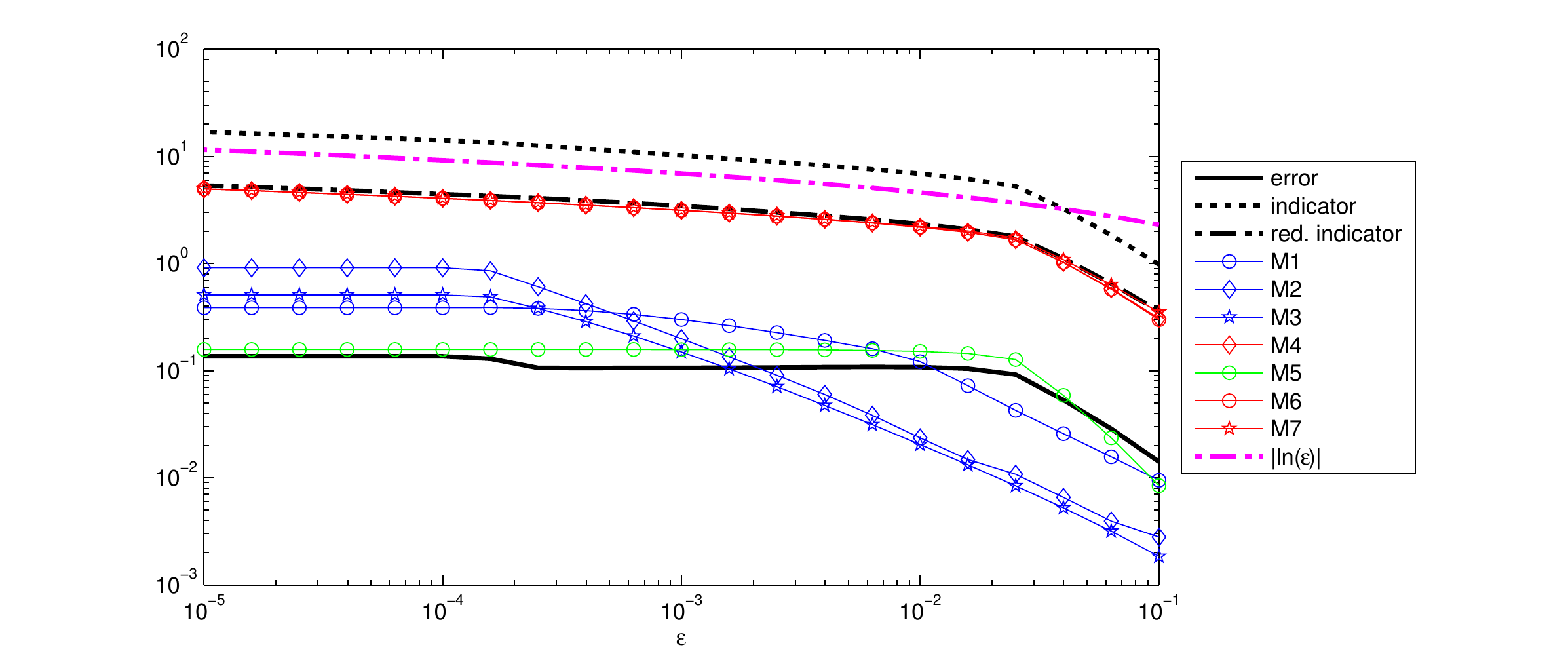}}
      \caption{Error and estimated error of \eqref{eq:num_example2} for $N=64$ on a Shishkin mesh}
             \label{fig:estim:eps}
  \end{figure}
  Therein for each component of the indicator a line is shown. Additionally,
  a solid black line represents the real error, and a black dash-dot line
  represents a modified indicator.
  The modified indicator takes only the maxima of $M^1$, $M^3$, $M^4$ and $M^7$.
  In numerical simulations this modification represents the behaviour of the
  error much better than the real indicator. For another motivation, see also
  Remark~\ref{rem:green:estimator}.

  In Figure~\ref{fig:estim:eps} both indicators behave like $|\ln\eps|$, which is
  also given for comparison as a line in magenta. But the real error stays almost 
  constant for $\eps$ becoming smaller. 
  Thus, there is a $|\ln\eps|$-dependence in our estimators coming
  from the Green's function estimates, although they are sharp.
  This behaviour was seen for several different examples.

  As a consequence we will use from now on the heuristic indicator 
  \begin{align*}
     \eta
      &:=\bigg(
             \max_{\stackrel{i=0,\dots,N}{j=1,\dots,M}}M^1_{ij}+
             \max_{\stackrel{i=0,\dots,N}{j=1,\dots,M}}M^2_{ij}+
             \max_{\stackrel{i=0,\dots,N}{j=1,\dots,M}}M^3_{ij}+\\&\hspace{1.2cm}
             \max_{\stackrel{i=1,\dots,N}{j=0,\dots,M}}M^4_{ij}+
             \max_{\stackrel{i=1,\dots,N}{j=0,\dots,M}}M^5_{ij}+
             \max_{\stackrel{i=1,\dots,N}{j=0,\dots,M}}M^6_{ij}+
             \max_{\stackrel{i=1,\dots,N}{j=0,\dots,M}}M^7_{ij}
            \bigg)
     \intertext{and the modified indicator}
     \widetilde\eta
      &:=\bigg(
             \max_{\stackrel{i=0,\dots,N}{j=1,\dots,M}}M^1_{ij}+
             \max_{\stackrel{i=0,\dots,N}{j=1,\dots,M}}M^3_{ij}+
             \max_{\stackrel{i=1,\dots,N}{j=0,\dots,M}}M^4_{ij}+
             \max_{\stackrel{i=1,\dots,N}{j=0,\dots,M}}M^7_{ij}
            \bigg)
      \intertext{where}
      M_{ij}^1 &:= \min\{\eps^{1/2} k_j,k_j^2\ln(2+\eps/\kappa_k)\}
                  \min\{|D_y^2\mathbf{u}_{i,j-1}|,|D_y^2\mathbf{u}_{ij}|\},\\
      M_{ij}^2 &:= \eps^{1/2} k_j^2
                  |D_y^-D_y^2\mathbf{u}_{i,j}|,\qquad\qquad
      M_{ij}^3  := k_j^2
                  (1+|D_y^-\mathbf{u}_{ij}|^2),\\
      M_{ij}^4 &:= \eps h_i\max\{|[D_{x}^2\mathbf{u}]_{i-1,j}|,|[D_{x}^2\mathbf{u}]_{ij}|\},\\
      M_{ij}^5 &:= h_{i}^2(1+|[D_{x}^-\mathbf{u}]_{ij}|^2),\qquad\qquad
      M_{ij}^6  := h_{i}\max\{|[\widetilde D_{x}\mathbf{u}]_{i-1,j}|,|[\widetilde D_{x}\mathbf{u}]_{ij}|\},\\
      M_{ij}^7 &:= h_{i}(1+|[D_{x}^-\mathbf{u}]_{ij}|).
   \end{align*}
   
  Figure~\ref{fig:estim:noeps}
  \begin{figure}[tbp]
      \centerline{\includegraphics[width=0.8\textwidth]{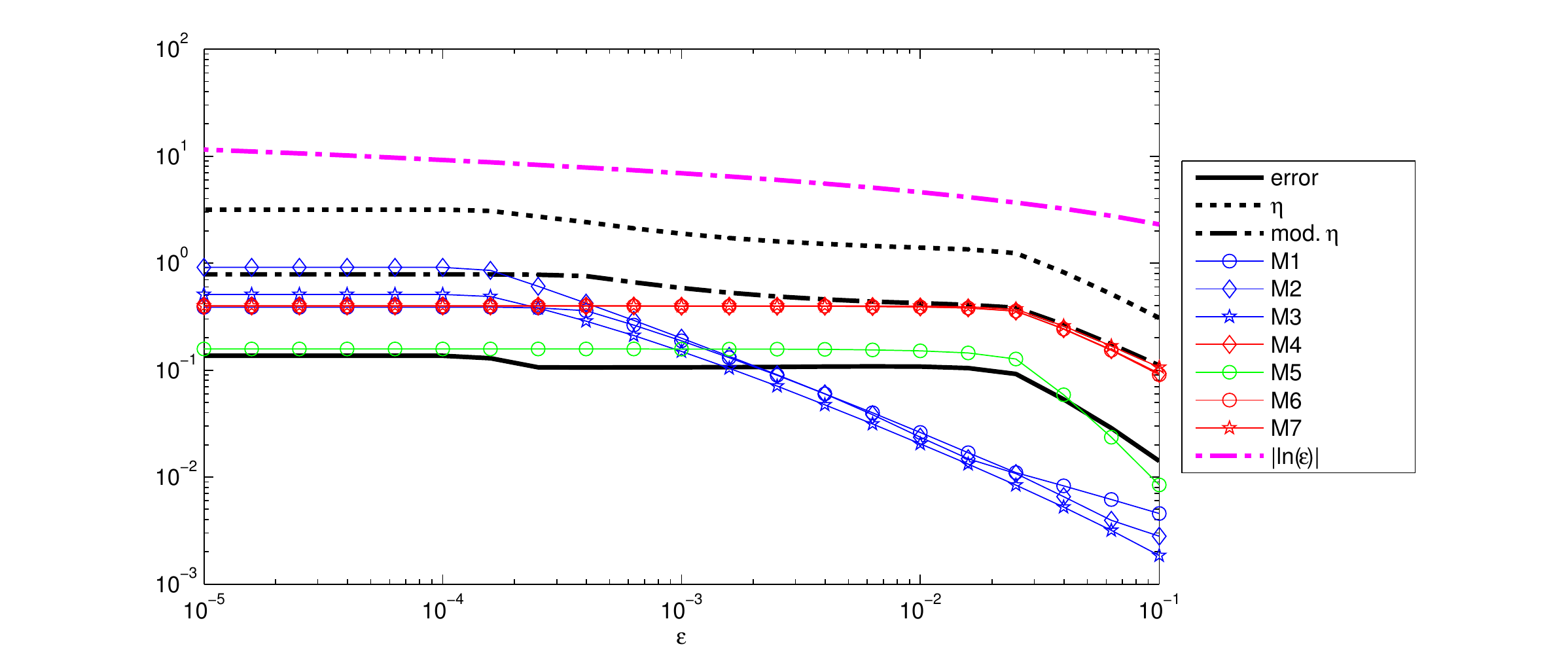}}
      \caption{Error and modified estimated error of \eqref{eq:num_example2} for $N=64$ on a Shishkin mesh}
             \label{fig:estim:noeps}
  \end{figure}
  shows the behaviour of these modified indicators.
  Obviously, there is no dependence on $\eps$ any longer and the errors
  are caught quite well.

   \subsubsection*{Convergence in $N$ on a-priori adapted meshes}
   For our second experiment we let $\eps=10^{-6}$ be constant,
   chose a-priori adapted meshes, apply the
   modified upwind method and estimate the error with $\eta$ and $\widetilde\eta$.
   Figures~\ref{fig:estim:Shishkin}
  \begin{figure}[tbp]
      \centerline{\includegraphics[width=0.75\textwidth]{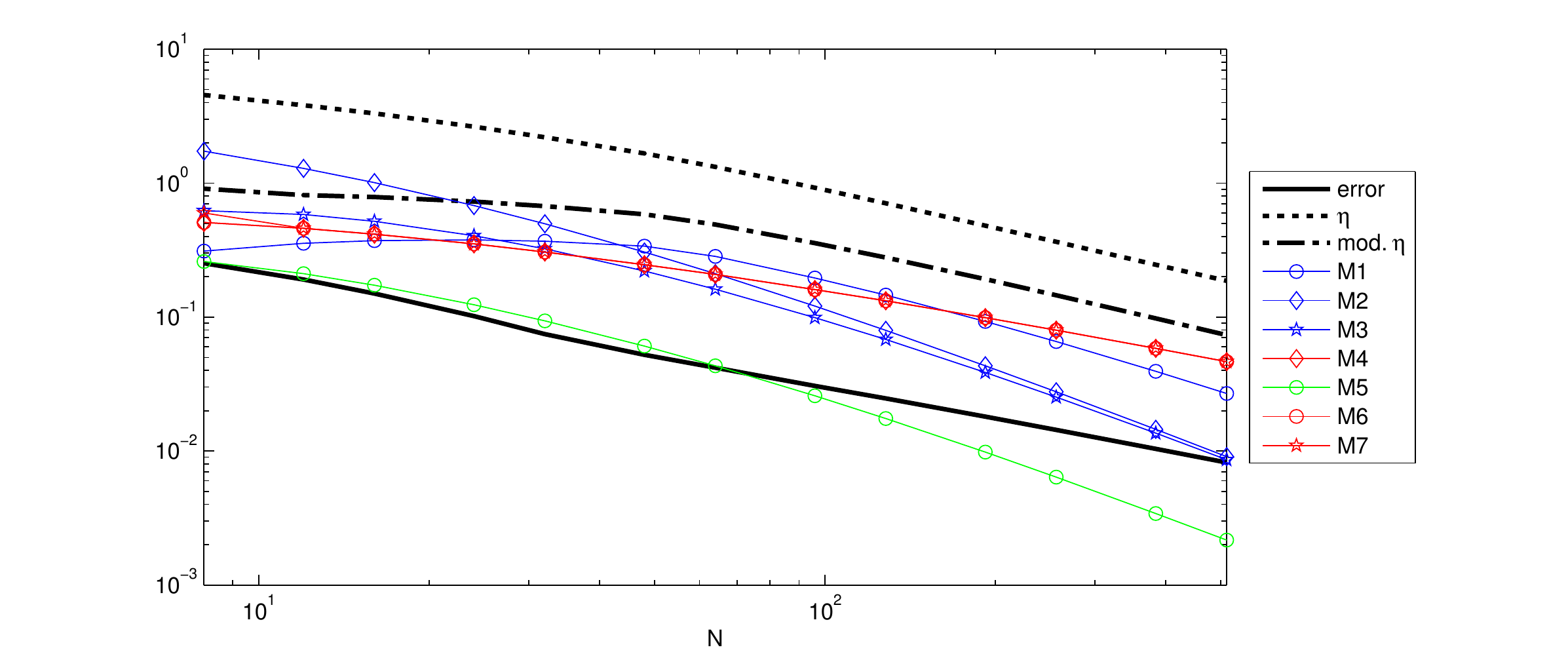}}
      \caption{Error and estimated error of \eqref{eq:num_example2} for $\eps=10^{-6}$ on Shishkin meshes}
             \label{fig:estim:Shishkin}
  \end{figure}
  and \ref{fig:estim:BSmesh}
  \begin{figure}[tbp]
      \centerline{\includegraphics[width=0.75\textwidth]{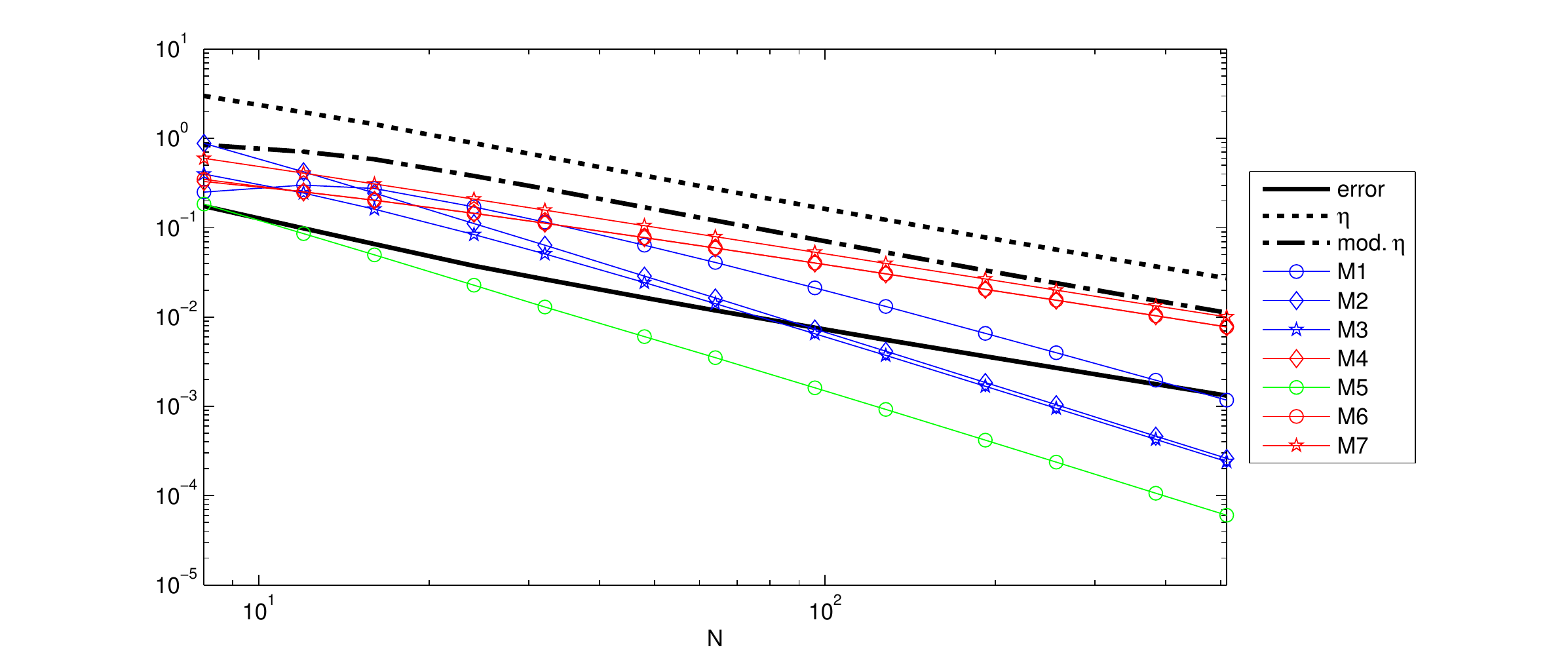}}
      \caption{Error and estimated error of \eqref{eq:num_example2} for $\eps=10^{-6}$ on a Bakhvalov S-meshes}
             \label{fig:estim:BSmesh}
  \end{figure}
  show the results for the two indicators and variable $N$. The principal behaviour of the errors is caught
  by both of them although the magnitude is wrong. We also observe the
  blue lines to fall much faster than the red lines. The reason behind is the formal
  second order convergence in $y$-direction of $M^1$ to $M^3$. This gives hope for
  a-posteriori mesh adaptation to behave better than a-priori adapted meshes.

  \subsubsection*{A-posteriori adapted meshes}
  \begin{figure}[tb]
      \centerline{\includegraphics[width=0.8\textwidth]{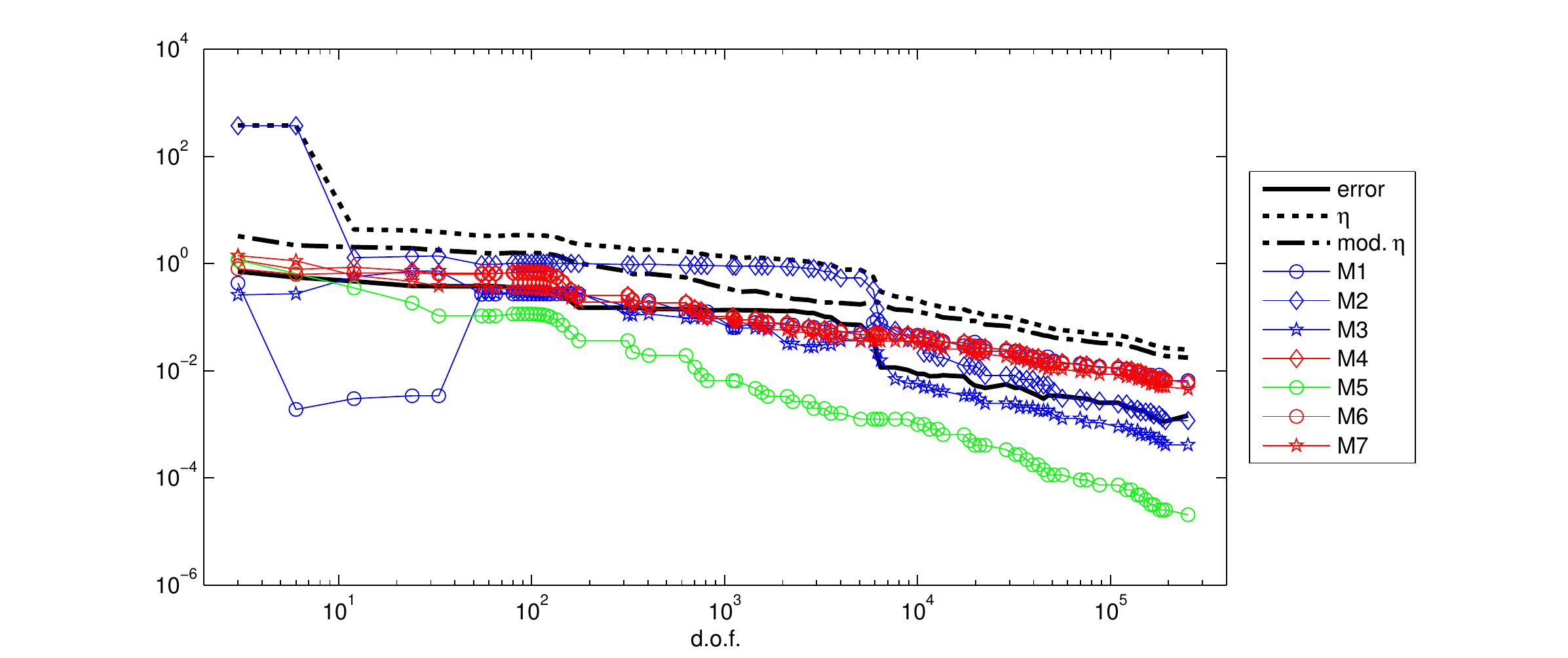}}
      \caption{Error and estimated error of \eqref{eq:num_example2} for $\eps=10^{-6}$ on an adapted mesh with initial Shishkin mesh}
             \label{fig:estim:adapt:Shishkin}
  \end{figure}
  \begin{figure}[tb]
      \centerline{\includegraphics[width=0.8\textwidth]{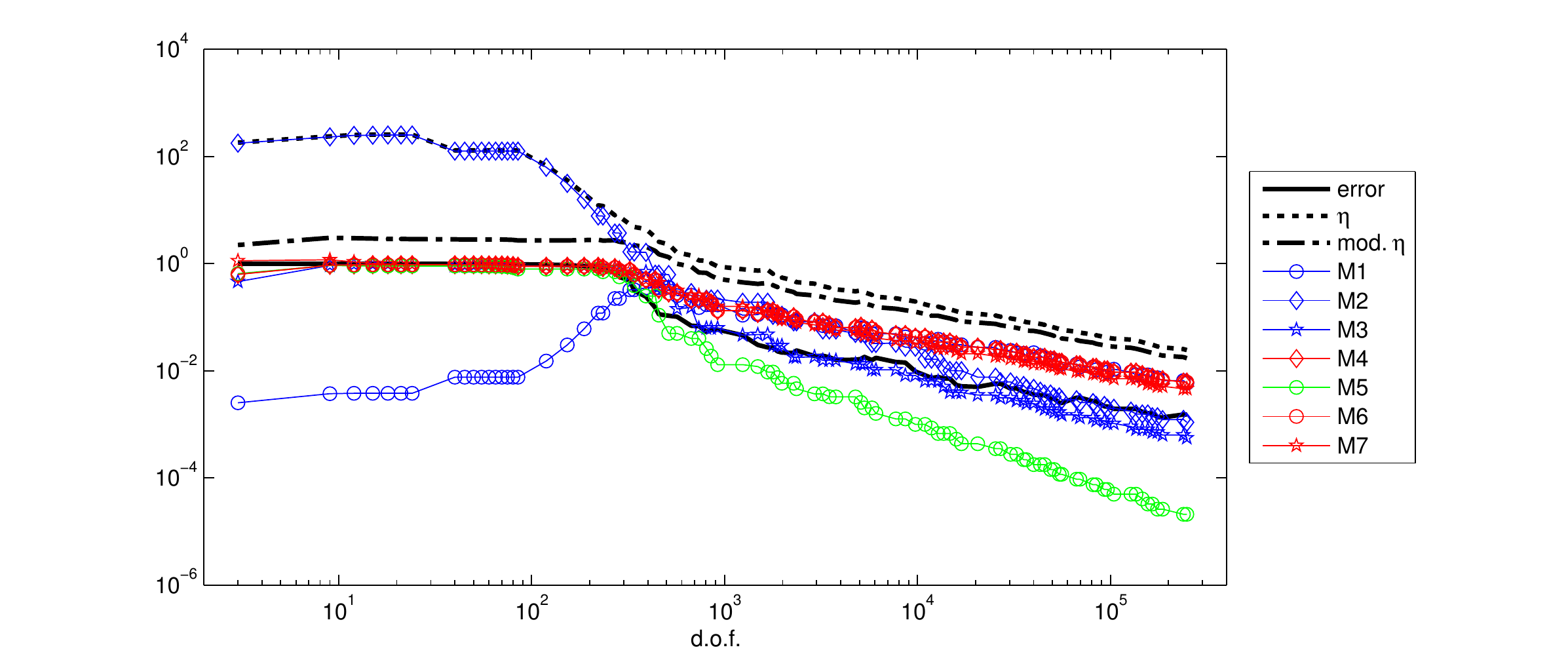}}
      \caption{Error and estimated error of \eqref{eq:num_example2} for $\eps=10^{-6}$ on an adapted mesh with initial equidistant mesh}
             \label{fig:estim:adapt:eq}
  \end{figure}
  Let us consider the following simple, anisotropic mesh adaptation approach.
  We start with a coarse initial, tensor-product mesh.
  In each step we compute the numerical solution on the given mesh
  and use the error indicators to decide, whether and where the $x$-part or
  the $y$-part of the tensor product mesh should be refined.
  This will be done as follows:
  \begin{enumerate}
   \item Compute $M_y:=\max\limits_{k=1,3}\{\max\{M_{ij}^k\}\}$ and
         $M_x:=\max\limits_{k=4,5,6,7}\{\max\{M_{ij}^k\}\}$.
   \item If $M_x>M_y$ we refine in $x$-direction, otherwise in $y$-direction.
         Assuming $M_x>M_y$ we collect all $i$ with $M_{ij}^k\geq\alpha\max\{M_{ij}^k\}$ for any $k=4,5,6,7$ and given $\alpha\in[0,1]$,
         and divide $[x_{i-1},x_i]$ into two intervals of equal length. 
         Similarly, we proceed in the other case and divide $[y_{j-1},y_j]$ 
         into two intervals for all $j$ with $M_{ij}^k\geq\alpha\max\{M_{ij}^k\}$ 
         for any $k=1,2,3$.
  \end{enumerate}
  With these refined partitions we construct a new tensor product mesh
  and the cycle begins again. This refinement process has a parameter $\alpha$
  influencing the marking of elements to refine.
  We chose $\alpha=0.9$ to refine only elements with large contributions.

  In Figure~\ref{fig:estim:adapt:Shishkin} and \ref{fig:estim:adapt:eq}
  the convergence results for $\eps=10^{-6}$ are shown until the number 
  of degrees of freedom reaches approximately $512^2$. 
  In Figure~\ref{fig:estim:adapt:Shishkin} initially a Shishkin mesh 
  of 4-by-4 cells was taken and in the end we have $1014\times 250$ cells.
  In Figure~\ref{fig:estim:adapt:eq} the initial mesh was equidistant with 4-by-4 cells, 
  and the final mesh has $992\times 252$ cells.

  We observe in both cases that our adaptation process reduces the error 
  nicely. The observed overall order of convergence (after some initial phase) 
  is $(\#dof)^{-1/2}$ where $\#dof$ is the number of degrees of freedom.

  Comparing the errors for the number of degrees of freedom taken to be 
  about $512^2$, Table~\ref{tab:estim:compare}
  \begin{table}[tbp]
   \begin{center}
    \caption{Comparison of errors of a-priori and a-posteriori meshes
             \label{tab:estim:compare}}
    \begin{tabular}{l|l|l}
       mesh & $\#dof$ & $\norm{\mathbf{u}^\B-u}{L_\infty(\Omega)}$\\
       \hline
       Shishkin-mesh                              &262144& 8.2024e-03\\
       Bakhvalov S-mesh                           &262144& 1.3226e-03\\
       adapted mesh with initial S-mesh           &253500& 1.4570e-03\\
       adapted mesh with initial equidistant mesh &249984& 1.5117e-03
    \end{tabular}
   \end{center}
  \end{table}
  shows the results on the a-posteriori adapted meshes to be comparable 
  to the a-priori adapted meshes. With the different number of cells
  in each direction the a-posteriori adapted meshes can reduce the error
  much better than a Shishkin mesh. Still, the grading of the Bakhvalov S-mesh
  gives a mesh with the smallest error.
  Moreover, the costs for an adaptive algorithm are high
  due to the repeated solving of the numerical problems on the different
  meshes.

%% file: outlook.tex
\chapter{Conclusion and Outlook}

We have presented convergence and supercloseness results for higher-order
finite-element methods, including stabilised methods like LPSFEM and SDFEM.
Having general polynomial spaces $\QS^\clubsuit_p$, convergence of order $p$
can be proved. If we use proper subspaces of $\QS_p$, like the Serendipity
space, we cannot apply the supercloseness techniques that are valid for the
full space $\QS_p$. But numerical results do also indicate, that for proper
subspaces no supercloseness property holds.

%
While numerical simulations indicate supercloseness properties of order $p+1$
for many methods, numerical analysis provides proof only for order $p+1/2$ in the
case of SDFEM.
Further research is needed to improve this situation. Some preliminary
results for the pure Galerkin method are topic of ongoing research. Here
a supercloseness property in the case of exponential boundary layers and odd polynomial degree $p$
of order $p+1/4$ could be proved, \cite{FrR13}. The proof therein can easily
be adapted to the case of characteristic boundary layers too. Nevertheless,
there is still a gap between theory and simulation of $3/4$ orders.

Although convergence and supercloseness can be proved in the energy and related
norms, these norms do not ``see'' the characteristic layers correctly. The layers are
under-represented in the resulting terms. An alternative is shown in the
balanced norm that has the right weighting of the norm components. But now
the Galerkin FEM is no longer coercive w.r.t. this norm.
Using certain stability arguments, for a modified bilinear SDFEM convergence in this norm
is proved. How the proof can be modified for the standard Galerkin FEM and other
stabilised methods, and for higher order methods in general is an open question.
Numerically, all these methods show the same convergence and supercloseness behaviour
in the energy and the balanced norm.

The use of a-priori adapted meshes requires knowledge about the layer-structure of
solutions to the considered problem. Alternatively, the mesh can be adapted \emph{after}
computation of an (approximative) numerical solution. For this a-posteriori mesh
adaptation, uniform error estimators or indicators are needed. We presented estimates
on the $L_1$-norm of the Green's function as an ingredient for $L_\infty$-error estimators.
A simple, first estimator for a finite difference method is also given and analysed.
The optimisation of this estimator and an extension of this approach to finite element
methods are open problems.

%% file: bibliography.tex
  \bibliographystyle{habil_sort}
  \bibliography{lit}

%% file: appendix_SLUB.tex
\makeatletter
\newcommand\appendix@section[1]{%
  \refstepcounter{section}%
  \addcontentsline{toc}{section}{\textbf{\thesection}\hspace*{10pt}\textbf{\appauthor}: \apptitle{} \apppub}%
  \label{#1}
}
\let\orig@section\section
\g@addto@macro\appendix{\let\section\appendix@section}
\makeatother
%
\appendix
\chapter*{Appendix}
\addstarredchapter{Appendix}
\refstepcounter{chapter}
%
%
\minitoc
\fancyhead[C]{\nouppercase Appendix \\}
\clearpage
%
%
\fancyhead[C]{\nouppercase Appendix \thesection\\}
\includepdfset{pages={1-1}, pagecommand={\thispagestyle{fancy}}}

%
%
%
\renewcommand{\appauthor}{S. Franz, G. Matthies}
\renewcommand{\apptitle}{Local projection stabilisation on S-type meshes for convection-diffusion problems with characteristic layers.}
\renewcommand{\apppub}{Computing, 87(3-4), 135--167, 2010}
\section{app:LPS}

\textbf{\appauthor}: \apptitle{} \apppub\\[1cm]
\begin{center}
 [The article is removed from this electronic version due to copyright reasons.]\\[1cm]
\end{center}

\textbf{Abstract:}\\

    Singularly perturbed convection-diffusion problems with exponential and
    characteristic layers are considered on the unit square. The discretisation is
    based on layer-adapted meshes. The standard Galerkin method and the local
    projection scheme are analysed for bilinear and higher order finite element
    where enriched spaces were used. For bilinears, first order convergence in the
    $\eps$-weighted energy norm is shown for both the Galerkin and the stabilised
    scheme. However, supercloseness results of second orders hold for the Galerkin
    method in the $\eps$-weighted energy norm and for the local projection scheme
    in the corresponding norm. For the enriched $\mathcal{Q}_p$-elements, $p\ge 2$,
    which already contain the space $\mathcal{P}_{p+1}$, a convergence order $p+1$ in the
    $\eps$-weighted energy norm is proved for both the Galerkin method and the
    local projection scheme. Furthermore, the local projection methods provides a
    supercloseness result of order $p+1$ in local projection norm.\\[1em]

\textbf{Keywords:}
 Singular perturbation,
 Characteristic layers,
 Shishkin meshes,
 Local projection\\

\textbf{Mathematics Subject Classification (2000):} 65N12, 65N30, 65N50\\

\textbf{DOI:} 10.1007/s00607-010-0079-y

\clearpage
%
%
\renewcommand{\appauthor}{S. Franz, G. Matthies}
\renewcommand{\apptitle}{Convergence on Layer-adapted Meshes and Anisotropic Interpolation Error Estimates of Non-Standard Higher Order Finite Elements.}
\renewcommand{\apppub}{Appl. Numer. Math., 61, 723--737, 2011}
\section{app:nonstandard}

\textbf{\appauthor}: \apptitle{} \apppub\\[1cm]
\begin{center}
 [The article is removed from this electronic version due to copyright reasons.]\\[1cm]
\end{center}

\textbf{Abstract:}\\

    For a general class of finite element spaces based on local polynomial
    spaces $\mathcal{E}$ with \mbox{$\mathcal{P}_p\subset\mathcal{E}\subset\mathcal{Q}_p$}
    we construct a vertex-edge-cell and point-value oriented interpolation operators
    that fulfil anisotropic interpolation error estimates.

    Using these estimates we prove $\eps$-uniform convergence of order $p$ for the
    Galerkin FEM and the LPSFEM for a singularly perturbed
    convection-diffusion problem with characteristic boundary layers.\\[1em]

\textbf{Keywords:}
singular perturbation,
characteristic layers,
exponential layers,
Shishkin meshes,
local-projection,
higher-order FEM\\

\textbf{Mathematics Subject Classification (2000):} 65N12, 65N30, 65N50\\

\textbf{DOI:} 10.1016/j.apnum.2011.02.001
\clearpage
%
%
\renewcommand{\appauthor}{S. Franz}
\renewcommand{\apptitle}{Superconvergence Using Pointwise Interpolation in Convection-Diffusion Problems.}
\renewcommand{\apppub}{Appl. Numer. Math., 76, 132--144, 2014}
\section{app:GLinter}

\textbf{\appauthor}: \apptitle{} \apppub\\[1cm]
\begin{center}
 [The article is removed from this electronic version due to copyright reasons.]\\[1cm]
\end{center}

\textbf{Abstract:}\\

     Considering a singularly perturbed convection-diffusion problem,
     we present an analysis for a superconvergence result using pointwise
     interpolation of Gau\ss-Lobatto type for higher-order
     streamline diffusion FEM.
     We show a useful connection between two different types of interpolation,
     namely a vertex-edge-cell interpolant and a pointwise interpolant.
     Moreover, different postprocessing operators are analysed and applied
     to model problems.\\[1em]

\textbf{Keywords:}
singular perturbation,
layer-adapted meshes,
superconvergence,
postprocessing \\

\textbf{Mathematics Subject Classification (2000):} 65N12, 65N30, 65N50\\

\textbf{DOI:} 10.1016/j.apnum.2013.07.007
\clearpage
%
%
\renewcommand{\appauthor}{S. Franz}
\renewcommand{\apptitle}{SDFEM with non-standard higher-order finite elements for a convec\-tion-diffusion problem with characteristic boundary layers.}
\renewcommand{\apppub}{BIT Numerical Mathematics, 51(3), 631--651, 2011}
\section{app:SDFEM}

\textbf{\appauthor}: \apptitle{} \apppub\\[1cm]
\begin{center}
 [The article is removed from this electronic version due to copyright reasons.]\\[1cm]
\end{center}

\textbf{Abstract:}\\

    Considering a singularly perturbed problem with exponential and
    characteristic layers, we show convergence for non-standard
    higher-order finite elements using the streamline diffusion
    finite element method (SDFEM). Moreover, for the standard
    higher-order space $\QS_p$ supercloseness of the
    numerical solution w.r.t. an interpolation of the exact solution
    in the streamline diffusion norm of order $p+1/2$ is proved.
\\[1em]

\textbf{Keywords:}
singular perturbation,
characteristic layers,
exponential layers,
Shishkin mesh,
SDFEM,
higher order\\
            
\textbf{Mathematics Subject Classification (2000):} 65N12, 65N30, 65N50\\

\textbf{DOI:} 10.1007/s10543-010-0307-z
\clearpage
%
%
\renewcommand{\appauthor}{S. Franz}
\renewcommand{\apptitle}{Convergence Phenomena of $Q_p$-Elements for Convection-Diffusion Problems.}
\renewcommand{\apppub}{Numer. Methods Partial Differential Equations, 29(1), 280--296, 2013}
\section{app:phenomena}

\textbf{\appauthor}: \apptitle{} \apppub\\[1cm]
\begin{center}
 [The article is removed from this electronic version due to copyright reasons.]\\[1cm]
\end{center}

\textbf{Abstract:}\\

    We present a numerical study for singularly perturbed
    convection-diffusion problems using higher-order
    Galerkin and Streamline Diffusion FEM.
    We are especially interested in convergence and superconvergence
    properties with respect to different interpolation operators.
    For this we investigate pointwise interpolation and
    vertex-edge-cell interpolation.
\\[1em]

\textbf{Keywords:}
singular perturbation,
boundary layers,
layer-adapted meshes,
superconvergence
\\

\textbf{Mathematics Subject Classification (2000):} 65N12, 65N30, 65N50\\

\textbf{DOI:} 10.1002/num.21709
\clearpage
%
%
\renewcommand{\appauthor}{S. Franz, H.-G. Roos}
\renewcommand{\apptitle}{Error estimation in a balanced norm for a convection-diffusion problem with characteristic boundary layers.}
\renewcommand{\apppub}{Calcolo,\newline DOI:10.1007/s10092-013-0093-5, 2013}
\section{app:balanced}
\renewcommand{\apppub}{Calcolo, DOI:10.1007/s10092-013-0093-5, 2013}

\textbf{\appauthor}: \apptitle{} \apppub\\[1cm]
\begin{center}
 [The article is removed from this electronic version due to copyright reasons.]\\[1cm]
\end{center}

\textbf{Abstract:}\\

       The $\eps$-weighted energy norm is the natural norm for
       singularly perturbed convection-diffusion problems with exponential 
       layers.
       But, this norm is too weak to recognise features of
       characteristic layers.

       We present an error analysis in a differently weighted energy norm---a balanced
       norm---that overcomes this drawback.
\\[1em]

\textbf{Keywords:}
singular perturbation,
characteristic and exponential layers,
Shishkin mesh,
SDFEM, balanced norm\\

\textbf{Mathematics Subject Classification (2000):} 65N12, 65N30, 65N50\\

\textbf{DOI:} 10.1007/s10092-013-0093-5
\clearpage
%
%
\renewcommand{\appauthor}{S. Franz, N. Kopteva}
\renewcommand{\apptitle}{Green's function estimates for a singularly perturbed convection-diffusion problem.}
\renewcommand{\apppub}{Journal of Differential Equations, 252, 1521--1545, 2012}
\section{app:green}

\textbf{\appauthor}: \apptitle{} \apppub\\[1cm]
\begin{center}
 [The article is removed from this electronic version due to copyright reasons.]\\[1cm]
\end{center}

\textbf{Abstract:}\\

    We consider a singularly perturbed convection-diffusion
    problem posed in the unit square with a horizontal convective direction.
    Its solutions exhibit parabolic and exponential boundary layers.
    Sharp estimates of the Green's function and its
    first- and second-order derivatives are derived in the $L_1$ norm.
    The dependence of these estimates on the small diffusion parameter
    is shown explicitly.
    The obtained estimates will be used in a forthcoming numerical analysis
    of the considered problem.
\\[1em]

\textbf{Keywords:}
Green's function,
singular perturbations,
convection-diffusion
\\

\textbf{Mathematics Subject Classification (2000):} 35J08, 35J25, 65N15\\

\textbf{DOI:} 10.1016/j.jde.2011.07.033
\clearpage
%
%
\renewcommand{\appauthor}{S. Franz, N. Kopteva}
\renewcommand{\apptitle}{On the Sharpness of Green's function estimates for a convection-diffusion problem.}
\renewcommand{\apppub}{In Proceedings of the Fifth Conference on Finite Difference Methods: Theory and Applications (FDM: T\&A 2010), 44--57, Rousse University Press, 2011}
\section{app:green_sharp}

\textbf{\appauthor}: \apptitle{} \apppub\\[1cm]
\begin{center}
 [The article is removed from this electronic version due to copyright reasons.]\\[1cm]
\end{center}

\textbf{Abstract:}\\

   Linear singularly perturbed convection-diffusion
   problems with characteristic layers are considered in three dimensions.
   We demonstrate the sharpness of our recently obtained upper bounds for the associated Green's
   function and its derivatives in the $L_1$ norm.
   For this, in this paper we establish
   the corresponding lower bounds.
   Both upper and lower bounds explicitly show any dependence
   on the singular perturbation parameter.
\\[1em]

\textbf{Keywords:}
Green's function,
singular perturbations,
convection-diffusion,
a posteriori error estimates\\

\textbf{Mathematics Subject Classification (2000):} 35J08, 35J25, 65N15\\

\textbf{ISBN:} 978-954-8467-44-5
\clearpage

%% file: affirmation.tex
\thispagestyle{empty}

\begin{center}
 \textbf{Erklärung}
\end{center}

   Hiermit versichere ich, dass ich die vorliegene Arbeit ohne
   unzul\"assige Hilfe Dritter und ohne Benutzung anderer als der
   angegebenen Hilsmittel angefertigt habe. Die aus fremden Quellen
   direkt oder indirekt \"ubernommenen Gedanken sind als solche
   kenntlich gemacht. Die Arbeit wurde bisher weder im In- noch im
   Ausland in gleicher oder \"ahnlicher Form einer anderen
   Pr\"ufungsbeh\"orde vorgelegt.

\vfill

Dresden, den 31.05.2013.